\documentclass{amsart}
\usepackage[T1]{fontenc}

\usepackage{amssymb}
\usepackage{amsmath}
\usepackage{amsthm}
\usepackage{pgfplots} 
\pgfplotsset{compat=1.9}
\usepackage{graphicx}
\usepackage[utf8]{inputenc}
\usepackage{tikz}
\usepackage{tikz-cd}
\usepackage{bbm}
\usepackage{subcaption}
\usepackage[boxruled]{algorithm2e}
\usepackage{mathtools}
\usepackage{lipsum}
\usepackage[title,titletoc]{appendix}
\usepackage{booktabs}
\usepackage{here}
\usepackage{natbib,enumerate} 
\usepackage{hyperref}
\usepackage{faktor}
\usepackage{comment}
\usepackage{bm}
\usepackage{adjustbox}
\usepackage{circuitikz}
\usepackage{moresize}
\usepackage{pdflscape}
\usepackage{mathpazo}
\usepackage{euler}
\usepackage{faktor}
\usepackage[a4paper,left=3cm,right=3cm,top=3cm,bottom=3cm]{geometry}
\usepackage{microtype}

\renewcommand{\phi}{\varphi}

\theoremstyle{definition}
\newtheorem{theorem}{Theorem}
\numberwithin{theorem}{section}
\newtheorem{proposition}[theorem]{Proposition}
\newtheorem{lemma}[theorem]{Lemma}

\newtheorem{remark}[theorem]{Remark}
\newtheorem{definition}[theorem]{Definition}

\newtheorem{claim}[theorem]{Claim}

\newcommand{\Hom}{\text{Hom}}

\DeclareMathOperator{\Hilb}{\operatorname{Hilb}}

\newcommand\define{\stackrel{\mathclap{\mbox{\tiny def}}}{=}}

\RequirePackage{tikz-cd}
\RequirePackage{amssymb}
\usetikzlibrary{calc}
\usetikzlibrary{decorations.pathmorphing}

\tikzset{curve/.style={settings={#1},to path={(\tikztostart)
    .. controls ($(\tikztostart)!\pv{pos}!(\tikztotarget)!\pv{height}!270:(\tikztotarget)$)
    and ($(\tikztostart)!1-\pv{pos}!(\tikztotarget)!\pv{height}!270:(\tikztotarget)$)
    .. (\tikztotarget)\tikztonodes}},
    settings/.code={\tikzset{quiver/.cd,#1}
        \def\pv##1{\pgfkeysvalueof{/tikz/quiver/##1}}},
    quiver/.cd,pos/.initial=0.35,height/.initial=0}

\tikzset{tail reversed/.code={\pgfsetarrowsstart{tikzcd to}}}
\tikzset{2tail/.code={\pgfsetarrowsstart{Implies[reversed]}}}
\tikzset{2tail reversed/.code={\pgfsetarrowsstart{Implies}}}
\tikzset{no body/.style={/tikz/dash pattern=on 0 off 1mm}}

\usepackage{xcolor}
\hypersetup{
    colorlinks,
    linkcolor={red!50!black},
    citecolor={blue!50!black},
    urlcolor={blue!80!black}
}

\begin{document}

\newsavebox{\yangianbox}
\sbox{\yangianbox}{$Y_{t_1,t_2}(\widehat{\mathfrak{gl}}_1)$}

\title[An action on the cohomology of the moduli space of stable sheaves on surfaces]{An action of the affine Yangian \texorpdfstring{\usebox{\yangianbox}}{Y_{t_1,t_2}(gl_1)} on the cohomology of the moduli space of stable sheaves on surfaces}

\author[Claudio Pfammatter]{Claudio Pfammatter}

%\date{\today}

\address{Master's program in Mathematics, Chair of Representation Theory, EPFL, Switzerland}

\email{claudio.pfammatter@epfl.ch}

\begin{abstract}
We investigate the action of the affine Yangian $Y_{t_1,t_2}(\widehat{\mathfrak{gl}}_1)$ on the singular cohomology of the moduli space of stable sheaves $\mathcal{M}$ on a smooth projective surface $S$. Through an intersection-theoretic analysis of nested moduli spaces, we obtain a proof of the representation $$Y_{t_1,t_2}(\widehat{\mathfrak{gl}}_1) \curvearrowright H_{\mathcal{M}}.$$ This intersection-theoretic approach extends classical constructions to the full Yangian structure and circumvents the $K$-theoretic machinery traditionally required to establish this action.
\end{abstract}

\maketitle

%\vspace{-2em}
%\tableofcontents

\section{Introduction}
\label{section_introduction}

\subsection{Motivation}
Let $S$ be a smooth projective surface over $\mathbb{C}$. In \cite{goettsche_betti_numbers}, Göttsche computed the Poincaré polynomials of $S^{[n]}$, the \textit{Hilbert scheme of $n$ points on $S$}. Vafa and Witten \cite{vafa_coupling_test} subsequently observed that as $n$ ranges from $0$ to $\infty$, the generating function of these Poincaré polynomials coincides with the character of the Fock space representation of the infinite-dimensional Heisenberg algebra $\widehat{\mathfrak{gl}}_1$. Motivated by this insight, Nakajima proposed studying all such Hilbert schemes simultaneously, packaging them into the \textit{Hilbert scheme of points on $S$} defined by
\begin{equation}
    \label{equation_definition_Hilbert_scheme}
   \operatorname{Hilb}(S) \define \coprod_{n=0}^\infty S^{[n]}. 
\end{equation}

This observation catalyzed a central theme in geometric representation theory: the realization of infinite-dimensional algebras on moduli spaces. Independently, Grojnowski \cite{grojnowksi_heisenberg_action} and Nakajima \cite{nakajima_heisenberg_action} constructed an action of $\widehat{\mathfrak{gl}}_1$ on the singular cohomology $H_{\operatorname{Hilb}(S)}$ \footnote{We employ this non-standard notation for singular cohomology with integer coefficients to emphasize that our results hold for other (co)homology theories as well, such as Borel–Moore
homology or Chow rings (see Remark \ref{remark_Chow}).} using correspondences. We follow the formulation of \textit{loc.\ cit.} The foundational objects in this construction are the \textit{nested Hilbert schemes}
\[
\mathfrak{C}_i^\bullet \define \Big\{ \mathcal{J}' \subset_{i u} \mathcal{J} \Big\} \subset \operatorname{Hilb}(S) \times S \times \operatorname{Hilb}(S),
\]
where the notation $\mathcal{J}' \subset_{i u} \mathcal{J}$ indicates that $\mathcal{J}'$ is a subsheaf of the ideal sheaf $\mathcal{J} \subset \mathcal{O}_S$, and the quotient $\mathcal{J}/\mathcal{J}'$ is a length $i$ skyscraper sheaf supported at an arbitrary closed point $u \in S$. 

To define the action, one considers the natural projection maps from this correspondence space given by the diagram
\begin{equation*}
    \begin{tikzcd}[row sep=small]
	& {\mathfrak{C}_i^\bullet} &&&& {\mathcal{J}' \subset_{i u} \mathcal{J}} \\
	{\operatorname{Hilb}(S)} & S & {\operatorname{Hilb}(S)} && {\mathcal{J}'} & u & {\mathcal{J}.}
	\arrow["{p_+}"', from=1-2, to=2-1]
	\arrow["{p_S}", from=1-2, to=2-2]
	\arrow["{p_-}", from=1-2, to=2-3]
	\arrow[maps to, from=1-6, to=2-5]
	\arrow[maps to, from=1-6, to=2-6]
	\arrow[maps to, from=1-6, to=2-7]
\end{tikzcd}
\end{equation*}
We construct annihilation and creation operators for $i \in \mathbb{Z}_{>0}$ via push-pull formulas:
\begin{equation*}
    \mathcal{P}[\pm i] \define (p_\mp \times p_S)_* p_\pm^*: H_{\operatorname{Hilb}(S)} \rightarrow H_{\operatorname{Hilb}(S) \times S}.
\end{equation*}
By proving that these operators satisfy the defining commutation relations of the Heisenberg algebra, Nakajima turned $H_{\operatorname{Hilb}(S)}$ into the Fock space of $\widehat{\mathfrak{gl}}_1$, which explains Göttsche's computation in a representation-theoretic way. This action was subsequently used to understand, among other things, the Gromov--Witten theory of $\operatorname{Hilb}(S)$ in \cite{oberdieck_gromov-witten} and its Chow ring in \cite{maulik_chow}.

As the field progressed, the focus naturally expanded along three distinct axes: the algebraic structures acting on these spaces, the moduli spaces themselves, and the (co)homology theories employed. On the algebraic side, attention shifted from the Heisenberg algebra to larger quantum groups, for instance, the quantum toroidal $\mathfrak{gl}_1$. Geometrically, it is natural to replace the rank $1$ ideal sheaves of the Hilbert scheme with higher-rank stable sheaves. Fix an ample divisor $H$ on $S$, a rank $r \in \mathbb{Z}_{>0}$, and a first Chern class $c_1 \in H^2(S, \mathbb{Z})$. Under the standard assumptions 
\begin{align}
    \label{equation_assumption_A}
    &\textbf{Assumption A:} \qquad \operatorname{gcd}(r,c_1 \cdot H) =1 \\
    \label{equation_assumption_S}
    &\textbf{Assumption S:} \qquad \,\text{ either } \omega_S \cong \mathcal{O}_S \text{ or } c_1(\omega_S) \cdot H < 0.
\end{align}
the \textit{moduli space of stable sheaves} 
\[
\mathcal{M} \define \coprod_{c_2 = -\infty}^\infty \mathcal{M}_{(r,c_1,c_2)}
\]
is a smooth projective scheme that provides a natural generalization of $\operatorname{Hilb}(S)$. In this higher-rank setting, Baranovsky successfully lifted the Grojnowski--Nakajima construction, establishing a Heisenberg action on $H_{\mathcal{M}}$ in \cite{baranovsky_moduli_stable_sheaves}.

Moving along the third axis, it becomes necessary to replace singular cohomology with $K$-theory because the quantum toroidal $\mathfrak{gl}_1$ naturally acts on $K_\mathcal{M}$. This algebra, also called the Ding--Iohara--Miki algebra, is denoted by $U_{q_1,q_2}(\ddot{\mathfrak{gl}}_1)$ and was studied in \cite{feigin_commutative_algebra} and \cite{feigin_quantum_toroidal}. However, establishing an action in $K$-theory proved significantly more challenging. The primary obstacle to this generalization is geometric: the classical higher-length analogues of the nested correspondences $\mathfrak{C}_i^\bullet$ for stable sheaves are poorly behaved for $i > 1$ (it is not even known whether they are equidimensional).

To bypass this issue, we define operators using certain $K$-theory classes on the correspondence 
\[
\mathfrak{Z}_i^\bullet \define \Big\{ \mathcal{F}_0 \subset_u \mathcal{F}_1 \subset_u \cdots \subset_u \mathcal{F}_i \Big\},
\]
where $\mathcal{F}_0, \ldots, \mathcal{F}_i$ are stable sheaves on the surface $S$, and $\mathcal{F}_{k-1} \subset_u \mathcal{F}_k$ means that $\mathcal{F}_{k-1} \subset \mathcal{F}_k$ and $\mathcal{F}_k/\mathcal{F}_{k-1} \cong \mathbb{C}_u$ for $u \in S$. The nested moduli space $\mathfrak{Z}_i^\bullet$ can be thought of as a resolution of $\mathfrak{C}_i^\bullet$. Utilizing the shuffle algebra, Neguţ constructed an action of the quantum toroidal $\mathfrak{gl}_1$ on $K_\mathcal{M}$ in \cite{negut_shuffle_surfaces}. This program culminated in derived-categorical operators whose $K$-theoretic shadows realize the elliptic Hall algebra in \cite{negut_hecke_correspondences}. 

By combining \textit{loc.\ cit.} with \cite{zhao_commutator_Hilbert_schemes}, one obtains a geometric $U_{q_1,q_2}(\ddot{\mathfrak{gl}}_1)$ action on the derived category of $\operatorname{Hilb}(S)$. To establish this, Zhao studied certain nested and quadruple Hilbert schemes through the lens of the minimal model program, showing that these spaces possess either canonical or semidivisorial log terminal singularities. He finally realized a weak categorification of Neguţ's action on $K_\mathcal{M}$ via two intersection-theoretic descriptions of the quadruple moduli space of stable sheaves in \cite{zhao_commutator_moduli_stable}.

Despite the success of these frameworks, they rely heavily on sophisticated $K$-theoretic and derived-categorical machinery. An intersection-theoretic construction, however, remains a central objective in the field (see Theorem 1.2 of \cite{negut_rationaltrigonometricellipticalgebras}) and would allow for an action on singular cohomology without relying on Chern character descents from $K$-theory or weak categorifications.

Replacing $K$-theory with singular cohomology on the geometric side dictates a corresponding shift on the algebraic side. As developed by Tsymbaliuk \cite{tsymbaliuk_yangian}, the affine Yangian $Y_{t_1,t_2}(\widehat{\mathfrak{gl}}_1)$ (see Definition \ref{definition_Yangian} for a description of this algebra) arises as the natural ``additivization'' of $U_{q_1,q_2}(\ddot{\mathfrak{gl}}_1)$, mirroring the classical setting of simple Lie algebras $\mathfrak{g}$, where the Yangian $Y_h(\mathfrak{g})$ acts as the ``additivization'' of the quantum loop group $U_q(L\mathfrak{g})$. %Furthermore, Arbesfeld and Schiffmann established in \cite{schiffmann_Yangian} that $Y_{t_1,t_2}(\widehat{\mathfrak{gl}}_1)$ is isomorphic to the algebra $\mathbf{SH}^{\mathbf{c}}$ introduced in \cite{schiffmann_Cherednik}. 
Whereas the quantum toroidal $\mathfrak{gl}_1$ acts on the $K$-theory of the moduli space $\mathcal{M}$, its ``additivization'' provides the algebraic framework required to act on singular cohomology.

The goal of the present paper is to provide such a construction. Specifically, we establish the action of the affine Yangian $Y_{t_1,t_2}(\widehat{\mathfrak{gl}}_1)$ on $H_{\mathcal{M}}$ via intersection theory.

\subsection{Main results}
\label{section_main_results}
We utilize the nested moduli space of stable sheaves $\mathfrak{Z}_{(u)}$ (see Definition \ref{definition_nested_stable_space} for details). The natural projection maps from this correspondence space are given by the diagram
\[
\begin{tikzcd}[row sep=small]
	& {\mathfrak{Z}_{(u)}} &&&& {\mathcal{F}' \subset_{u} \mathcal{F}} \\
	{\mathcal{M}} & S & {\mathcal{M}} && {\mathcal{F}'} & u & {\mathcal{F},}
	\arrow["{p_+}"', from=1-2, to=2-1]
	\arrow["{p_S}", from=1-2, to=2-2]
	\arrow["{p_-}", from=1-2, to=2-3]
	\arrow[maps to, from=1-6, to=2-5]
	\arrow[maps to, from=1-6, to=2-6]
	\arrow[maps to, from=1-6, to=2-7]
\end{tikzcd}
\]
where $\mathcal{F}'$ and $\mathcal{F}$ are stable sheaves on $S$. Let $\mathcal{L}$ denote the tautological line bundle on $\mathfrak{Z}_{(u)}$, defined such that its fiber over a point $(\mathcal{F}' \subset_{u} \mathcal{F})$ is $\Gamma(S, \mathcal{F}/\mathcal{F}')$, and set $\ell = c_1(\mathcal{L})$. We construct a family of operators via push-pulls:
\begin{align}
\label{equation_operator_E}
    E_d &\define (p_+ \times p_S)_*\bigl( \ell^d \cdot p_-^* \bigr): H_{\mathcal{M}} \rightarrow H_{\mathcal{M} \times S}, \\
\label{equation_operator_F}
    F_d &\define (p_- \times p_S)_*\bigl( \ell^d \cdot p_+^* \bigr): H_{\mathcal{M}} \rightarrow H_{\mathcal{M} \times S},
\end{align}
for any integer $d \ge 0$. While $E_d$ and $F_d$ rely on correspondences, the Cartan part of the Yangian acts directly via cup product. We capture this by introducing the operator $H_d$, defined through the formal power series
\begin{equation}
    \label{equation_power_series_H}
    H(z) \define 1 + \sum_{d=0}^\infty \frac{H_d}{z^{d+1}} \;:\; H_{\mathcal{M}} \xrightarrow{p_{\mathcal{M}}^*} H_{\mathcal{M} \times S} \xrightarrow{\cdot \frac{c(\mathcal{U},z+t)}{c(\mathcal{U},z)}} H_{\mathcal{M} \times S},
\end{equation}
where $t = c_1(\omega_S)$ denotes the first Chern class of the canonical bundle of $S$, and $\mathcal{U}$ is the universal sheaf on $\mathcal{M} \times S$. The defining relations of the Yangian action on $H_{\mathcal{M}}$ are expressed by assembling the generators into the formal series
\begin{align}
\label{equation_power_series_E}
    E(z) &\define \sum_{d=0}^\infty \frac{E_d}{z^{d+1}} = (p_+ \times p_S)_*\left( \frac{1}{z-\ell} \cdot p_-^* \right), \\ 
\label{equation_power_series_F}
    F(z) &\define \sum_{d=0}^\infty \frac{F_d}{z^{d+1}} = (p_- \times p_S)_*\left( \frac{1}{z-\ell} \cdot p_+^* \right).
\end{align}

The central theorem of this paper establishes that the geometrically defined power series $E(z)$, $F(w)$, and $H(z)$ satisfy the defining commutation relations of the algebra $Y_{t_1,t_2}(\widehat{\mathfrak{gl}}_1)$. We state this result as follows.

\begin{theorem}%[\cite{negut_rationaltrigonometricellipticalgebras}, Theorem 3.13]
\label{theorem_relations_geometric_operators_stable}
    Let $S$ be a smooth projective surface over $\mathbb{C}$ and fix an ample divisor $H$ as well as a pair $(r,c_1) \in \mathbb{N} \times H^2(S,\mathbb{Z})$. Assume $S$ satisfies Assumptions A and S stated in \eqref{equation_assumption_A} and \eqref{equation_assumption_S}, respectively. We then have the following equalities of operators $H_{\mathcal{M}} \rightarrow H_{\mathcal{M} \times \textcolor{red}{S} \times \textcolor{blue}{S}}$
    \begin{align}
        \label{equation_relation_cohomology_E_E_intro}
        \Bigg[ \textcolor{red}{E(z)}\textcolor{blue}{E(w)}\widetilde{\zeta}^\text{coh}_-(w-z) \Bigg]_{z^{<0},w^{<0}} &= \Bigg[ \textcolor{blue}{E(w)}\textcolor{red}{E(z)}\widetilde{\zeta}^\text{coh}_+(z-w) \Bigg]_{z^{<0},w^{<0}} \\
        \label{equation_relation_cohomology_F_F_intro}
        \Bigg[ \textcolor{blue}{F(w)}\textcolor{red}{F(z)}\widetilde{\zeta}^\text{coh}_-(w-z) \Bigg]_{z^{<0},w^{<0}} &= \Bigg[ \textcolor{red}{F(z)}\textcolor{blue}{F(w)}\widetilde{\zeta}^\text{coh}_+(z-w) \Bigg]_{z^{<0},w^{<0}} \\
        \label{equation_relation_cohomology_H_E_intro}
        \textcolor{red}{H(z)}\textcolor{blue}{E(w)} &= \Bigg[ \textcolor{blue}{E(w)}\textcolor{red}{H(z)} \frac{\zeta^\text{coh}(z-w)}{\zeta^\text{coh}(w-z)} \Bigg]_{z \gg w,z^{\leq 0},w^{<0}} \\
        \label{equation_relation_cohomology_H_F_intro}
        \textcolor{blue}{F(w)}\textcolor{red}{H(z)} &= \Bigg[ \textcolor{red}{H(z)}\textcolor{blue}{F(w)} \frac{\zeta^\text{coh}(z-w)}{\zeta^\text{coh}(w-z)} \Bigg]_{z \gg w,z^{\leq 0},w^{<0}} \\
        \label{equation_relation_cohomology_H_H_intro}
        \bigl[ \textcolor{red}{H(z)},\textcolor{blue}{H(w)} \bigr] &= 0 \\
        \label{equation_relation_cohomology_E_F_intro}
        \bigl[ \textcolor{red}{E(z)},\textcolor{blue}{F(w)} \bigr] &= (-1)^{r-1} \frac{1}{t} (\operatorname{id}_{\mathcal{M}} \times \Delta)_* \left( \frac{H(z)-H(w)}{z-w} \right).
    \end{align}
\end{theorem}

The cohomological rational functions $\zeta^\text{coh}(x)$ and $\widetilde{\zeta}^\text{coh}_\pm$ are defined in \eqref{equation_zeta_cohomology} and \eqref{equation_zeta_tilde_cohomology}. To ensure clarity, the color of each operator indicates the corresponding factor of $\textcolor{red}{S} \times \textcolor{blue}{S}$ in which it operates, as in \eqref{equation_composition_geometric_operators_1} and \eqref{equation_composition_geometric_operators_2}. The bracket notation is explained below Definition \ref{definition_Yangian}.

The proofs of Theorems \ref{theorem_relations_geometric_operators_stable} and \ref{theorem_action_Yangian_stable} were sketched in \cite{negut_rationaltrigonometricellipticalgebras} using $K$-theoretic and excision arguments. In this paper, we fill in the details and provide a direct intersection-theoretic proof of the key commutator relation \eqref{equation_relation_cohomology_E_F_intro}. This stands in contrast to the excision trick employed in \textit{loc. cit.} and recovers the Yangian action natively within singular cohomology without passing through the Chern character isomorphism.

\begin{theorem}%[\cite{negut_rationaltrigonometricellipticalgebras}, Theorem 1.2]
\label{theorem_action_Yangian_stable}
    Let $S$ be a smooth projective surface over $\mathbb{C}$ and fix an ample divisor $H$ as well as a pair $(r,c_1) \in \mathbb{N} \times H^2(S,\mathbb{Z})$. Assume $S$ satisfies Assumptions A and S. Then, there is an action $$Y_{t_1,t_2}( \widehat{\mathfrak{gl}}_1 ) \curvearrowright H_{\mathcal{M}}$$ in the sense of Definition \ref{definition_action_Yangian}.
\end{theorem}

\begin{remark}
\label{remark_Chow}
    We adopt the non-standard notation $H_X$ to denote the singular cohomology of a scheme $X$ to emphasize that our results hold for other (co)homology theories, such as Borel--Moore homology or Chow rings. Indeed, the only properties we use are pull-back maps for local complete intersection morphisms, pushforward maps for proper morphisms, and Chern classes, all of which exist in these theories. The intersection-theoretic techniques employed here, namely base change, the projection formula, Mayer--Vietoris, excision, and the excess intersection formula, are natively established in the Chow ring setting (see \cite{fulton_intersection_theory}).
\end{remark}

\subsection{Strategy}
We now outline the proof of Theorem \ref{theorem_relations_geometric_operators_stable}. To verify the relations \eqref{equation_relation_cohomology_E_E_intro} and \eqref{equation_relation_cohomology_F_F_intro}, we adapt the derived-categorical approach developed in \cite{negut_hecke_correspondences}. The identities \eqref{equation_relation_cohomology_H_E_intro} and \eqref{equation_relation_cohomology_H_F_intro} are handled by modifying the $K$-theoretic arguments from \cite{negut_shuffle_surfaces}. Relation \eqref{equation_relation_cohomology_H_H_intro} follows immediately from the definition of the power series $H(z)$, as cup product operators naturally commute. While translating these arguments to the cohomological setting requires careful adaptation, the core technical challenge lies in proving the commutator relation \eqref{equation_relation_cohomology_E_F_intro}.

Rather than working directly with the generating series $E(z)$ and $F(w)$, we compute the commutator at the level of the individual operators $E_a$ and $F_b$, building upon results and strategies introduced by Zhao in \cite{zhao_commutator_moduli_stable} and \cite{zhao_commutator_Hilbert_schemes}. We first show that the composition $E_aF_b$ is realized as a correspondence on the nested moduli space $\mathfrak{Z}^\uparrow$ which is irreducible. Conversely, the reverse composition $F_bE_a$ is governed by a correspondence on $\mathfrak{Z}^\downarrow$. Crucially, the space $\mathfrak{Z}^\downarrow$ fails to be irreducible; in fact, it is not even equidimensional for stable sheaves of rank $r > 1$ (see Claim \ref{claim_irreducibility_nested_stable}). 

This failure presents the primary geometric hurdle of our construction. We overcome this via an application of the excess intersection formula, supported by Mayer--Vietoris, excision, and Grothendieck--Verdier duality (see Lemma \ref{lemma_base_change_stable}). By introducing the quadruple moduli space $\mathfrak{Y}^{\mathrm{rot}}$, we are able to split the commutator $\bigl[E_a,F_b\bigr]$ into two distinct contributions (see diagrams \eqref{equation_proof_E_F_4_stable} and \eqref{equation_proof_E_F_5_stable}). 

The subsequent simplification of these contributions relies on the formalism of Chern and Segre classes, exploiting the realization of various nested moduli spaces as projectivizations of two-term complexes. Finally, we deploy formal residue calculus to simplify the resulting expression.

\subsection{Organization}
The paper is organized as follows. Section \ref{chapter_moduli_spaces_operators} collects the geometric preliminaries: after establishing our conventions for correspondences, Chern polynomials, and projective bundles, we introduce the moduli space of stable sheaves $\mathcal{M}$ and the nested and quadruple moduli spaces that underpin our operators. We also record the dimension estimates and define the multi-step generalizations $E_{(d_1,\ldots,d_l)}$ and $F_{(d_1,\ldots,d_l)}$ needed in the sequel.

With the geometric setup in place, Section \ref{chapter_action_Yangian} recalls the definition of the affine Yangian $Y_{t_1,t_2}(\widehat{\mathfrak{gl}}_1)$ and establishes the commutation relations satisfied by the geometric operators $E_d$, $F_d$, and $H_d$. These relations are shown to be the cohomological analogues of the defining relations of the Yangian, culminating in the proofs of Theorems \ref{theorem_relations_geometric_operators_stable} and \ref{theorem_action_Yangian_stable}.

The technical core of the paper is contained in Section \ref{section_commutator_E_F}, which is devoted to the intersection-theoretic computation of the commutator $\bigl[E(z),F(w)\bigr]$.

\subsection{Outlook} 
We conclude this introduction by outlining two avenues for future research that extend the framework developed in this paper. One direction involves translating our intersection-theoretic framework from the moduli space of stable sheaves to the setting of Nakajima quiver varieties. These varieties are classically constructed as moduli spaces of stable representations of preprojective algebras and come equipped with Hecke correspondences that relate representations of different dimension vectors. In \cite{zhao_quiver_varieties}, Zhao analyzes the birational geometry of nested quiver varieties, proving that the blow-up of the diagonal of a fiber product of Hecke correspondences is isomorphic to a smooth quadruple moduli space. This result reveals a geometric structure fundamentally analogous to the space $\mathfrak{Y}^{\mathrm{rot}}$ utilized in our construction. Consequently, one could attempt to reconstruct the affine Yangian action on the singular cohomology (or Chow ring) of quiver varieties.

A second direction lies in connecting our framework with the $\mathcal{W}_{1+\infty}$-algebra structures explored by Li, Qin, and Wang in \cite{li_qin_wang_W_algebra}. In their work, the authors express certain Chern character operators (which are proved to be the zero-modes of vertex operators) in terms of Nakajima's Heisenberg operators, subsequently establishing their commutation relations. Given the structural parallels between $\mathcal{W}_{1+\infty}$ and the affine Yangian $Y_{t_1,t_2}(\widehat{\mathfrak{gl}}_1)$, it would be instructive to establish an explicit dictionary between their $\mathcal{W}$-algebra generators and the operators $E_d$, $F_d$, and $H_d$ defined in this paper. More fundamentally, realizing this dictionary raises the question of whether the intersection-theoretic machinery developed here can be deployed to recover the $\mathcal{W}$-algebra commutation relations geometrically, thereby bypassing the need for combinatorial vertex algebra techniques.

\subsection*{Acknowledgements}
This work was carried out as part of my Master's thesis at EPFL, under the supervision of Professor Andrei Neguţ. I express my deepest gratitude to him for his guidance and support, and to Archi Kaushik for many insightful mathematical discussions. Financial support from the EPFL Student Support Program is gratefully acknowledged.

\section{Geometric preliminaries}
\label{chapter_moduli_spaces_operators}

\subsection{Conventions}
We adopt the language of correspondences, following \cite{negut_rationaltrigonometricellipticalgebras}. A \textit{correspondence} between smooth projective varieties $X$ and $Y$ is a cohomology class $\Gamma \in H_{X \times Y}$ that induces an operator
\[
\Phi_{\Gamma}: H_Y \xrightarrow{p_Y^*} H_{X \times Y} \xrightarrow{\cdot \Gamma} H_{X \times Y} \xrightarrow{p_{X*}} H_X,
\]
where $p_X: X \times Y \rightarrow X$ and $p_Y: X \times Y \rightarrow Y$ are the projections and $\cdot$ denotes cup product. Such a correspondence is often represented by the diagram
\[
\begin{tikzcd}[row sep=small]
	& \Gamma \\
	& {X \times Y} \\
	X && Y,
	\arrow[dashed, no head, from=1-2, to=2-2]
	\arrow["{p_X}"', from=2-2, to=3-1]
	\arrow["{p_Y}", from=2-2, to=3-3]
\end{tikzcd}
\]
which encodes the operation of pulling back from $Y$, cupping with $\Gamma$, and pushing forward to $X$.

Correspondences admit a natural composition: for varieties $X,Y,Z$ and classes $\Gamma \in H_{X \times Y}$, $\Gamma' \in H_{Y \times Z}$, the composition $\Gamma \circ \Gamma'$ is defined by
\[
p_{XZ*}\left( p_{XY}^*(\Gamma) \cdot p_{YZ}^*(\Gamma') \right),
\]
where $p_{XY}$, $p_{XZ}$, and $p_{YZ}$ are the standard projections.

\begin{proposition}
    Given smooth projective varieties $X,Y,Z$ and cohomology classes $\Gamma \in H_{X \times Y}$, $\Gamma' \in H_{Y \times Z}$, we have
    \[
    \Phi_\Gamma \circ \Phi_{\Gamma'} = \Phi_{\Gamma \circ \Gamma'}.
    \]
    Moreover, $\Phi_{[\Delta]} = \operatorname{id}_{H_X}$, where $\Delta \subset X \times X$ is the diagonal.
\end{proposition}

Although our results are stated as equalities of operators, they actually hold at the level of correspondences. This distinction is non-trivial: the map from correspondences to operators is injective only for cohomology theories satisfying the Künneth decomposition $H_{X \times Y} \cong H_X \otimes H_Y$ and having a non-degenerate intersection pairing. The latter fails for Chow rings and holds for singular cohomology only after tensoring with $\mathbb{Q}$. Thus, establishing an identity of correspondences is generally stronger than proving the corresponding operator equality.

We next fix our conventions for Chern polynomials, again following \cite{negut_rationaltrigonometricellipticalgebras}. For a locally free sheaf $\mathcal{V}$ of rank $v$ on a scheme $X$, let $c_i(\mathcal{V}) \in H_X$ denote its \textit{Chern classes}. We assemble them into the \textit{Chern polynomial}
\[
c(\mathcal{V},z) \define \sum_{i=0}^v z^{v-i}(-1)^i c_i(\mathcal{V}) \in H_X[z].
\]
Because Chern classes depend only on the $K$-theory class of $\mathcal{V}$, the same is true of $c(\mathcal{V},z)$. Consequently, we may define Chern polynomials for sheaves of homological dimension $1$: given a short exact sequence of coherent sheaves
\[
0 \rightarrow \mathcal{W} \rightarrow \mathcal{V} \rightarrow \mathcal{E} \rightarrow 0
\]
with $\mathcal{V}$ and $\mathcal{W}$ locally free of ranks $v$ and $w$, we set
\[
c(\mathcal{E},z) \define \frac{c(\mathcal{V},z)}{c(\mathcal{W},z)}.
\]
This expression is independent of the chosen resolution, so we obtain well-defined Chern classes for $\mathcal{E}$ by expanding
\[
c(\mathcal{E},z) = \sum_{i=0}^\infty z^{v-w-i}(-1)^i c_i(\mathcal{E}) \in H_X((z^{-1})),
\]
where $H_X((z^{-1}))$ denotes formal Laurent series in $z^{-1}$ with coefficients in $H_X$.

We will frequently encounter projective bundles. For a locally free sheaf $\mathcal{V}$ on a scheme $X$, its \textit{projectivization} is
\[
\mathbb{P}_X(\mathcal{V}) \define \mathbf{Proj}_X\bigl(\operatorname{Sym}^\bullet(\mathcal{V})\bigr),
\]
with structure morphism $\rho: \mathbb{P}_X(\mathcal{V}) \to X$ (see Section II.7 of \cite{hartshorne_algebraic_geometry} for details). For any $f: T \to X$, there is a bijection
\[
\operatorname{Hom}_{\mathbf{Sch}_X}\bigl(T, \mathbb{P}_X(\mathcal{V})\bigr) \;\longleftrightarrow\; 
\Big\{ (\mathcal{L}, \tau) \;\Big|\; \mathcal{L} \text{ a line bundle on } T,\; 
\tau: f^*\mathcal{V} \twoheadrightarrow \mathcal{L} \text{ a surjection} \Big\}.
\]
In particular, there is a tautological surjection $\rho^*\mathcal{V} \twoheadrightarrow \mathcal{O}(1)$, inducing $\mathcal{O}(-1) \hookrightarrow \rho^*\mathcal{V}^\vee$, and closed points of $\mathbb{P}_X(\mathcal{V})$ parametrize one-dimensional quotients of the fibers of $\mathcal{V}$.

Given a map $\mathcal{W} \to \mathcal{V}$ of locally free sheaves on $X$, we consider closed embeddings
\begin{equation}
    \label{equation_definition_closed_embedding}
    Z(\sigma) \xhookrightarrow{i} \mathbb{P}_X(\mathcal{V}),
\end{equation}
where $i$ is cut out by the section $\sigma: \rho^*\mathcal{W} \rightarrow \rho^*\mathcal{V} \overset{\text{taut}}{\twoheadrightarrow} \mathcal{O}(1).$

Pushforwards of powers of $c_1(\mathcal{O}(1))$ along projective bundle morphisms are naturally computed using Segre classes. We follow \cite{fulton_intersection_theory}, noting that Fulton's dual definition of the projective bundle may introduce extra signs.

\begin{definition}[\cite{fulton_intersection_theory}, Chapter 4]
    Let $X$ be a smooth projective variety and $\mathcal{V}$ a locally free sheaf of rank $v$ on $X$. The $i$-\textit{th Segre class} of $\mathcal{V}$ is
    \[
    s_i(\mathcal{V}) \define (-1)^i \rho_* \left( c_1(\mathcal{O}(1))^{v-1+i} \right) \in A^i(X),
    \]
    where $A(X)$ is the Chow ring of $X$ and $\rho: \mathbb{P}_X(\mathcal{V}) \to X$ is the structure morphism. The \textit{total Segre class} is
    \[
    s(\mathcal{V}) \define \sum_{i=0}^\infty s_i(\mathcal{V}) \in A(X).
    \]
\end{definition}

Denote the \textit{total Chern class} of a locally free sheaf $\mathcal{V}$ on a smooth projective variety $X$ by $c(\mathcal{V})$. The total Segre and Chern classes are mutual inverses in the Chow ring.

\begin{proposition}[\cite{eisenbud_harris_2016}, Proposition 10.3]
    For a locally free sheaf $\mathcal{V}$ on a smooth projective variety $X$, we have
    \[
    s(\mathcal{V}) \, c(\mathcal{V}) = 1 \in A(X).
    \]
\end{proposition}

This identity has two immediate consequences. First, using $c_i(\mathcal{V}^\vee) = (-1)^i c_i(\mathcal{V})$, we obtain $s_i(\mathcal{V}^\vee) = (-1)^i s_i(\mathcal{V}).$ Second, for any short exact sequence of locally free sheaves
\[
0 \rightarrow \mathcal{W} \rightarrow \mathcal{V} \rightarrow \mathcal{E} \rightarrow 0,
\]
Whitney's formula gives $c(\mathcal{V}) = c(\mathcal{W})c(\mathcal{E})$; taking inverses yields $s(\mathcal{V}) = s(\mathcal{W})s(\mathcal{E}).$

\subsection{The moduli space of stable sheaves}
\label{subsection_moduli_stable}

We now introduce the moduli space of stable sheaves. After recalling the notion of stability for coherent sheaves, we state the representability theorem that constructs $\mathcal{M}$ as a projective scheme. Our exposition follows \cite{negut_shuffle_surfaces, negut_hecke_correspondences}; for a comprehensive treatment, we refer the reader to \cite{huybrechts_lehn_moduli}.

Let $S$ be a smooth projective surface over $\mathbb{C}$ and fix an ample divisor $H$. The \textit{Hilbert polynomial} of a coherent sheaf $\mathcal{F}$ on $S$ is given by 
$$P_\mathcal{F}(m) \define \chi(\mathcal{F} \otimes \mathcal{O}_S(mH)).$$
An application of the Hirzebruch--Riemann--Roch theorem yields 
$$P_\mathcal{F}(m) = \frac{rH \cdot H}{2}m^2 + \left( c_1 \cdot H - \frac{rH \cdot \omega_S}{2} \right)m + \left( \frac{c_1 \cdot c_1}{2}-c_2-\frac{c_1 \cdot \omega_S}{2}+r\chi(\mathcal{O}_S) \right),$$ 
where $r$, $c_1$, and $c_2$ denote the rank, first, and second Chern classes of $\mathcal{F}$, respectively. Additionally, $\omega_S$ represents the canonical bundle of $S$ (or the corresponding divisor). We define the \textit{reduced Hilbert polynomial} as 
$$p_\mathcal{F}(m) \define \frac{P_\mathcal{F}(m)}{r}$$ 
and note that it takes the form 
\begin{equation}
    \label{equation_reduced_Hilbert_polynomial}
    \frac{c_1 \cdot H}{r}m + \frac{1}{r} \left( \frac{c_1 \cdot c_1}{2}-c_2-\frac{c_1 \cdot \omega_S}{2} \right) + \operatorname{polynomial}(m),
\end{equation} 
where $\operatorname{polynomial}(m)$ contains terms independent of $r$, $c_1$, and $c_2$. 

For two polynomials $P(m)$ and $Q(m)$, we write $P(m) \geq Q(m)$ whenever this inequality holds for all sufficiently large $m$. Equivalently, this means that the coefficients of $P$ dominate those of $Q$ in lexicographic order.

\begin{definition}
    \label{definition_semistable_sheaf}
    A torsion-free coherent sheaf $\mathcal{F}$ on $S$ is called \textit{semistable} if 
    \begin{equation}
        \label{equation_definition_semistable_sheaf}
        p_\mathcal{F}(m) \geq p_\mathcal{G}(m)
    \end{equation} 
    holds for all proper subsheaves $\mathcal{G} \subset \mathcal{F}$.

    If the inequality in \eqref{equation_definition_semistable_sheaf} is strict for all proper $\mathcal{G} \subset \mathcal{F}$, we call $\mathcal{F}$ \textit{stable}.
\end{definition}

Before constructing the moduli space, we record some elementary consequences of the stability conditions. Henceforth, we impose Assumptions A and S on $S$ with respect to a pair $(r,c_1) \in \mathbb{N} \times H^2(S,\mathbb{Z})$.

Since the difference of reduced Hilbert polynomials $p_\mathcal{F}(m)-p_\mathcal{G}(m)$ is linear in $m$, a strict inequality in \eqref{equation_definition_semistable_sheaf} is equivalent to either
\begin{equation*}
    \frac{c_1 \cdot H}{r} > \frac{c_1' \cdot H}{r'}
\end{equation*}
or
$$\frac{c_1 \cdot H}{r} = \frac{c_1' \cdot H}{r'} \qquad \text{and} \qquad \frac{1}{r} \left( \frac{c_1 \cdot c_1}{2}-c_2-\frac{c_1 \cdot \omega_S}{2} \right) > \frac{1}{r'} \left( \frac{c_1' \cdot c_1'}{2}-c_2'-\frac{c_1' \cdot \omega_S}{2} \right),$$ 
where $r$, $c_1$, and $c_2$ denote the invariants of $\mathcal{F}$, and $r'$, $c_1'$, $c_2'$ those of the subsheaf $\mathcal{G}$. These observations yield the following result.

\begin{proposition}[\cite{negut_shuffle_surfaces}, Section 5]
    \label{proposition_semistable_stable}
    Under Assumption A of \eqref{equation_assumption_A}, a sheaf $\mathcal{F}$ is stable if and only if it is semistable.
\end{proposition}

In light of this equivalence, we henceforth focus on stable sheaves. We will frequently consider Hecke modifications: for two coherent sheaves $\mathcal{F}'$ and $\mathcal{F}$, the notation $\mathcal{F}' \subset_u \mathcal{F}$ indicates that $\mathcal{F}'$ is a subsheaf of $\mathcal{F}$ whose quotient is a skyscraper sheaf supported at the closed point $u \in S$. The following result is crucial for our discussion.

\begin{proposition}[\cite{negut_shuffle_surfaces}, Proposition 5.5]
    \label{proposition_Hecke_stable}
    Under Assumption A of \eqref{equation_assumption_A}, for any Hecke modification $\mathcal{F}' \subset_u \mathcal{F}$ the sheaf $\mathcal{F}$ is stable if and only if $\mathcal{F}'$ is stable.
\end{proposition}

A key feature of stable sheaves, which facilitates the construction of their moduli space, is their rigidity.

\begin{proposition}[\cite{huybrechts_lehn_moduli}, Corollary I.1.2.8]
    \label{proposition_stable_simple}
    Let $\mathcal{F}$ be a stable sheaf. Then, $\mathcal{F}$ is simple, i.e., 
    $$\Hom(\mathcal{F},\mathcal{F}) \cong \mathbb{C}.$$
\end{proposition}

We now fix an ample divisor $H$ on $S$, a pair $(r,c_1) \in \mathbb{N} \times H^2(S,\mathbb{Z})$, and an integer $c_2 \in \mathbb{Z}$. Consider the functor 
$$\mathcal{S}table_{(r,c_1,c_2)}: \mathbf{Sch}_\mathbb{C} \rightarrow \mathbf{Set}$$ 
mapping a $\mathbb{C}$-scheme $T$ to the set 
\begin{align}
    \label{equation_definition_functor_moduli_space}
    \Big\{ \mathcal{F} \;\Big|\; &\mathcal{F} \text{ a coherent sheaf on } T \times S, \text{ flat over } T \text{ such that } \mathcal{F}_t \text{ is stable} \\ \nonumber
    &\text{with invariants } (r,c_1,c_2) \text{ for every closed point } t \in T \Big\} \Big/ \sim,
\end{align}
where $\sim$ is the equivalence relation induced by tensoring with a line bundle pulled back from $T$. The following theorem, due to Huybrechts and Lehn, establishes the existence of the moduli space of stable sheaves.

\begin{theorem}[\cite{huybrechts_lehn_moduli}]
    \label{theorem_existence_moduli_space_of_stable_sheaves}
    Fix a pair $(r,c_1) \in \mathbb{N} \times H^2(S,\mathbb{Z})$ and $c_2 \in \mathbb{Z}$. Under Assumption A of \eqref{equation_assumption_A}, the functor $\mathcal{S}table_{(r,c_1,c_2)}$ is representable by a projective scheme $\mathcal{M}_{(r,c_1,c_2)}$. Under Assumption S of \eqref{equation_assumption_S}, the moduli space $\mathcal{M}_{(r,c_1,c_2)}$ is smooth of dimension
    \begin{equation}
        \label{equation_dimension_moduli_space}
        \operatorname{const} + 2rc_2,
    \end{equation}
    where $\operatorname{const}$ depends only on $S$, $H$, $r$, and $c_1$.
\end{theorem}

The precise formula for $\operatorname{const}$ can be found in \cite{huybrechts_lehn_moduli}. Due to the equivalence relation $\sim$ in \eqref{equation_definition_functor_moduli_space}, the universal sheaf $\mathcal{U}_{(r,c_1,c_2)}$ on $\mathcal{M}_{(r,c_1,c_2)} \times S$ is not unique, but is defined only up to tensoring with a line bundle pulled back from $\mathcal{M}_{(r,c_1,c_2)}$.

As in the definition of the Hilbert scheme of points, we consider the disjoint union
\begin{equation}
\label{equation_moduli_space_of_stable_sheaves}
    \mathcal{M} \define \coprod_{c_2 = -\infty}^\infty \mathcal{M}_{(r,c_1,c_2)}.
\end{equation}
The scheme $\mathcal{M}$ is called the \textit{moduli space of stable sheaves} on $S$ with invariants $r$ and $c_1$. The second Chern class $c_2$ is bounded below, as shown by the following standard result.

\begin{proposition}[\cite{huybrechts_lehn_moduli}, Theorem I.3.4.1]
    \label{proposition_Bogomolov_inequality}
    Fix a pair $(r,c_1) \in \mathbb{N} \times H^2(S,\mathbb{Z})$. If 
    $c_2 < \frac{r-1}{2r}c_1^2$, 
    then $\mathcal{M}_{(r,c_1,c_2)} = \varnothing$.
\end{proposition}

We refer to Section 5 of \cite{negut_shuffle_surfaces} or Chapter I.4 of \cite{huybrechts_lehn_moduli} for the details of the GIT construction of $\mathcal{M}$, which we have omitted here.

\subsection{Nested and quadruple moduli spaces}
\label{subsection_nested_stable}
We now introduce the nested and quadruple moduli spaces that underpin our geometric operators. These spaces are standard in the framework of \cite{negut_hecke_correspondences}. We recall their definitions and main properties here, but refer to \textit{loc. cit.} for the proofs of the results stated below.

Let $\mathcal{U}$ denote the universal sheaf on $\mathcal{M} \times S$. It is obtained as the disjoint union of the universal sheaves on the components $\mathcal{M}_{(r,c_1,c_2)} \times S$ over all $c_2$. These sheaves are compatible as $c_2$ varies, so the nested moduli spaces below carry universal sheaves that are naturally contained in one another. The universal sheaf $\mathcal{U}$ is flat over $\mathcal{M}$ and has homological dimension $1$. This means it admits a length‑$1$ resolution by locally free sheaves.

\begin{proposition}[\cite{negut_shuffle_surfaces}, Proposition 2.2]
    \label{proposition_short_exact_sequence_W_V_U}
    There exists a short exact sequence
    \begin{equation}
        \label{equation_short_exact_sequence_W_V_U}
        0 \rightarrow \mathcal{W} \rightarrow \mathcal{V} \rightarrow \mathcal{U} \rightarrow 0
    \end{equation}
    with $\mathcal{W}$ and $\mathcal{V}$ locally free sheaves on $\mathcal{M} \times S$.
\end{proposition}

A \textit{set partition} is an equivalence relation on a finite ordered set, represented symbolically. For instance, $(u,v,w)$ denotes the partition of a $3$-element set into three distinct $1$-element subsets, while $(u,v,u)$ identifies the first and last elements. The size $|\lambda|$ is the number of elements in the underlying set.

\begin{definition}
    \label{definition_nested_stable_space}
    Let $\lambda$ be a set partition of size $n$. We define the scheme
    \begin{align}
        \label{equation_definition_nested_Hilbert_scheme}
        \mathfrak{Z}_\lambda \define \Big\{ \big( \mathcal{F}_0 \subset_{u_1} \mathcal{F}_1 \subset_{u_2} \cdots \subset_{u_n} \mathcal{F}_n \big) \;\Big|\; &\text{stable sheaves, } u_1,\dots,u_n \in S, \\ \nonumber
        &\text{and } u_i = u_j \text{ whenever } i \sim j \text{ in } \lambda \Big\}.
    \end{align}
    If $n = 0$, we set $\mathfrak{Z}_\varnothing \define \mathcal{M}.$
\end{definition}

We will use the abbreviation $\mathfrak{Z}_n^\bullet \define \mathfrak{Z}_{(u,\dots,u)}$, where the bullet indicates that all support points coincide. Definition \ref{definition_nested_stable_space} describes only the closed points of $\mathfrak{Z}_\lambda$. The precise moduli interpretation is that $\mathfrak{Z}_\lambda$ represents the functor $\mathbf{Sch}_\mathbb{C} \to \mathbf{Set}$ sending a $\mathbb{C}$-scheme $T$ to the set
\begin{align}
    \label{equation_functor_Z_lambda}
    \Big\{ &\text{flags of coherent sheaves } \mathcal{F}_0 \subset \mathcal{F}_1 \subset \cdots \subset \mathcal{F}_n \text{ on } T \times S \;\Big|\; \text{flat over } T, \\ \nonumber
    &\text{such that for every closed point } t \in T \text{ the fibers } \mathcal{F}_{0,t},\dots,\mathcal{F}_{n,t} \text{ are} \\ \nonumber
    &\text{stable sheaves, together with the data i) and ii) below} \Big\} \Big/ \sim,
\end{align}
where $\sim$ is induced by tensoring all sheaves in a flag with the same line bundle pulled back from $T$. The additional data are:
\begin{enumerate}[i)]
    \item morphisms $u_1,\dots,u_n: T \to S$ such that $u_i = u_j$ whenever $i \sim j$ in $\lambda$;
    \item line bundles $\mathcal{L}_{0,1},\dots,\mathcal{L}_{n-1,n}$ on $T$ satisfying $\mathcal{J}_i/\mathcal{J}_{i-1} \cong \Gamma^i_* \mathcal{L}_{i-1,i}$, where $\Gamma^i: T \hookrightarrow T \times S$ denotes the graph of $u_i$.
\end{enumerate}

\begin{remark}
    \label{remark_moduli_spaces_closed_points}
    Once Proposition \ref{proposition_projective_bundle_description} is established, an induction on $|\lambda|$ shows that the functor \eqref{equation_functor_Z_lambda} is representable. Hence the scheme $\mathfrak{Z}_\lambda$ of Definition \ref{definition_nested_stable_space} exists. Throughout this paper, we describe moduli spaces by their closed points for simplicity. The corresponding relative statements for families are left as exercises to the interested reader.
\end{remark}

Let $\lambda$ be a set partition of size $n$. Denote by $|\lambda$ (respectively $\lambda|$) the partition obtained by dropping the first (respectively last) element of $\lambda$. With this notation, we have natural morphisms
\begin{equation}
    \label{equation_pi_minus_plus}
    \pi_-: \mathfrak{Z}_\lambda \rightarrow \mathfrak{Z}_{|\lambda} \times S^\# \qquad\text{and}\qquad 
    \pi_+: \mathfrak{Z}_\lambda \rightarrow \mathfrak{Z}_{\lambda|} \times S^\#,
\end{equation}
which forget the first and last sheaf in the flag, respectively. The symbol $\#$ indicates whether the forgotten support point is recorded in the target. It equals $1$ if the first (respectively last) element of $\lambda$ is isolated (that is, not equivalent to any other element), and $0$ otherwise. Thus $\pi_-$ (respectively $\pi_+$) records $u_1 \in S$ (respectively $u_n \in S$) precisely when $\# = 1$.

For any $0 \leq i \leq n$, there is a morphism $p_i: \mathfrak{Z}_\lambda \rightarrow \mathcal{M}$ remembering only $\mathcal{F}_i$ in the flag. Set $$\mathcal{U}_i \define (p_i \times \operatorname{id}_S)^*\mathcal{U},$$ where $\mathcal{U}$ is the universal sheaf on $\mathcal{M} \times S$. Pulling back the short exact sequence from Proposition \ref{proposition_short_exact_sequence_W_V_U} along $p_i \times \operatorname{id}_S$ gives
\begin{equation}
    \label{equation_short_exact_sequence_W_i_V_i_U_i}
    0 \rightarrow \mathcal{W}_i \rightarrow \mathcal{V}_i \rightarrow \mathcal{U}_i \rightarrow 0,
\end{equation}
with $\mathcal{W}_i$, $\mathcal{V}_i$ defined analogously. Although $p_i$ is not flat, the sequence remains exact because all sheaves in \eqref{equation_short_exact_sequence_W_V_U} are flat over $\mathcal{M}$.

The inclusion $\mathcal{F}_{i-1} \subset \mathcal{F}_i$ induces an inclusion of universal sheaves
\begin{equation}
    \label{equation_inclusion_universal_sheaves}
    \mathcal{U}_{i-1} \hookrightarrow \mathcal{U}_i
\end{equation}
on $\mathfrak{Z}_\lambda \times S$, for any $1 \leq i \leq n$ .

On $\mathfrak{Z}_\lambda$, we have \textit{tautological line bundles} $\mathcal{L}_{i-1,i}$ for $1 \leq i \leq n$. Their fiber over a closed point $\big( \mathcal{F}_0 \subset_{u_1} \mathcal{F}_1 \subset_{u_2} \cdots \subset_{u_n} \mathcal{F}_n \big)$ is $\Gamma(S, \mathcal{F}_i / \mathcal{F}_{i-1})$. Concretely, using the inclusion \eqref{equation_inclusion_universal_sheaves}, let $p_{\mathfrak{Z}_\lambda}: \mathfrak{Z}_\lambda \times S \rightarrow \mathfrak{Z}_\lambda$ be the projection. Then
\begin{equation*}
    \mathcal{L}_{i-1,i} \define p_{\mathfrak{Z}_\lambda*} \big( \mathcal{U}_i / \mathcal{U}_{i-1} \big).
\end{equation*}
The quotient $\mathcal{U}_i / \mathcal{U}_{i-1}$ is supported on the graph $\Gamma: \mathfrak{Z}_\lambda \rightarrow \mathfrak{Z}_\lambda \times S$ of $p_i$. This yields a short exact sequence
\begin{equation}
    \label{equation_short_exact_sequence_universal_sheaves_1}
    0 \rightarrow \mathcal{U}_{i-1} \rightarrow \mathcal{U}_i \rightarrow \Gamma_* \mathcal{L}_{i-1,i} \rightarrow 0.
\end{equation}

%The following projective bundle description is the key technical input for the representability of the nested moduli space $\mathfrak{Z}_\lambda$.

\begin{proposition}[\cite{negut_hecke_correspondences}, Proposition 2.19]
    \label{proposition_projective_bundle_description}
    The morphism $\pi_-$ of \eqref{equation_pi_minus_plus} can be realized as the diagonal arrow in
\[\begin{tikzcd}[row sep=small]
	{\mathfrak{Z}_\lambda} & {\mathbb{P}_{\mathfrak{Z}_{|\lambda} \times S} \left( \mathcal{V}_1 \right)} && {\mathfrak{Z}_\lambda} & {\mathbb{P}_{\mathfrak{Z}_{|\lambda}} \left( \Gamma^{u*} \mathcal{V}_1 \right)} \\
	& {\mathfrak{Z}_{|\lambda} \times S} &&& {\mathfrak{Z}_{|\lambda},}
	\arrow["{i_-}", hook, from=1-1, to=1-2]
	\arrow["{\pi_-}"', from=1-1, to=2-2]
	\arrow["{\rho_-}", from=1-2, to=2-2]
	\arrow["{i_-}", hook, from=1-4, to=1-5]
	\arrow["{\pi_-}"', from=1-4, to=2-5]
	\arrow["{\rho_-}", from=1-5, to=2-5]
\end{tikzcd}\]
    where $\Gamma^u: \mathfrak{Z}_{|\lambda} \rightarrow \mathfrak{Z}_{|\lambda} \times S$ is the graph of the map recording the forgotten support point. The embedding $i_-$ is cut out by the following composition on the projectivization:
    \begin{align*}
        \rho_-^*\mathcal{W}_1 \rightarrow \rho_-^*\mathcal{V}_1 \overset{\text{taut}}{\twoheadrightarrow} &\mathcal{O}(1) \qquad \qquad \text{if } \#=1, \\
        \rho_-^*\Gamma^{u*}\mathcal{W}_1 \rightarrow \rho_-^*\Gamma^{u*}\mathcal{V}_1 \overset{\text{taut}}{\twoheadrightarrow} &\mathcal{O}(1) \qquad \qquad \text{if } \#=0.
    \end{align*}
    The line bundle $\mathcal{L}_{0,1}$ on $\mathfrak{Z}_\lambda$ is the restriction of $\mathcal{O}(1)$.

    Similarly, $\pi_+$ factors as
\[\begin{tikzcd}[row sep=small]
	{\mathfrak{Z}_\lambda} & {\mathbb{P}_{\mathfrak{Z}_{\lambda|} \times S} \left( \mathcal{W}_{n-1}^\vee \otimes \omega_S \right)} && {\mathfrak{Z}_\lambda} & {\mathbb{P}_{\mathfrak{Z}_{\lambda|}} \left( \Gamma^{u*} \left(\mathcal{W}_{n-1}^\vee \otimes \omega_S\right) \right)} \\
	& {\mathfrak{Z}_{\lambda|} \times S} &&& {\mathfrak{Z}_{\lambda|},}
	\arrow["{i_+}", hook, from=1-1, to=1-2]
	\arrow["{\pi_+}"', from=1-1, to=2-2]
	\arrow["{\rho_+}", from=1-2, to=2-2]
	\arrow["{i_+}", hook, from=1-4, to=1-5]
	\arrow["{\pi_+}"', from=1-4, to=2-5]
	\arrow["{\rho_+}", from=1-5, to=2-5]
\end{tikzcd}\]
    where $\omega_S$ denotes the canonical line bundle and its pull-backs. The embedding $i_+$ is cut out by
    \begin{align*}
        \rho_+^* \left( \mathcal{V}_{n-1}^\vee \otimes \omega_S \right) \rightarrow \rho_+^* \left( \mathcal{W}_{n-1}^\vee \otimes \omega_S \right) \overset{\text{taut}}{\twoheadrightarrow} &\mathcal{O}(1) \qquad \qquad \text{if } \#=1, \\
        \rho_+^* \Gamma^{u*}\left( \mathcal{V}_{n-1}^\vee \otimes \omega_S \right) \rightarrow \rho_+^* \Gamma^{u*}\left( \mathcal{W}_{n-1}^\vee \otimes \omega_S \right) \overset{\text{taut}}{\twoheadrightarrow} &\mathcal{O}(1) \qquad \qquad \text{if } \#=0.
    \end{align*}
    The line bundle $\mathcal{L}_{n-1,n}$ on $\mathfrak{Z}_\lambda$ is the restriction of $\mathcal{O}(-1)$.
\end{proposition}

We now summarize the geometric properties of the nested moduli spaces $\mathfrak{Z}_\lambda$ for $n = |\lambda| \leq 4$. By irreducibility of $\mathfrak{Z}_\lambda$, we mean that for every connected component $C$ of $\mathcal{M} \times S^n$, the preimage $\pi^{-1}(C)$ is irreducible, where
\[
\pi: \mathfrak{Z}_\lambda \rightarrow \mathcal{M} \times S^n,\qquad
\big( \mathcal{F}_0 \subset_{u_1} \mathcal{F}_1 \subset_{u_2} \cdots \subset_{u_n} \mathcal{F}_n \big) \mapsto \big(\mathcal{F}_n, u_1, \dots, u_n\big).
\]

\begin{proposition}[\cite{negut_hecke_correspondences}, Propositions 2.26--2.33]
    \label{proposition_nested_Hilbert_schemes_geometric_properties}
    The following hold for nested moduli spaces.
    \begin{enumerate}[i)]
        \item $\mathfrak{Z}_{(u)}$ and $\mathfrak{Z}_{(u,u)}$ are smooth and irreducible.
        \item $\mathfrak{Z}_{(u,v)}$, $\mathfrak{Z}_{(u,u,u)}$, $\mathfrak{Z}_{(u,u,v)}$, and $\mathfrak{Z}_{(u,v,v)}$ are locally complete intersection and irreducible.
        \item $\mathfrak{Z}_{(u,v,u)}$ is Cohen--Macaulay and irreducible.
        \item $\mathfrak{Z}_{(u,u,u,u)}$ is locally complete intersection and has two irreducible components.
        \item $\mathfrak{Z}_{(u,u,v,u)}$ and $\mathfrak{Z}_{(u,v,u,u)}$ are Cohen--Macaulay and irreducible.
    \end{enumerate}
\end{proposition}

\begin{proposition}[\cite{negut_hecke_correspondences}, Proposition 2.34]
    \label{proposition_nested_Hilbert_schemes_normality}
    The nested moduli space $\mathfrak{Z}_\lambda$ is normal for
    \[
    \lambda \in \Big\{ (u,v), (u,u,v), (u,v,u), (u,v,v), (u,u,v,u), (u,v,u,u) \Big\}.
    \]
\end{proposition}

We next introduce the quadruple moduli spaces $\mathfrak{Y}$, $\mathfrak{Y}_-$, $\mathfrak{Y}_+$, and $\mathfrak{Y}_{-+}$. These parametrize diagrams of stable sheaves of the forms
\begin{equation}
    \label{equation_diagram_Y}
    \begin{tikzcd}[sep=small]
	& {\mathcal{J}_1} & \\
	{\mathcal{J}_0} && {\mathcal{J}_2} \\
	& {\mathcal{J}_1'}
	\arrow["v", hook, from=1-2, to=2-3]
	\arrow["u", hook, from=2-1, to=1-2]
	\arrow["v"', hook, from=2-1, to=3-2]
	\arrow["u"', hook, from=3-2, to=2-3]
\end{tikzcd}
\end{equation}
\begin{equation}
    \label{equation_diagram_Y_minus}
    \begin{tikzcd}[sep=small]
	& {\mathcal{J}_1} && \\
	{\mathcal{J}_0} && {\mathcal{J}_2} & {\mathcal{J}_3} \\
	& {\mathcal{J}_1'}
	\arrow["v", hook, from=1-2, to=2-3]
	\arrow["u", hook, from=2-1, to=1-2]
	\arrow["v"', hook, from=2-1, to=3-2]
	\arrow["u", hook, from=2-3, to=2-4]
	\arrow["u"', hook, from=3-2, to=2-3]
\end{tikzcd}
\end{equation}
\begin{equation}
    \label{equation_diagram_Y_plus}
    \begin{tikzcd}[sep=small]
	&& {\mathcal{J}_2} & \\
	{\mathcal{J}_0} & {\mathcal{J}_1} && {\mathcal{J}_3} \\
	&& {\mathcal{J}_2'}
	\arrow["v", hook, from=1-3, to=2-4]
	\arrow["u", hook, from=2-1, to=2-2]
	\arrow["u", hook, from=2-2, to=1-3]
	\arrow["v"', hook, from=2-2, to=3-3]
	\arrow["u"', hook, from=3-3, to=2-4]
\end{tikzcd}
\end{equation}
\begin{equation}
    \label{equation_diagram_Y_minus_plus}
    \begin{tikzcd}[sep=small]
	&& {\mathcal{J}_2} && \\
	{\mathcal{J}_0} & {\mathcal{J}_1} && {\mathcal{J}_3} & {\mathcal{J}_4} \\
	&& {\mathcal{J}_2'}
	\arrow["v", hook, from=1-3, to=2-4]
	\arrow["u", hook, from=2-1, to=2-2]
	\arrow["u", hook, from=2-2, to=1-3]
	\arrow["v"', hook, from=2-2, to=3-3]
	\arrow["u", hook, from=2-4, to=2-5]
	\arrow["u"', hook, from=3-3, to=2-4]
\end{tikzcd}
\end{equation}
respectively. Each inclusion is colength $1$ and supported at the indicated point. As in Remark \ref{remark_moduli_spaces_closed_points}, we describe only closed points. The functorial interpretation is standard.

Each of these spaces carries tautological line bundles $\mathcal{L}_{i-1,i}$ with fibers $\Gamma(S, \mathcal{J}_i / \mathcal{J}_{i-1})$. Where applicable, there are also primed versions $\mathcal{L}'_{i-1,i}$. They satisfy $\mathcal{L}_{0,1} \mathcal{L}_{1,2} = \mathcal{L}'_{0,1} \mathcal{L}'_{1,2}$ on $\mathfrak{Y}$ and $\mathfrak{Y}_-$. On $\mathfrak{Y}_+$ and $\mathfrak{Y}_{-+}$, we have $\mathcal{L}_{1,2} \mathcal{L}_{2,3} = \mathcal{L}'_{1,2} \mathcal{L}'_{2,3}.$

There are natural morphisms
\begin{equation}
    \label{equation_definition_morphisms_uparrow}
    \begin{tikzcd}[row sep=small]
	{\mathfrak{Y}} & {\mathfrak{Y}_-} & {\mathfrak{Y}_+} & {\mathfrak{Y}_{-+}} \\
	{\mathfrak{Z}_{(u,v)}} & {\mathfrak{Z}_{(u,v,u)}} & {\mathfrak{Z}_{(u,u,v)}} & {\mathfrak{Z}_{(u,u,v,u)}}
	\arrow["{\pi^\uparrow}"', from=1-1, to=2-1]
	\arrow["{\pi^\uparrow}"', from=1-2, to=2-2]
	\arrow["{\pi^\uparrow}"', from=1-3, to=2-3]
	\arrow["{\pi^\uparrow}"', from=1-4, to=2-4]
\end{tikzcd}
\end{equation}
remembering the middle and top parts, and
\begin{equation}
    \label{equation_definition_morphisms_downarrow}
    \begin{tikzcd}[row sep=small]
	{\mathfrak{Y}} & {\mathfrak{Y}_-} & {\mathfrak{Y}_+} & {\mathfrak{Y}_{-+}} \\
	{\mathfrak{Z}_{(v,u)}} & {\mathfrak{Z}_{(v,u,u)}} & {\mathfrak{Z}_{(u,v,u)}} & {\mathfrak{Z}_{(u,v,u,u)}}
	\arrow["{\pi^\downarrow}"', from=1-1, to=2-1]
	\arrow["{\pi^\downarrow}"', from=1-2, to=2-2]
	\arrow["{\pi^\downarrow}"', from=1-3, to=2-3]
	\arrow["{\pi^\downarrow}"', from=1-4, to=2-4]
\end{tikzcd}
\end{equation}
remembering the middle and bottom parts. 

\begin{comment}
The following proposition describes these morphisms projectively.

\begin{proposition}[\cite{negut_hecke_correspondences}, Proposition 2.37]
    \label{proposition_projective_bundle_description_Y}
    There exists a locally free sheaf $\mathcal{E}$ on $\mathfrak{Z}_{(u,v)}$ (respectively on $\mathfrak{Z}_{(v,u)}$) such that $\pi^\uparrow$ (respectively $\pi^\downarrow$) factors as
\[\begin{tikzcd}
	{\mathfrak{Y}} & {\mathbb{P}_{\mathfrak{Z}_{(u,v)}} \left( \mathcal{E} \right)} && {\mathfrak{Y}} & {\mathbb{P}_{\mathfrak{Z}_{(v,u)}} \left( \mathcal{E} \right)} \\
	& {\mathfrak{Z}_{(u,v)}} &&& {\mathfrak{Z}_{(v,u)}}
	\arrow["{i^\uparrow}", hook, from=1-1, to=1-2]
	\arrow["{\pi^\uparrow}"', from=1-1, to=2-2]
	\arrow["{\rho^\uparrow}", from=1-2, to=2-2]
	\arrow["{i^\downarrow}", from=1-4, to=1-5]
	\arrow["{\pi^\downarrow}"', from=1-4, to=2-5]
	\arrow["{\rho^\downarrow}", from=1-5, to=2-5]
\end{tikzcd}\]
    with the embeddings cut out by compositions of maps as in \eqref{equation_definition_closed_embedding}. The same holds for $\mathfrak{Y}_-$, $\mathfrak{Y}_+$, and $\mathfrak{Y}_{-+}$ with the obvious substitutions.
\end{proposition}

Together with Proposition \ref{proposition_projective_bundle_description}, this establishes the representability of the quadruple spaces. 
\end{comment}

The following vanishing locus result will be essential in establishing the relations of Section \ref{subsection_relations}.

\begin{proposition}[\cite{negut_hecke_correspondences}, Proposition 2.39]
    \label{proposition_vanishing_locus_line_bundle_map}
    On $\mathfrak{Y}$, the composition
    $
    \mathcal{F}_1/\mathcal{F}_0 \hookrightarrow \mathcal{F}_2/\mathcal{F}_0 \twoheadrightarrow \mathcal{F}_2/\mathcal{F}_1'
    $
    induces a map of line bundles $\mathcal{L}_{0,1} \rightarrow \mathcal{L}_{1,2}'$. Its zero locus is
    $$
    \Big\{ \mathcal{F}_1 = \mathcal{F}_1', u = v \Big\} \hookrightarrow \mathfrak{Y},
    $$
    which is isomorphic to $\mathfrak{Z}_{(u,u)}$. Analogous statements hold for $\mathfrak{Y}_-$, $\mathfrak{Y}_+$, and $\mathfrak{Y}_{-+}$.
\end{proposition}

The properties of the quadruple spaces are as follows.

\begin{proposition}[\cite{negut_hecke_correspondences}, Propositions 2.41--2.43]
    \label{proposition_moduli_spaces_Y_geometric_properties}
    The following hold.
    \begin{enumerate}[i)]
        \item $\mathfrak{Y}$ is smooth and irreducible.
        \item $\mathfrak{Y}_-$ and $\mathfrak{Y}_+$ are locally complete intersection and irreducible.
        \item $\mathfrak{Y}_{-+}$ is locally complete intersection and has two irreducible components.
    \end{enumerate}
\end{proposition}

\begin{proposition}[\cite{negut_hecke_correspondences}, Proposition 2.44]
    \label{proposition_moduli_spaces_Y_reduced}
    The schemes $\mathfrak{Y}$, $\mathfrak{Y}_-$, $\mathfrak{Y}_+$, and $\mathfrak{Y}_{-+}$ are reduced.
\end{proposition}

Finally, we record the birationality of the maps $\pi^\uparrow$ and $\pi^\downarrow$.

\begin{proposition}
    \label{proposition_birationality_uparrow_downarrow}
    The morphisms \eqref{equation_definition_morphisms_uparrow} and \eqref{equation_definition_morphisms_downarrow} are birational.
\end{proposition}

\begin{proof}
    We prove the statement for $\pi^\uparrow: \mathfrak{Y} \rightarrow \mathfrak{Z}_{(u,v)}$. The nested moduli spaces
    $\{ \mathcal{F}_0 \subset_u \mathcal{F}_1 \subset_v \mathcal{F}_2 \}$ and $\{ \mathcal{F}_0 \subset_v \mathcal{F}_1' \subset_u \mathcal{F}_2 \}
    $
    are isomorphic over $\{u \neq v\}$. Given $\big(\mathcal{F}_0 \subset_u \mathcal{F}_1 \subset_v \mathcal{F}_2\big)$ with $u \neq v$, construct $\mathcal{F}_1'$ by gluing
    $$
    \begin{cases}
        \mathcal{F}_1'|_{S \setminus \{v\}} &\hspace{-1em}= \mathcal{F}_0|_{S \setminus \{v\}} \\
        \mathcal{F}_1'|_{S \setminus \{u\}} \, &\hspace{-1em}= \mathcal{F}_2|_{S \setminus \{u\}}.
    \end{cases}
    $$
    This is well-defined because $\mathcal{F}_0|_{S \setminus \{u,v\}} = \mathcal{F}_2|_{S \setminus \{u,v\}}$. The inverse assignment is analogous. Hence $\pi^\uparrow$ is birational. The other cases are identical.
\end{proof}

In fact, these morphisms satisfy stronger vanishing of higher direct images (see Proposition 2.45 of \cite{negut_hecke_correspondences}). We will not need this, however. Since we work with singular cohomology, birationality of proper morphisms is sufficient.

\subsection{Dimension estimates}
\label{subsection_dimension_estimates_stable}
We now establish the dimension estimates for the nested and quadruple moduli spaces of stable sheaves. These estimates rely on a careful analysis of the associated Quot schemes. Again, we state the results without proof and refer to Section 5 of \cite{negut_hecke_correspondences} for full details.

Let $\lambda$ be a set partition of size $n$ and denote by $k$ the number of distinct elements of $\lambda$. 

\begin{definition}
    The \textit{expected dimension} of a connected component of $\mathfrak{Z}_\lambda$ is defined as 
    $$\operatorname{const} + r\big(c_2(\mathcal{F}_0)+c_2(\mathcal{F}_n)\big)+k$$ 
    for any closed point $\big( \mathcal{F}_0 \subset_{u_1} \mathcal{F}_1 \subset_{u_2} \cdots \subset_{u_n} \mathcal{F}_n \big) \in \mathfrak{Z}_\lambda$. Here, $\operatorname{const}$ refers to the constant appearing in \eqref{equation_dimension_moduli_space} that depends only on $S$, $H$, $r$, and $c_1$. 
\end{definition}

The following proposition, drawn from \cite{negut_hecke_correspondences}, collects the dimension computations relevant to our geometric operators.

\begin{proposition}[\cite{negut_hecke_correspondences}, Propositions 2.26--2.33]
    \label{proposition_dimension_nested_stable}
    All irreducible components of the nested moduli space of stable sheaves $\mathfrak{Z}_\lambda$ are of expected dimension for any set partition
    \[
    \lambda \in \Big\{ (u), (u,u), (u,v), (u,u,u), (u,u,v), (u,v,u), (u,v,v), (u,u,u,u), (u,u,v,u), (u,v,u,u) \Big\}.
    \]
\end{proposition}

We now turn to the dimensions of the quadruple moduli spaces. Their expected dimensions are naturally governed by the maps $\pi^\uparrow$ and $\pi^\downarrow$ introduced previously.

\begin{definition}
    The \textit{expected dimension} of a connected component of the moduli spaces $\mathfrak{Y}$, $\mathfrak{Y}_-$, $\mathfrak{Y}_+$, and $\mathfrak{Y}_{-+}$ is defined as the dimension of the respective connected component of the codomains of the morphisms $\pi^\uparrow$ or $\pi^\downarrow$. By Proposition \ref{proposition_dimension_nested_stable}, this is equal to the actual dimension of these nested moduli schemes.
\end{definition}

\begin{proposition}[\cite{negut_hecke_correspondences}, Propositions 2.41--2.43]
    \label{proposition_dimension_quadruple_stable}
    All irreducible components of the moduli spaces $\mathfrak{Y}$, $\mathfrak{Y}_-$, $\mathfrak{Y}_+$, and $\mathfrak{Y}_{-+}$ are of expected dimension.    
\end{proposition}

\subsection{Multi-step operators}
\label{subsection_multi_step_operators}
The arguments of Section \ref{subsection_relations} require multi-step generalizations of the operators \(E_d\) and \(F_d\) defined in Section \ref{section_main_results}. For any integer \(l > 0\), define
\begin{equation}
    \label{equation_operator_E_multiple}
    E_{(d_1,\ldots,d_l)} \define (p_+ \times p_S)_*\Bigl( \ell^{d_1} \cdot \pi_{+*}\pi_-^*\bigl( \ell^{d_2} \cdot \pi_{+*} \cdots \pi_-^*(\ell^{d_l} \cdot p_-^*) \cdots \bigr) \Bigr): H_{\mathcal{M}} \rightarrow H_{\mathcal{M} \times S},
\end{equation}
as illustrated by the diagram
\[
\begin{tikzcd}[column sep=tiny,row sep=small]
    {\Big\{\mathcal{F}_0 \subset_u \mathcal{F}_1 \subset_u \mathcal{F}_2\Big\}} & {\ell^{d_2}} & \cdots & {\ell^{d_{l-1}}} & {\Big\{\mathcal{F}_{l-2} \subset_u \mathcal{F}_{l-1} \subset_u \mathcal{F}_l\Big\}} \\
    {\Big\{\mathcal{F}_0 \subset_u \mathcal{F}_1\Big\}} & {\Big\{\mathcal{F}_1 \subset_u \mathcal{F}_2\Big\}} && {\Big\{\mathcal{F}_{l-2} \subset_u \mathcal{F}_{l-1}\Big\}} & {\Big\{\mathcal{F}_{l-1} \subset_u \mathcal{F}_l\Big\}} \\
    {\mathcal{M} \times S} & {\ell^{d_1}} && {\ell^{d_l}} & {\mathcal{M}.}
    \arrow["{\pi_+}"', from=1-1, to=2-1]
    \arrow["{\pi_-}", from=1-1, to=2-2]
    \arrow[dashed, no head, from=1-2, to=2-2]
    \arrow["{\pi_+}"', from=1-3, to=2-2]
    \arrow["{\pi_-}", from=1-3, to=2-4]
    \arrow[dashed, no head, from=1-4, to=2-4]
    \arrow["{\pi_+}"', from=1-5, to=2-4]
    \arrow["{\pi_-}", from=1-5, to=2-5]
    \arrow["{p_+ \times p_S}"', from=2-1, to=3-1]
    \arrow["{p_-}", from=2-5, to=3-5]
    \arrow[dashed, no head, from=3-2, to=2-1]
    \arrow[dashed, no head, from=3-4, to=2-5]
\end{tikzcd}
\]
Similarly, we define
\begin{equation}
    \label{equation_operator_F_multiple}
    F_{(d_1,\ldots,d_l)} \define (p_- \times p_S)_*\Bigl( \ell^{d_1} \cdot \pi_{-*}\pi_+^*\bigl( \ell^{d_2} \cdot \pi_{-*} \cdots \pi_+^*(\ell^{d_l} \cdot p_+^*) \cdots \bigr) \Bigr): H_{\mathcal{M}} \rightarrow H_{\mathcal{M} \times S},
\end{equation}
with the evident diagram.

\begin{comment}
\[
\begin{tikzcd}[column sep=tiny,row sep=small]
    {\Big\{\mathcal{F}_0 \supset_u \mathcal{F}_1 \supset_u \mathcal{F}_2\Big\}} & {\ell^{d_2}} & \cdots & {\ell^{d_{l-1}}} & {\Big\{\mathcal{F}_{l-2} \supset_u \mathcal{F}_{l-1} \supset_u \mathcal{F}_l\Big\}} \\
    {\Big\{\mathcal{F}_0 \supset_u \mathcal{F}_1\Big\}} & {\Big\{\mathcal{F}_1 \supset_u \mathcal{F}_2\Big\}} && {\Big\{\mathcal{F}_{l-2} \supset_u \mathcal{F}_{l-1}\Big\}} & {\Big\{\mathcal{F}_{l-1} \supset_u \mathcal{F}_l\Big\}} \\
    {\mathcal{M} \times S} & {\ell^{d_1}} && {\ell^{d_l}} & {\mathcal{M}}
    \arrow["{\pi_-}"', from=1-1, to=2-1]
    \arrow["{\pi_+}", from=1-1, to=2-2]
    \arrow[dashed, no head, from=1-2, to=2-2]
    \arrow["{\pi_-}"', from=1-3, to=2-2]
    \arrow["{\pi_+}", from=1-3, to=2-4]
    \arrow[dashed, no head, from=1-4, to=2-4]
    \arrow["{\pi_-}"', from=1-5, to=2-4]
    \arrow["{\pi_+}", from=1-5, to=2-5]
    \arrow["{p_- \times p_S}"', from=2-1, to=3-1]
    \arrow["{p_+}", from=2-5, to=3-5]
    \arrow[dashed, no head, from=3-2, to=2-1]
    \arrow[dashed, no head, from=3-4, to=2-5]
\end{tikzcd}
\]
\end{comment}

For \(l=1\), these definitions recover the operators \(E_{d_1}\) and \(F_{d_1}\) from \eqref{equation_operator_E} and \eqref{equation_operator_F}, respectively.

\section{The action on the cohomology of the moduli space of stable sheaves}
\label{chapter_action_Yangian}

\subsection{The Yangian of \texorpdfstring{$\widehat{\mathfrak{gl}}_1$}{gl1}}
\label{subsection_Yangian}

We begin by recalling the definition of the affine Yangian \(Y_{t_1,t_2}(\widehat{\mathfrak{gl}}_1)\) and the relations among its generators. Our exposition follows \cite{negut_rationaltrigonometricellipticalgebras}. We work over the polynomial ring \(\mathbb{Z}[t_1,t_2]\) and set \(t = t_1 + t_2\). Define the rational functions
\begin{equation} \label{equation_zeta_Yangian}
    \zeta^\mathbb{C}(x) \define \frac{(x+t_1)(x+t_2)}{x(x+t)} \qquad \text{and} \qquad \widetilde{\zeta}^\mathbb{C}(x) \define \zeta^\mathbb{C}(x)(x+t)(x-t) = \frac{(x+t_1)(x+t_2)(x-t)}{x}.
\end{equation}

\begin{definition}[\cite{tsymbaliuk_yangian}]
    \label{definition_Yangian}
    The \textit{affine Yangian of $\widehat{\mathfrak{gl}}_1$} is the algebra
    \[
    Y_{t_1,t_2}( \widehat{\mathfrak{gl}}_1 ) \define 
    \faktor{\mathbb{Z}[t_1,t_2]\langle e_d, f_d, h_d \rangle_{d \geq 0}}
    {\text{relations (\ref{equation_relation_Yangian_e_e})--(\ref{equation_relation_Yangian_e_f})}}.
    \]
    In terms of the generating series
    \[
    e(z) = \sum_{d=0}^\infty \frac{e_d}{z^{d+1}},\qquad 
    f(z) = \sum_{d=0}^\infty \frac{f_d}{z^{d+1}},\qquad 
    h(z) = 1 + \sum_{d=0}^\infty \frac{h_d}{z^{d+1}},
    \]
    the defining relations take the form
    \begin{align}
        \label{equation_relation_Yangian_e_e}
        \Bigg[ e(z)e(w) \widetilde{\zeta}^\mathbb{C}(w-z) \Bigg]_{z^{<0},w^{<0}} 
        &= \Bigg[ e(w)e(z) \widetilde{\zeta}^\mathbb{C}(z-w) \Bigg]_{z^{<0},w^{<0}} \\
        \label{equation_relation_Yangian_f_f}
        \Bigg[ f(w)f(z) \widetilde{\zeta}^\mathbb{C}(w-z) \Bigg]_{z^{<0},w^{<0}} 
        &= \Bigg[ f(z)f(w) \widetilde{\zeta}^\mathbb{C}(z-w) \Bigg]_{z^{<0},w^{<0}} \\
        \label{equation_relation_Yangian_h_e}
        h(z)e(w) &= \Bigg[ e(w)h(z) \frac{\zeta^\mathbb{C}(z-w)}{\zeta^\mathbb{C}(w-z)} \Bigg]_{z \gg w,\, z^{\leq 0},\, w^{<0}} \\
        \label{equation_relation_Yangian_h_f}
        f(w)h(z) &= \Bigg[ h(z)f(w) \frac{\zeta^\mathbb{C}(z-w)}{\zeta^\mathbb{C}(w-z)} \Bigg]_{z \gg w,\, z^{\leq 0},\, w^{<0}} \\
        \label{equation_relation_Yangian_h_h}
        \bigl[ h(z),h(w) \bigr] &= 0 \\
        \label{equation_relation_Yangian_e_f}
        \bigl[ e(z),f(w) \bigr] &= \frac{t_1 t_2}{t} \cdot \frac{h(z)-h(w)}{z-w}.
    \end{align}
\end{definition}

We now clarify the bracket notation in \eqref{equation_relation_Yangian_e_e}--\eqref{equation_relation_Yangian_e_f}. In (\ref{equation_relation_Yangian_e_e}) and (\ref{equation_relation_Yangian_f_f}), one cancels the factor \(z-w\) from the denominator and equates the coefficients of all monomials \(\{ z^a w^b \}_{a<0,\; b<0}\) on both sides. Since
\[
\widetilde{\zeta}^\mathbb{C}(z-w) \cdot (z-w) = (z-w)^3 - (t_1^2 + t_1 t_2 + t_2^2)(z-w) - t_1 t_2 t,
\]
these two relations are equivalent, for all \(a,b \geq 0\), to
\begin{align}
    \label{equation_relation_Yangian_e_e_operators}
        \bigl[ e_{a+3},e_b \bigr] - 3\bigl[ e_{a+2},e_{b+1} \bigr] + 3\bigl[ e_{a+1},e_{b+2} \bigr] - \bigl[ e_a,e_{b+3} \bigr] \\
        - (t_1^2 + t_1 t_2 + t_2^2) \bigl( \bigl[ e_{a+1},e_b \bigr] - \bigl[ e_a,e_{b+1} \bigr] \bigr) + t_1 t_2 t \bigl( e_a e_b + e_b e_a \bigr) &= 0, \nonumber \\
        \nonumber \\
    \label{equation_relation_Yangian_f_f_operators}
        \bigl[ f_b,f_{a+3} \bigr] - 3\bigl[ f_{b+1},f_{a+2} \bigr] + 3\bigl[ f_{b+2},f_{a+1} \bigr] - \bigl[ f_{b+3},f_a \bigr] \\
        - (t_1^2 + t_1 t_2 + t_2^2) \bigl( \bigl[ f_b,f_{a+1} \bigr] - \bigl[ f_{b+1},f_a \bigr] \bigr) + t_1 t_2 t \bigl( f_b f_a + f_a f_b \bigr) &= 0. \nonumber
\end{align}

\begin{comment}
In (\ref{equation_relation_Yangian_h_e}) and (\ref{equation_relation_Yangian_h_f}), the rational function
\[
\frac{\zeta^\mathbb{C}(z-w)}{\zeta^\mathbb{C}(w-z)}
\]
is expanded in non-negative powers of \(w/z\), and the coefficients of all \(\{ z^a w^b \}_{a \leq 0,\; b < 0}\) are compared. 
\end{comment}

The expansion
\begin{equation}
    \label{equation_expansion_rational_function}
    \frac{\zeta^\mathbb{C}(z-w)}{\zeta^\mathbb{C}(w-z)} 
    = \frac{(z-w+t_1)(z-w+t_2)(z-w-t)}{(z-w-t_1)(z-w-t_2)(z-w+t)} 
    = 1 + \sum_{i=3}^\infty \sum_{j=0}^{i-2} t_1 t_2 \, \gamma_{ij} \frac{w^j}{z^i},
\end{equation}
where the \(\gamma_{ij}\) are polynomials in \(t = t_1 + t_2\) and \(t_1 t_2\), has coefficients divisible by \(t_1t_2\) beyond the leading term. Using \eqref{equation_expansion_rational_function}, relations (\ref{equation_relation_Yangian_h_e}) and (\ref{equation_relation_Yangian_h_f}) become
\begin{align}
    \label{equation_relation_Yangian_h_e_operators}
    h_ae_b &= e_bh_a + t_1t_2 \sum_{i=3}^\infty\sum_{j=0}^{i-2} \gamma_{ij} e_{b+j}h_{a-i}, \\
    \nonumber \\
    \label{equation_relation_Yangian_h_f_operators}
    f_bh_a &= h_af_b + t_1t_2 \sum_{i=3}^\infty\sum_{j=0}^{i-2} \gamma_{ij} h_{a-i}f_{b+j}
\end{align}
for all \(a,b \geq 0\), with the convention that \(h_{-1} = 1\) and \(h_{-2} = h_{-3} = \cdots = 0\).

Relation (\ref{equation_relation_Yangian_h_h}) simply states that $\bigl[ h_a, h_b \bigr] = 0$ for all \(a,b \geq 0\). Finally, equating coefficients in (\ref{equation_relation_Yangian_e_f}) yields
\begin{equation}
    \label{equation_relation_Yangian_e_f_operators}
    \bigl[ e_a, f_b \bigr] = - \frac{t_1 t_2}{t} \, h_{a+b}
\end{equation}
for all \(a,b \geq 0\).

Occasionally, one imposes additional cubic relations on the Yangian generators, analogous to the Drinfeld--Serre relations in finite types (see, e.g., \cite{tsymbaliuk_yangian}).

\begin{remark}
    \label{remark_Drinfeld-Serre}
    In some conventions, the generators are required to satisfy
    \[
    \sum_{\sigma \in S_3} \bigl[ e_{d_{\sigma(1)}}, \bigl[ e_{d_{\sigma(2)}}, e_{d_{\sigma(3)}+1} \bigr] \bigr]
    = \sum_{\sigma \in S_3} \bigl[ f_{d_{\sigma(1)}}, \bigl[ f_{d_{\sigma(2)}}, f_{d_{\sigma(3)}+1} \bigr] \bigr]
    = 0
    \]
    for all \(d_1,d_2,d_3 \geq 0\).
\end{remark}

Two adjustments to Definition \ref{definition_Yangian} are needed for the geometric interpretation.

\begin{remark}
    \label{remark_adjustments_Yangian_geometry}
    First, since the relations are symmetric in \(t_1\) and \(t_2\), we regard the Yangian as a \(\mathbb{Z}[t_1,t_2]^{\mathrm{sym}}\)-algebra.

    Second, we require that every commutator \(\bigl[x,y\bigr]\) be divisible by \(t_1 t_2\). While this holds for the commutators in (\ref{equation_relation_Yangian_h_e_operators}), (\ref{equation_relation_Yangian_h_f_operators}), and (\ref{equation_relation_Yangian_e_f_operators}), it is not guaranteed for \(\bigl[ e_a, e_b \bigr]\) and \(\bigl[ f_b, f_a \bigr]\) by (\ref{equation_relation_Yangian_e_e_operators}) and (\ref{equation_relation_Yangian_f_f_operators}) alone. To remedy this, we formally adjoin symbols
    \[
    \frac{\bigl[ e_{d_1}, \ldots, \bigl[ e_{d_{k-1}}, \bigl[ e_{d_k}, e_{d_{k+1}} \bigr] \bigr] \ldots \bigr]}{(t_1 t_2)^k}
    \quad\text{and}\quad
    \frac{\bigl[ f_{d_1}, \ldots, \bigl[ f_{d_{k-1}}, \bigl[ f_{d_k}, f_{d_{k+1}} \bigr] \bigr] \ldots \bigr]}{(t_1 t_2)^k}
    \]
    for all \(d_1,\ldots,d_{k+1} \geq 0\), subject to the Leibniz rule and the Jacobi identity.
\end{remark}

\subsection{Relations between operators}
\label{subsection_relations}

This section establishes the relations among the operators \(E_d\), \(F_d\), and \(H_d\) defined in (\ref{equation_operator_E}), (\ref{equation_operator_F}), and (\ref{equation_power_series_H}), respectively. For any two operators $x, y: H_{\mathcal{M}} \rightarrow H_{\mathcal{M} \times S}$, we define the colored compositions
\begin{align}
    \label{equation_composition_geometric_operators_1}
    \textcolor{red}{x}\textcolor{blue}{y} &: H_{\mathcal{M}} \xrightarrow{\textcolor{blue}{y}} H_{\mathcal{M} \times \textcolor{blue}{S}} \xrightarrow{\textcolor{red}{x} \boxtimes \operatorname{id}_{\textcolor{blue}{S}}} H_{\mathcal{M} \times \textcolor{red}{S} \times \textcolor{blue}{S}}, \\[4pt]
    \label{equation_composition_geometric_operators_2}
    \textcolor{blue}{y}\textcolor{red}{x} &: H_{\mathcal{M}} \xrightarrow{\textcolor{red}{x}} H_{\mathcal{M} \times \textcolor{red}{S}} \xrightarrow{\operatorname{id}_{\textcolor{red}{S}} \! \boxtimes \textcolor{blue}{y}} H_{\mathcal{M} \times \textcolor{red}{S} \times \textcolor{blue}{S}}.
\end{align}
In these compositions, the colors of the operators match the color of the corresponding factor of \(\textcolor{red}{S} \times \textcolor{blue}{S}\). We write $\bigl[\textcolor{red}{x},\textcolor{blue}{y}\bigr] = \textcolor{red}{x}\textcolor{blue}{y} - \textcolor{blue}{y}\textcolor{red}{x}$ and denote the diagonal embedding by \(\Delta: S \hookrightarrow \textcolor{red}{S} \times \textcolor{blue}{S}\).

The next proposition gives explicit formulas for certain commutators involving operators of type \(E\). It is the cohomological analogue of Theorem 1.4 of \cite{negut_hecke_correspondences}.

\begin{proposition}
    \label{proposition_commutator_cohomology_E}
    The following identities of operators \(H_{\mathcal{M}} \rightarrow H_{\mathcal{M} \times \textcolor{red}{S} \times \textcolor{blue}{S}}\) hold:
    \begin{align}
        \label{equation_relation_cohomology_E_E_operators}
        &\bigl[ \textcolor{red}{E_{(d_1, \ldots, d_l)}},\textcolor{blue}{E_k} \bigr] \\ \nonumber =\, &(\operatorname{id}_{\mathcal{M}} \times \Delta)_* \sum_{i=1}^l \left( \begin{cases}
            -\sum_{k \leq a \leq d_i-1} E_{(d_1, \ldots, d_{i-1}, a, d_i+k-1-a, d_{i+1}, \ldots, d_l)} \quad &\text{if } d_i > k \\
            0 \quad &\text{if } d_i = k \\
            \sum_{d_i \leq a \leq k-1} E_{(d_1, \ldots, d_{i-1}, a, d_i+k-1-a, d_{i+1}, \ldots, d_l)} \quad &\text{if } d_i < k
        \end{cases} \right)
    \end{align}
    for all \(d_1, \ldots, d_l, k \geq 0\).
\end{proposition}
\begin{proof}
    We adapt the strategy of Theorem 1.4 in \cite{negut_hecke_correspondences} to the cohomological setting. Beginning with the commutator \(\bigl[ \textcolor{red}{E_{(d_1,d_2)}},\textcolor{blue}{E_k} \bigr]\) for \(d_1,d_2,k \geq 0\), cohomological base change (justified by the dimension estimates of Section \ref{subsection_dimension_estimates_stable}) realizes the composition \(\textcolor{red}{E_{(d_1,d_2)}}\textcolor{blue}{E_k}\) via the diagram
\begin{equation*} 
\begin{tikzcd}[row sep=small]
	& {\ell_{01}^{d_1} \ell_{12}^{d_2} \ell_{23}^k} & \\
	& {\Big\{\mathcal{F}_0 \subset_{\textcolor{red}{u}} \mathcal{F}_1 \subset_{\textcolor{red}{u}} \mathcal{F}_2 \subset_{\textcolor{blue}{v}} \mathcal{F}_3 \Big\}} \\
	{\mathcal{M} \times \textcolor{red}{S} \times \textcolor{blue}{S}} && {\mathcal{M},}
	\arrow[dashed, no head, from=1-2, to=2-2]
	\arrow["{(p_+ \times p_{\textcolor{red}{S}} \times \operatorname{id}_{\textcolor{blue}{S}}) \circ (\pi_+ \times \operatorname{id}_{\textcolor{blue}{S}}) \circ (\pi_+ \times p_{\textcolor{blue}{S}})  }"', from=2-2, to=3-1]
	\arrow["{p_- \circ \pi_- \circ \pi_-}", from=2-2, to=3-3]
\end{tikzcd}\end{equation*}
    where \(\ell_{01} = c_1(\mathcal{L}_{0,1})\), \(\ell_{12} = c_1(\mathcal{L}_{1,2})\), and \(\ell_{23} = c_1(\mathcal{L}_{2,3})\). Similarly, \(\textcolor{blue}{E_k}\textcolor{red}{E_{(d_1,d_2)}}\) is induced by
\[\begin{tikzcd}[row sep=small]
	& {\ell_{01}'^k \ell_{12}'^{d_1} \ell_{23}^{d_2}} & \\
	& {\Big\{\mathcal{F}_0 \subset_{\textcolor{blue}{v}} \mathcal{F}_1' \subset_{\textcolor{red}{u}} \mathcal{F}_2 \subset_{\textcolor{red}{u}} \mathcal{F}_3 \Big\}} \\
	{\mathcal{M} \times \textcolor{red}{S} \times \textcolor{blue}{S}} && {\mathcal{M},}
	\arrow[dashed, no head, from=1-2, to=2-2]
	\arrow["{(p_+ \times \operatorname{id}_{\textcolor{red}{S}} \times  p_{\textcolor{blue}{S}}) \circ (\pi_+ \times p_{\textcolor{red}{S}}) \circ \pi_+}"', from=2-2, to=3-1]
	\arrow["{{p_- \circ \pi_- \circ \pi_-}}", from=2-2, to=3-3]
\end{tikzcd}\] 
    with \(\ell'_{01} = c_1(\mathcal{L}_{0,1}')\) and \(\ell'_{12} = c_1(\mathcal{L}_{1,2}')\). By Proposition \ref{proposition_birationality_uparrow_downarrow}, the compositions \(\textcolor{red}{E_{(d_1,d_2)}}\textcolor{blue}{E_k}\) and \(\textcolor{blue}{E_k}\textcolor{red}{E_{(d_1,d_2)}}\) are realized by the following diagrams
\[\begin{tikzcd}[row sep=small]
	& {\ell_{01}^{d_1}\ell_{12}^{d_2}\ell_{23}^k} &&&& {\ell_{01}'^k\ell_{12}'^{d_1}\ell_{23}^{d_2}} & \\
	& {\mathfrak{Y}_+} &&&& {\mathfrak{Y}_-} \\
	{\mathcal{M} \times \textcolor{red}{S} \times \textcolor{blue}{S}} && {\mathcal{M}} && {\mathcal{M} \times \textcolor{red}{S} \times \textcolor{blue}{S}} && {\mathcal{M}.}
	\arrow[dashed, no head, from=1-2, to=2-2]
	\arrow[dashed, no head, from=1-6, to=2-6]
	\arrow[from=2-2, to=3-1]
	\arrow[from=2-2, to=3-3]
	\arrow[from=2-6, to=3-5]
	\arrow[from=2-6, to=3-7]
\end{tikzcd}\]

    On \(\mathfrak{Y}_+\), we compare the classes \(\ell_{01}^{d_1} \ell_{12}^{d_2} \ell_{23}^{k}\) and \(\ell_{01}^{d_1} \ell_{12}'^{k} \ell_{23}'^{d_2}\). Let \(j\) be the inclusion of the vanishing locus
    \[
    Z\left( \mathcal{L}_{1,2} \rightarrow \mathcal{L}_{2,3}' \right) = Z\left( \mathcal{L}_{1,2}' \rightarrow \mathcal{L}_{2,3} \right) = \Big\{\mathcal{J}_2 = \mathcal{J}_2', \textcolor{red}{u} = \textcolor{blue}{v} \Big\},
    \]
    as in Proposition \ref{proposition_vanishing_locus_line_bundle_map}. A telescoping sum gives
    \begin{align*}
        \ell_{01}^{d_1} \ell_{12}^{d_2} \ell_{23}^{k} - \ell_{01}^{d_1} \ell_{12}'^{k} \ell_{23}'^{d_2}
        = \;&\ell_{01}^{d_1} \Big( \ell_{12}^{d_2}\ell_{23}^{k}-\ell_{12}^{d_2-1}\ell_{23}^{k}\ell_{23}'+\ldots+\ell_{12}\ell_{23}^k\ell_{23}'^{d_2-1}-\ell_{23}^k\ell_{23}'^{d_2} \\
        &+\ell_{23}^k\ell_{23}'^{d_2}-\ell_{23}^{k-1}\ell_{23}'^{d_2}\ell_{12}'+\ldots+\ell_{23}\ell_{23}'^{d_2}\ell_{12}'^{k-1}-\ell_{23}'^{d_2}\ell_{12}'^k \Big) \\
        = \;&j_*\Bigg( \ell_{01}^{d_1}\left( \begin{cases}
            -\sum_{k\leq a \leq d_2-1} \ell_{12}^a\ell_{23}^{d_2+k-1-a} \quad &\text{if } d_2 > k \\
            0 \quad &\text{if } d_2 = k \\
            \sum_{d_2\leq a \leq k-1} \ell_{12}^a\ell_{23}^{d_2+k-1-a} \quad &\text{if } d_2 < k
        \end{cases} \right) \Bigg),
    \end{align*}
    where the second equality follows from the projection formula. We have also used that the restrictions of \(\mathcal{L}_{1,2}\) and \(\mathcal{L}_{2,3}\) to the vanishing locus coincide with those of \(\mathcal{L}_{1,2}'\) and \(\mathcal{L}_{2,3}'\), respectively.

    On \(\mathfrak{Y}_-\), we compare \(\ell_{01}^{d_1}\ell_{12}^k\ell_{23}^{d_2}\) and \(\ell_{01}'^k\ell_{12}'^{d_1}\ell_{23}^{d_2}\). Let \(j'\) be the inclusion of the vanishing locus
    \[
    Z \left( \mathcal{L}_{0,1} \rightarrow \mathcal{L}_{1,2}' \right) = Z \left( \mathcal{L}_{0,1}' \rightarrow \mathcal{L}_{1,2} \right) = \Big\{ \mathcal{F}_1=\mathcal{F}_1', \textcolor{red}{u} = \textcolor{blue}{v} \Big\}.
    \]
    The same telescoping argument, replacing \(\ell_{01}\) by \(\ell_{12}'\) and \(\ell_{12}\) by \(\ell_{01}'\), yields
    \begin{equation*}
        \ell_{01}^{d_1}\ell_{12}^k\ell_{23}^{d_2}-\ell_{01}'^k\ell_{12}'^{d_1}\ell_{23}^{d_2}
        = j'_*\Bigg( \left( \begin{cases}
            -\sum_{k\leq a \leq d_1-1} \ell_{01}^a\ell_{12}^{d_1+k-1-a} \quad &\text{if } d_1 > k \\
            0 \quad &\text{if } d_1 = k \\
            \sum_{d_1\leq a \leq k-1} \ell_{01}^a\ell_{12}^{d_1+k-1-a} \quad &\text{if } d_1 < k
        \end{cases} \right) \ell_{23}^{d_2} \Bigg).
    \end{equation*}

    Hence,
    \begin{align*}
        \bigl[ \textcolor{red}{E_{(d_1,d_2)}},\textcolor{blue}{E_k} \bigr] = (\operatorname{id}_{\mathcal{M}} \times \Delta)_* \Bigg( &\begin{cases}
            -\sum_{k \leq a \leq d_2-1} E_{(d_1, a, d_2+k-1-a)} \quad &\text{if } d_2 > k \\
            0 \quad &\text{if } d_2 = k \\
            \sum_{d_2 \leq a \leq k-1} E_{(d_1, a, d_2+k-1-a)} \quad &\text{if } d_2 < k
        \end{cases} \\
        + &\begin{cases}
            -\sum_{k \leq a \leq d_1-1} E_{(a, d_1+k-1-a, d_2)} \quad &\text{if } d_1 > k \\
            0 \quad &\text{if } d_1 = k \\
            \sum_{d_1 \leq a \leq k-1} E_{(a, d_1+k-1-a, d_2)} \quad &\text{if } d_1 < k
        \end{cases} \Bigg).
    \end{align*}
    The commutator \(\bigl[ \textcolor{red}{E_d},\textcolor{blue}{E_k} \bigr]\) follows similarly; we omit the details.

    For \(l > 2\), the exchanges of \(d_1\) and \(d_l\) with \(k\) are handled by \(\mathfrak{Y}_+\) and \(\mathfrak{Y}_-\), respectively. The intermediate indices \(2 \leq i \leq l-1\) require a separate argument, which we illustrate for \(l=3\).

    By means of the same techniques as above, we realize \(\textcolor{red}{E_{(d_1,d_2,d_3)}}\textcolor{blue}{E_k}\) as
    \[
    (\operatorname{id}_{\mathcal{M}} \times \Delta)_* \left(\begin{cases}
        -\sum_{k \leq a \leq d_3-1} E_{(d_1,d_2,a,d_3+k-1-a)} \quad &\text{if } d_3 > k \\
        0 \quad &\text{if } d_3 = k \\
        \sum_{d_3 \leq a \leq k-1} E_{(d_1,d_2,a,d_3+k-1-a)} \quad &\text{if } d_3 < k
    \end{cases}  \right)
    \]
    plus the contribution
    \begin{equation}
        \label{equation_proof_E_E_4}
        \begin{tikzcd}[row sep=small]
	& {\ell_{01}^{d_1} \ell_{12}^{d_2} \ell_{23}^k \ell_{34}^{d_3}} & \\
	& {\mathfrak{Y}_{-+}} \\
	{\mathcal{M} \times \textcolor{red}{S} \times \textcolor{blue}{S}} && {\mathcal{M}.}
	\arrow[dashed, no head, from=1-2, to=2-2]
	\arrow["{(p_+ \times p_{\textcolor{red}{S}} \times \operatorname{id}_{\textcolor{blue}{S}}) \circ (\pi_+ \times \operatorname{id}_{\textcolor{blue}{S}}) \circ (\pi_+ \times p_{\textcolor{blue}{S}}) \circ \pi_+ \circ \pi^{\uparrow}}"', from=2-2, to=3-1]
	\arrow["{p_- \times \pi_- \times \pi_- \times \pi_- \times \pi^{\uparrow}}", from=2-2, to=3-3]
\end{tikzcd}
    \end{equation}
    The operator induced by (\ref{equation_proof_E_E_4}) equals
    \[
    (\operatorname{id}_{\mathcal{M}} \times \Delta)_* \left(\begin{cases}
        -\sum_{k \leq a \leq d_2-1} E_{(d_1,a,d_2+k-1-a,d_3)} \quad &\text{if } d_2 > k \\
        0 \quad &\text{if } d_2 = k \\
        \sum_{d_2 \leq a \leq k-1} E_{(d_1,a,d_2+k-1-a,d_3)} \quad &\text{if } d_2 < k
    \end{cases}  \right)
    \]
    plus
    \begin{equation}
        \label{equation_proof_E_E_5}
        \begin{tikzcd}[sep=small]
	{\ell_{01}^{d_1} \ell_{12}^k \ell_{23}^{d_2}} &&& \\
	{\mathfrak{Y}_-} && {\Big\{ \mathcal{F}_2 \subset_{\textcolor{red}{u}} \mathcal{F}_3 \subset_{\textcolor{red}{u}} \mathcal{F}_4 \Big\}} & {\ell^{d_3}} \\
	& {\Big\{ \mathcal{F}_2 \subset_{\textcolor{red}{u}} \mathcal{F}_3 \Big\}} && {\Big\{ \mathcal{F}_3 \subset_{\textcolor{red}{u}} \mathcal{F}_4 \Big\}} \\
	{\mathcal{M} \times \textcolor{red}{S} \times \textcolor{blue}{S}} &&& {\mathcal{M}.}
	\arrow[dashed, no head, from=1-1, to=2-1]
	\arrow["{{\pi_- \circ \pi_- \circ \pi^{\uparrow}}}", from=2-1, to=3-2]
	\arrow["{{(p_+ \times p_{\textcolor{red}{S}} \times \operatorname{id}_{\textcolor{blue}{S}}) \circ (\pi_+ \times p_{\textcolor{blue}{S}}) \circ \pi_+ \circ \pi^{\uparrow}}}"{description, pos=0.8}, from=2-1, to=4-1]
	\arrow["{{\pi_+}}", from=2-3, to=3-2]
	\arrow["{{\pi_-}}"', from=2-3, to=3-4]
	\arrow[dashed, no head, from=2-4, to=3-4]
	\arrow["{{p_-}}", from=3-4, to=4-4]
\end{tikzcd}
    \end{equation}
    Finally, the operator from (\ref{equation_proof_E_E_5}) coincides with
    \[
    (\operatorname{id}_{\mathcal{M}} \times \Delta)_* \left(\begin{cases}
        -\sum_{k \leq a \leq d_1-1} E_{(a,d_1+k-1-a,d_2,d_3)} \quad &\text{if } d_1 > k \\
        0 \quad &\text{if } d_1 = k \\
        \sum_{d_1 \leq a \leq k-1} E_{(a,d_1+k-1-a,d_2,d_3)} \quad &\text{if } d_1 < k
    \end{cases}  \right) + \textcolor{blue}{E_k}\textcolor{red}{E_{(d_1,d_2,d_3)}}.
    \]
    Combining these expressions gives
    \begin{align*}
        \bigl[ \textcolor{red}{E_{(d_1,d_2,d_3)}},\textcolor{blue}{E_k} \bigr] = 
        \;(\operatorname{id}_{\mathcal{M}} \times \Delta)_* \Bigg(&\begin{cases}
        -\sum_{k \leq a \leq d_3-1} E_{(d_1,d_2,a,d_3+k-1-a)} \quad &\text{if } d_3 > k \\
        0 \quad &\text{if } d_3 = k \\
        \sum_{d_3 \leq a \leq k-1} E_{(d_1,d_2,a,d_3+k-1-a)} \quad &\text{if } d_3 < k
        \end{cases} \\
        + &\begin{cases}
        -\sum_{k \leq a \leq d_2-1} E_{(d_1,a,d_2+k-1-a,d_3)} \quad &\text{if } d_2 > k \\
        0 \quad &\text{if } d_2 = k \\
        \sum_{d_2 \leq a \leq k-1} E_{(d_1,a,d_2+k-1-a,d_3)} \quad &\text{if } d_2 < k
        \end{cases} \\
        + &\begin{cases}
        -\sum_{k \leq a \leq d_1-1} E_{(a,d_1+k-1-a,d_2,d_3)} \quad &\text{if } d_1 > k \\
        0 \quad &\text{if } d_1 = k \\
        \sum_{d_1 \leq a \leq k-1} E_{(a,d_1+k-1-a,d_2,d_3)} \quad &\text{if } d_1 < k
        \end{cases}  \Bigg),
    \end{align*}
    which completes the proof.    
\end{proof}

Consequently, the commutator \(\left[ \textcolor{red}{E_d},\textcolor{blue}{E_{d'}} \right]\) is expressible in terms of the operators \(E_{(d_1,d_2)}\). Our goal, however, is to express it uniquely in terms of \(E_d\), \(F_d\), and \(H_d\). The following technical lemma is a key step.

\begin{lemma}
    \label{lemma_excess_intersection_formula_E}
    The following identities of operators \(H_{\mathcal{M}} \rightarrow H_{\mathcal{M} \times S}\) hold:
    \begin{equation}
        \label{equation_excess_intersection_formula_E}
        E_dE_{d'}|_\Delta = E_{(d+1,d')} - E_{(d,d'+1)} - c_1(\omega_S)E_{(d,d')}
    \end{equation}
    for all \(d, d' \geq 0\), where \(\omega_S\) denotes the canonical bundle of \(S\) and its various pull-backs.
\end{lemma}
\begin{proof}
    The operator \(E_dE_{d'}|_\Delta\) is induced by
\[\begin{tikzcd}[row sep=small]
	& {\ell^d} && {\ell^{d'}} & \\
	& {\Big\{ \mathcal{F}_0 \subset_u \mathcal{F}_1 \Big\}} && {\Big\{ \mathcal{F}_1 \subset_u \mathcal{F}_2 \Big\}} \\
	{\mathcal{M} \times S} && {\mathcal{M} \times S} && {\mathcal{M}.}
	\arrow[dashed, no head, from=1-2, to=2-2]
	\arrow[dashed, no head, from=1-4, to=2-4]
	\arrow["{p_+ \times p_S}"', from=2-2, to=3-1]
	\arrow["{p_- \times p_S}", from=2-2, to=3-3]
	\arrow["{p_+ \times p_S}"', from=2-4, to=3-3]
	\arrow["{p_-}", from=2-4, to=3-5]
\end{tikzcd}\]
    By Proposition 2.27 of \cite{negut_hecke_correspondences} and the excess intersection formula (see Theorem \ref{theorem_excess_intersection_formula}), this operator is realized by
    \begin{equation*}
        \begin{tikzcd}[row sep=small]
	& {\ell_{01}^{d} \ell_{12}^{d'} \left( \ell_{01}-\ell_{12}-c_1(\omega_S) \right)} & \\
	& {\Big\{ \mathcal{F}_0 \subset_u \mathcal{F}_1 \subset_u \mathcal{F}_2 \Big\}} \\
	{\mathcal{M} \times S} && {\mathcal{M},}
	\arrow[dashed, no head, from=1-2, to=2-2]
	\arrow["{(p_+ \times p_S) \circ \pi_+}"', from=2-2, to=3-1]
	\arrow["{p_- \circ \pi_-}", from=2-2, to=3-3]
\end{tikzcd}
    \end{equation*}
    which induces \(E_{(d+1,d')} - E_{(d,d'+1)} - c_1(\omega_S)E_{(d,d')}\), as claimed.    
\end{proof}

Combining Proposition \ref{proposition_commutator_cohomology_E} and Lemma \ref{lemma_excess_intersection_formula_E} yields a quadratic expression for the commutator \(\bigl[ \textcolor{red}{E_d},\textcolor{blue}{E_{d'}} \bigr]\) in terms of the operators \(E_d\).

\begin{proposition}
    \label{proposition_commutator_cohomology_E_quadratic}
    The following identities of operators \(H_{\mathcal{M}} \rightarrow H_{\mathcal{M} \times \textcolor{red}{S} \times \textcolor{blue}{S}}\) hold:
    \begin{align}
        \label{equation_relation_cohomology_E_E_operators_quadratic}
        \bigl[ \textcolor{red}{E_{a+3}},\textcolor{blue}{E_b} \bigr] - 3\bigl[ \textcolor{red}{E_{a+2}},\textcolor{blue}{E_{b+1}} \bigr] + 3\bigl[ \textcolor{red}{E_{a+1}},\textcolor{blue}{E_{b+2}} \bigr] - \bigl[ \textcolor{red}{E_a},\textcolor{blue}{E_{b+3}} \bigr] \\
        +(\textcolor{red}{t}-\textcolor{blue}{t})\left( \bigl[ \textcolor{red}{E_{a+2}},\textcolor{blue}{E_b} \bigr] - 2\bigl[ \textcolor{red}{E_{a+1}},\textcolor{blue}{E_{b+1}} \bigr] + \bigl[ \textcolor{red}{E_a},\textcolor{blue}{E_{b+2}} \bigr] \right) - \textcolor{red}{t}\textcolor{blue}{t}\left( \bigl[ \textcolor{red}{E_{a+1}},\textcolor{blue}{E_b} \bigr] - \bigl[ \textcolor{red}{E_a},\textcolor{blue}{E_{b+1}} \bigr] \right) \nonumber \\
        +(\operatorname{id}_{\mathcal{M}} \times \Delta)_* \left( E_{a+1}E_b|_\Delta - E_bE_{a+1}|_\Delta - E_aE_{b+1}|_\Delta + E_{b+1}E_a|_\Delta + tE_aE_b|_\Delta + tE_bE_a|_\Delta \right) = 0 \nonumber
    \end{align}
    for all \(a,b \geq 0\). Here \(\textcolor{red}{t}\) and \(\textcolor{blue}{t}\) denote the pull-backs of \(t = c_1(\omega_S)\) to \(H_{\textcolor{red}{S} \times \textcolor{blue}{S}}\) via the first and second projections, respectively. The symbol \(t\) in the last line may be taken as either \(\textcolor{red}{t}\) or \(\textcolor{blue}{t}\) due to the presence of \((\operatorname{id}_{\mathcal{M}} \times \Delta)_*\).
\end{proposition}

Relation (\ref{equation_relation_cohomology_E_E_operators_quadratic}) is the cohomological analogue of the Yangian relation (\ref{equation_relation_Yangian_e_e_operators}). As observed in Definition \ref{definition_Yangian}, the defining relations are best expressed in terms of generating series. The same holds for the geometric operators \(E_d, F_d, H_d\). To translate (\ref{equation_relation_cohomology_E_E_operators_quadratic}) into a power-series identity, we introduce the cohomological analogues of (\ref{equation_zeta_Yangian}):
\begin{equation}
    \label{equation_zeta_cohomology}
    \zeta^\text{coh}(x) \define 1 + \frac{[\Delta]}{x(x+t)} \in H_{\textcolor{red}{S} \times \textcolor{blue}{S}}(x),
\end{equation}
where \([\Delta] \in H_{\textcolor{red}{S} \times \textcolor{blue}{S}}\) is the fundamental class of the diagonal, and
\begin{equation}
    \label{equation_zeta_tilde_cohomology}
    \widetilde{\zeta}^\text{coh}_\pm \define \zeta(x)(x \pm \textcolor{red}{t})(x \mp \textcolor{blue}{t}) \in \frac{H_{\textcolor{red}{S} \times \textcolor{blue}{S}}[x]}{x},
\end{equation}
 with \(\textcolor{red}{t}, \textcolor{blue}{t}\) as above. Again, \(t\) in the denominator of (\ref{equation_zeta_cohomology}) may be treated as either \(\textcolor{red}{t}\) or \(\textcolor{blue}{t}\) due to \([\Delta]\).

As in the derivation of (\ref{equation_relation_Yangian_e_e_operators}) and (\ref{equation_relation_Yangian_f_f_operators}), observe that
\[
\widetilde{\zeta}^\text{coh}_\pm(z-w) \cdot (z-w) = (z-w)^3+(\textcolor{red}{t}-\textcolor{blue}{t})(z-w)^2 - \textcolor{red}{t}\textcolor{blue}{t}(z-w) + [\Delta](z-w) \mp [\Delta]t.
\]
Thus, relation (\ref{equation_relation_cohomology_E_E_operators_quadratic}) for all \(a,b \geq 0\) is equivalent to
\begin{equation}
    \label{equation_relation_cohomology_E_E}
        \Bigg[ \textcolor{red}{E(z)}\textcolor{blue}{E(w)}\widetilde{\zeta}^\text{coh}_-(w-z) \Bigg]_{z^{<0},w^{<0}} = \Bigg[ \textcolor{blue}{E(w)}\textcolor{red}{E(z)}\widetilde{\zeta}^\text{coh}_+(z-w) \Bigg]_{z^{<0},w^{<0}},
\end{equation}
where the square brackets have the same meaning as in Definition \ref{definition_Yangian}.

The same techniques yield a quadratic expression for \(\bigl[ \textcolor{blue}{F_d},\textcolor{red}{F_{d'}} \bigr]\).

\begin{proposition}
    \label{proposition_commutator_cohomology_F_quadratic}
    The following identities of operators \(H_{\mathcal{M}} \rightarrow H_{\mathcal{M} \times \textcolor{red}{S} \times \textcolor{blue}{S}}\) hold:
    \begin{align}
        \label{equation_relation_cohomology_F_F_operators_quadratic}
        \bigl[ \textcolor{blue}{F_b},\textcolor{red}{F_{a+3}} \bigr] - 3\bigl[ \textcolor{blue}{F_{b+1}},\textcolor{red}{F_{a+2}} \bigr] + 3\bigl[ \textcolor{blue}{F_{b+2}},\textcolor{red}{F_{a+1}} \bigr] - \bigl[ \textcolor{blue}{F_{b+3}},\textcolor{red}{F_a} \bigr] \\
        +(\textcolor{red}{t}-\textcolor{blue}{t})\left( \bigl[ \textcolor{blue}{F_b},\textcolor{red}{F_{a+2}} \bigr] - 2\bigl[ \textcolor{blue}{F_{b+1}},\textcolor{red}{F_{a+1}} \bigr] + \bigl[ \textcolor{blue}{F_{b+2}},\textcolor{red}{F_a} \bigr] \right) - \textcolor{red}{t}\textcolor{blue}{t}\left( \bigl[ \textcolor{blue}{F_b},\textcolor{red}{F_{a+1}} \bigr] - \bigl[ \textcolor{blue}{F_{b+1}},\textcolor{red}{F_a} \bigr] \right) \nonumber \\
        +(\operatorname{id}_{\mathcal{M}} \times \Delta)_* \left( F_bF_{a+1}|_\Delta -F_{a+1}F_b|_\Delta -F_{b+1}F_a|_\Delta + F_aF_{b+1}|_\Delta +tF_bF_a|_\Delta +tF_aF_b|_\Delta \right) = 0 \nonumber
    \end{align}
    for all \(a,b \geq 0\).
\end{proposition}

For all \(a,b \geq 0\), relation (\ref{equation_relation_cohomology_F_F_operators_quadratic}) is equivalent to
\begin{equation}
    \label{equation_relation_cohomology_F_F}
    \Bigg[ \textcolor{blue}{F(w)}\textcolor{red}{F(z)}\widetilde{\zeta}^\text{coh}_-(w-z) \Bigg]_{z^{<0},w^{<0}} = \Bigg[ \textcolor{red}{F(z)}\textcolor{blue}{F(w)}\widetilde{\zeta}^\text{coh}_+(z-w) \Bigg]_{z^{<0},w^{<0}}.   
\end{equation}

\begin{comment}
We record the following analogue of Proposition \ref{proposition_commutator_cohomology_E} for completeness.

\begin{proposition}
    \label{proposition_commutator_cohomology_F}
    The following identities of operators \(H_{\mathcal{M}} \rightarrow H_{\mathcal{M} \times \textcolor{red}{S} \times \textcolor{blue}{S}}\) hold:
    \begin{equation}
        \label{equation_relation_cohomology_F_F_operators}
        \bigl[ \textcolor{blue}{F_{d}},\textcolor{red}{F_{d'}} \bigr] = (\operatorname{id}_{\mathcal{M}} \times \Delta)_* \begin{cases}
            \sum_{d' \leq a \leq d-1} F_{(a, d+d'-1-a)} \quad &\text{if } d > d' \\
            0 \quad 
            &\text{if } d=d' \\
            -\sum_{d \leq a \leq d'-1} F_{(a, d+d'-1-a)} \quad &\text{if } d < d'
        \end{cases}
    \end{equation}
    for all \(d, d' \geq 0\).
\end{proposition}
\end{comment}

We now derive the $H$-$E$ relation.

\begin{proposition}
    \label{proposition_commutator_cohomology_E_H_series}
    The following identity of operators \(H_{\mathcal{M}} \rightarrow H_{\mathcal{M} \times \textcolor{red}{S} \times \textcolor{blue}{S}}\) holds:
    \begin{equation}
        \label{equation_relation_cohomology_H_E}
        \textcolor{red}{H(z)}\textcolor{blue}{E(w)} = \Bigg[ \textcolor{blue}{E(w)}\textcolor{red}{H(z)} \frac{\zeta^\text{coh}(z-w)}{\zeta^\text{coh}(w-z)} \Bigg]_{z \gg w,z^{\leq 0},w^{<0}}.
    \end{equation}
\end{proposition}
\begin{proof}
    A cohomological base change (justified by the dimension estimates of Section \ref{subsection_dimension_estimates_stable}) shows that \(\textcolor{red}{H(z)}\textcolor{blue}{E(w)}\) corresponds to
    \[
    \frac{1}{w-\ell} \cdot \frac{c(\mathcal{U}_0,z+t)}{c(\mathcal{U}_0,z)}
    \]
    on \(\{\mathcal{J}_0 \subset_{\textcolor{blue}{u}} \mathcal{J}_1\} \times \textcolor{red}{S}\), while \(\textcolor{blue}{E(w)}\textcolor{red}{H(z)}\) is given by
    \[
    \frac{1}{w-\ell} \cdot \frac{c(\mathcal{U}_1,z+t)}{c(\mathcal{U}_1,z)}.
    \]
    The short exact sequence
    \begin{equation}
    \label{equation_short_exact_sequence_universal_sheaves_2}
        0 \rightarrow \mathcal{U}_0 \rightarrow \mathcal{U}_1 \rightarrow \mathcal{L} \otimes \mathcal{O}_\Delta \rightarrow 0
    \end{equation}
    on \(\{\mathcal{J}_0 \subset_{\textcolor{blue}{u}} \mathcal{J}_1\} \times \textcolor{red}{S}\) relates the universal sheaves \(\mathcal{U}_0\) and \(\mathcal{U}_1\). In (\ref{equation_short_exact_sequence_universal_sheaves_2}), we abuse notation slightly: the sheaf \(\mathcal{O}_\Gamma\) from (\ref{equation_short_exact_sequence_universal_sheaves_1}) is the pull-back of the diagonal structure sheaf \(\mathcal{O}_\Delta\) in \(\textcolor{red}{S} \times \textcolor{blue}{S}\), so we write \(\mathcal{O}_\Delta\). We also denote the tautological line bundle \(\mathcal{L}\) and its pull-back by the same symbol. For further details, see the proof of Proposition 3.2 in \cite{negut_shuffle_surfaces}.

    The short exact sequence \eqref{equation_short_exact_sequence_universal_sheaves_2} yields
    \begin{equation}
        \label{equation_proof_H(z)_E(w)_2}
        \textcolor{red}{H(z)}\textcolor{blue}{E(w)} = \textcolor{blue}{E(w)}\textcolor{red}{H(z)} \cdot \frac{c(-\mathcal{L} \otimes \mathcal{O}_\Delta,z+t)}{c(-\mathcal{L} \otimes \mathcal{O}_\Delta,z)}.
    \end{equation}
    Since \(c(-\mathcal{L} \otimes \mathcal{O}_\Delta,z) = c(-\mathcal{O}_\Delta,z-\ell)\), we compute \(c(-\mathcal{O}_\Delta,z)\). Following the argument of Proposition 5.24 in \cite{negut_shuffle_surfaces}, assume the diagonal is cut out by a regular section \(\sigma: \mathcal{O}_{\textcolor{red}{S} \times \textcolor{blue}{S}} \rightarrow \mathcal{E}\) of a rank \(2\) locally free sheaf \(\mathcal{E}\) (this assumption can be removed as in \textit{loc.~cit.}). If \(e_1, e_2\) are the Chern roots of \(\mathcal{E}\), then
    \[
    c_1(\mathcal{E}) = c_1(\wedge^2\mathcal{E}) = e_1+e_2,\quad c_2(\mathcal{E}) = c_2(\mathcal{E}^\vee) = e_1e_2,\quad c_1(\mathcal{E}^\vee) = c_1(\wedge^2\mathcal{E}^\vee) = -e_1-e_2.
    \]
    The normal bundle of the diagonal is \(\mathcal{N}_{\Delta / \textcolor{red}{S} \times \textcolor{blue}{S}} \cong \mathcal{E}|_\Delta\), and the Koszul complex gives
    \[
    [\mathcal{O}_\Delta] = [\mathcal{O}_{\textcolor{red}{S} \times \textcolor{blue}{S}}] - [\mathcal{E^\vee}] + [\wedge^2\mathcal{E}^\vee]
    \]
    in \(K\)-theory. Since \(\sigma\) is regular, \([\Delta] = \Delta_*1 = c_{\text{top}}(\mathcal{E}) = e_1e_2\) (see Lemma \ref{lemma_fundamental_class_zero_locus}). Hence,
    \begin{align*}
        c(-\mathcal{O}_\Delta,z) &= \frac{c(\mathcal{E}^\vee,z)}{c(\mathcal{O}_{\textcolor{red}{S}\times \textcolor{blue}{S}})c(\wedge^2\mathcal{E}^\vee,z)} \\
        &= \frac{(z+e_1)(z+e_2)}{z(z+e_1+e_2)} \\
        &= 1 + \frac{[\Delta]}{z(z-t)},
    \end{align*}
    where the last equality uses \(\mathcal{N}_{\Delta / \textcolor{red}{S} \times \textcolor{blue}{S}} \cong \mathcal{T}_S\), so that \(e_1+e_2 = -t\) after identifying Chern classes of \(\mathcal{T}_S\) with those of \(\Omega^1_S\).

    Now, on the left-hand side of \eqref{equation_relation_cohomology_H_E}, \(z\) has non-positive powers and \(w\) has negative powers. In \cite{negut_rationaltrigonometricellipticalgebras}, Neguţ establishes
    \begin{equation}
        \label{equation_proof_H(z)_E(w)_3}
        \frac{\zeta^\text{coh}(z-w)}{\zeta^\text{coh}(w-z)} = 1 + \sum_{i=3}^\infty \sum_{j=0}^{i-2} \Delta_*(\gamma_{ij}) \frac{w^j}{z^i},
    \end{equation}
    where \(\gamma_{ij}\) is the same polynomial in \(t_1+t_2 = c_1(\Omega^1_S)\) and \(t_1t_2 = c_2(\Omega^1_S)\) as in (\ref{equation_expansion_rational_function}).

    It remains to compare
    \begin{equation}
        \label{equation_proof_H(z)_E(w)_4}
        \frac{1}{w-\ell} \cdot \frac{1+\frac{[\Delta]}{(z-\ell)(z-\ell+t)}}{1+\frac{[\Delta]}{(z-\ell)(z-\ell-t)}} = \left( \sum_{k=0}^\infty \frac{\ell^k}{w^{k+1}} \right) \cdot \left( 1 + \sum_{i=3}^\infty \sum_{j=0}^{i-2} \Delta_*(\gamma_{ij}) \frac{\ell^j}{z^i} \right)
    \end{equation}
    and
    \begin{equation}
        \label{equation_proof_H(z)_E(w)_5}
        \frac{1}{w-\ell} \cdot \frac{1 + \frac{[\Delta]}{(z-w)(z-w+t)}}{1 + \frac{[\Delta]}{(z-w)(z-w-t)}} = \left( \sum_{k=0}^\infty \frac{\ell^k}{w^{k+1}} \right) \cdot \left(1 + \sum_{i=3}^\infty \sum_{j=0}^{i-2} \Delta_*(\gamma_{ij}) \frac{w^j}{z^i}\right).
    \end{equation}
    The coefficients of \(z^i w^j\) for \(i \leq 0\) and \(j < 0\) in \eqref{equation_proof_H(z)_E(w)_4} and \eqref{equation_proof_H(z)_E(w)_5} agree, proving the proposition. Note that in \eqref{equation_proof_H(z)_E(w)_4}, \(w\) appears only with non-positive powers, while in \eqref{equation_proof_H(z)_E(w)_5} it appears with positive powers. This explains the bracket notation in \eqref{equation_relation_cohomology_H_E}.
\end{proof}

Using (\ref{equation_proof_H(z)_E(w)_3}), relation (\ref{equation_relation_cohomology_H_E}) is equivalent, for all \(a,b \geq 0\), to
\begin{equation*}
    \textcolor{red}{H_a}\textcolor{blue}{E_b} = \textcolor{blue}{E_b}\textcolor{red}{H_a} + (\operatorname{id}_{\mathcal{M}} \times \Delta)_* \left( \sum_{i=3}^\infty \sum_{j=0}^{i-2} \gamma_{ij} \textcolor{blue}{E_{b+j}}\textcolor{red}{H_{a-i}}|_\Delta \right),
\end{equation*}
where \(\textcolor{blue}{E_{b+j}}\textcolor{red}{H_{a-i}}|_\Delta\) is an operator
\[
H_{\mathcal{M}} \rightarrow H_{\mathcal{M} \times \textcolor{red}{S} \times \textcolor{blue}{S}} \xrightarrow{|_\Delta} H_{\mathcal{M} \times S},
\]
and we adopt the convention \(H_{-1} = 1\), \(H_{-2} = H_{-3} = \cdots = 0\). This is structurally identical to (\ref{equation_relation_Yangian_h_e_operators}).

The same argument gives the \(F\)-\(H\) relation.

\begin{proposition}
    \label{proposition_commutator_cohomology_F_H_series}
    The following identity of operators \(H_{\mathcal{M}} \rightarrow H_{\mathcal{M} \times \textcolor{red}{S} \times \textcolor{blue}{S}}\) holds:
    \begin{equation}
        \label{equation_relation_cohomology_H_F}
        \textcolor{blue}{F(w)}\textcolor{red}{H(z)} = \Bigg[ \textcolor{red}{H(z)}\textcolor{blue}{F(w)} \frac{\zeta^\text{coh}(z-w)}{\zeta^\text{coh}(w-z)} \Bigg]_{z \gg w,z^{\leq 0},w^{<0}}.
    \end{equation}
\end{proposition}

Equivalently, for all \(a,b \geq 0\),
\begin{equation*}
    \textcolor{blue}{F_b}\textcolor{red}{H_a} = \textcolor{red}{H_a}\textcolor{blue}{F_b} + (\operatorname{id}_{\mathcal{M}} \times \Delta)_* \left( \sum_{i=3}^\infty \sum_{j=0}^{i-2} \gamma_{ij} \textcolor{red}{H_{a-i}}\textcolor{blue}{F_{b+j}}|_\Delta \right),
\end{equation*}
with the same conventions as above.

Before discussing \(\bigl[ \textcolor{red}{E(z)},\textcolor{blue}{F(w)} \bigr]\), we record the commutation of the \(H\)-series.

\begin{proposition}
    \label{proposition_commutator_cohomology_H_H_series}
    The following identity of operators \(H_{\mathcal{M}} \rightarrow H_{\mathcal{M} \times \textcolor{red}{S} \times \textcolor{blue}{S}}\) holds:
    \begin{equation}
        \label{equation_relation_cohomology_H_H}
        \bigl[ \textcolor{red}{H(z)},\textcolor{blue}{H(w)} \bigr] = 0.
    \end{equation}
\end{proposition}
\begin{proof}
    This follows immediately from \eqref{equation_power_series_H}, as cup product operators commute.
\end{proof}

Comparing coefficients yields $\bigl[ \textcolor{red}{H_a},\textcolor{blue}{H_b} \bigr] = 0$ for all \(a,b \geq 0\). Finally, we state the commutator \(\bigl[ \textcolor{red}{E_a},\textcolor{blue}{F_b} \bigr]\). Its proof is deferred to Section \ref{section_commutator_E_F}.

\begin{proposition}
\label{proposition_commutator_cohomology_E_F_operators_stable}
    The following identity of operators \(H_{\mathcal{M}} \rightarrow H_{\mathcal{M} \times \textcolor{red}{S} \times \textcolor{blue}{S}}\) holds:
    \begin{equation}
        \label{equation_relation_cohomology_E_F_operators_stable}
        \bigl[ \textcolor{red}{E_a},\textcolor{blue}{F_b} \bigr] = (-1)^r(\operatorname{id}_{\mathcal{M}} \times \Delta)_* \left( \frac{H_{a+b}}{t} \right)
    \end{equation}
    for all \(a,b \geq 0\).
\end{proposition}

Passing to generating series yields
\begin{equation}
    \label{equation_relation_cohomology_E_F}
    \bigl[ \textcolor{red}{E(z)},\textcolor{blue}{F(w)} \bigr] = \frac{1}{t} (\operatorname{id}_{\mathcal{M}} \times \Delta)_* \left( \frac{H(z)-H(w)}{z-w} \right),
\end{equation}
which is equivalent to (\ref{equation_relation_cohomology_E_F_operators_stable}).

Theorem \ref{theorem_relations_geometric_operators_stable} is the main result of this section. In \cite{negut_rationaltrigonometricellipticalgebras}, Neguţ stated it as a consequence of the \(K\)-theoretic results in \cite{negut_shuffle_surfaces} via the Chern character isomorphism. Our proof avoids this detour.

\begin{proof}[Proof of Theorem \ref{theorem_relations_geometric_operators_stable}]
    This is immediate from \eqref{equation_relation_cohomology_E_E}, \eqref{equation_relation_cohomology_F_F}, \eqref{equation_relation_cohomology_H_E}, \eqref{equation_relation_cohomology_H_F}, \eqref{equation_relation_cohomology_H_H}, and \eqref{equation_relation_cohomology_E_F}.
\end{proof}

\subsection{The action on cohomology}
\label{subsection_action}

We now define an action of the affine Yangian of $\widehat{\mathfrak{gl}}_1$ on $H_{\mathcal{M}}$, following \cite{negut_rationaltrigonometricellipticalgebras}. For any two operators $x,y \in \{E_d, F_d, H_d\}_{d \geq 0}$, the commutator $\bigl[ \textcolor{red}{x},\textcolor{blue}{y} \bigr]$ vanishes away from the diagonal of $\textcolor{red}{S} \times \textcolor{blue}{S}$. By the excision theorem in cohomology, we obtain
\[
\bigl[ \textcolor{red}{x},\textcolor{blue}{y} \bigr] = (\operatorname{id}_{\mathcal{M}} \times \Delta)_* z
\]
for some operator $z: H_{\mathcal{M}} \rightarrow H_{\mathcal{M} \times S}$, where $\operatorname{id}_{\mathcal{M}} \times \Delta: \mathcal{M} \times S \hookrightarrow \mathcal{M} \times \textcolor{red}{S} \times \textcolor{blue}{S}$ is the diagonal embedding. The map $(\operatorname{id}_{\mathcal{M}} \times \Delta)_*$ is injective because $p_{\mathcal{M}} \times p_{\textcolor{red}{S}}$ provides a left inverse to $\operatorname{id}_{\mathcal{M}} \times \Delta$. Hence $z$ is uniquely determined; we set
\[
\bigl[ \textcolor{red}{x},\textcolor{blue}{y} \bigr]_{\mathrm{red}} \define z,
\]
and say that $x$ and $y$ have \textit{diagonal commutator}.

Consider the ring homomorphism
\[
\phi: \mathbb{Z}[t_1,t_2]^{\mathrm{sym}} \rightarrow H_S
\]
sending $t_1 + t_2$ to $c_1(\Omega^1_S)$ and $t_1 t_2$ to $c_2(\Omega^1_S)$, and let $p_{\mathcal{M}}: \mathcal{M} \times S \rightarrow \mathcal{M}$ and $p_S: \mathcal{M} \times S \rightarrow S$ denote the standard projections.

\begin{definition}[\cite{negut_rationaltrigonometricellipticalgebras}, Definition 3.19]
    \label{definition_action_Yangian}
    An action $Y_{t_1,t_2}( \widehat{\mathfrak{gl}}_1 ) \curvearrowright H_{\mathcal{M}}$ is an abelian group homomorphism $$Y_{t_1,t_2}( \widehat{\mathfrak{gl}}_1 ) \xrightarrow{\Phi} \operatorname{Hom}\left(H_{\mathcal{M}},H_{\mathcal{M} \times S}\right)$$ satisfying the following properties for all $x,y \in Y_{t_1,t_2}( \widehat{\mathfrak{gl}}_1 )$ and all $\gamma \in \mathbb{Z}[t_1,t_2]^\mathrm{sym}$.
    \begin{enumerate}[i)]
        \item \textbf{Unit:} $$\Phi(1) = \Bigg( H_{\mathcal{M}} \xrightarrow{p_{\mathcal{M}}^*} H_{\mathcal{M} \times S} \Bigg)$$
        \item $\mathbb{Z}[t_1,t_2]^\mathrm{sym}$\textbf{-linearity:} $$\Phi(\gamma x) = \Bigg( H_{\mathcal{M}} \xrightarrow{\Phi(x)} H_{\mathcal{M} \times S} \xrightarrow{\cdot p_S^*(\phi(\gamma))} H_{\mathcal{M} \times S} \Bigg)$$
        \item \textbf{Multiplicativity:} $$\Phi(\textcolor{red}{x}\textcolor{blue}{y}) = \Bigg( H_{\mathcal{M}} \xrightarrow{\Phi(\textcolor{blue}{y})} H_{\mathcal{M} \times \textcolor{blue}{S}} \xrightarrow{\Phi(\textcolor{red}{x}) \boxtimes \operatorname{id}_{\textcolor{blue}{S}}} H_{\mathcal{M} \times \textcolor{red}{S} \times \textcolor{blue}{S}} \xrightarrow{|_\Delta} H_{\mathcal{M} \times S} \Bigg)$$
        \item \textbf{Commutator:} $\Phi(x)$ and $\Phi(y)$ have diagonal commutator, and $$\bigl[ \Phi(x),\Phi(y) \bigr]_\mathrm{red} = \Phi\left( \frac{\bigl[x,y\bigr]}{t_1t_2} \right),$$ where the right-hand side is well-defined because of Remark \ref{remark_adjustments_Yangian_geometry}.
    \end{enumerate}
\end{definition}

With the previous definition in mind, we now prove our main result.

\begin{proof}[Proof of Theorem \ref{theorem_action_Yangian_stable}] 
    A direct comparison between the abstract Yangian relations \eqref{equation_relation_Yangian_e_e}--\eqref{equation_relation_Yangian_e_f} and the identities established in Theorem \ref{theorem_relations_geometric_operators_stable} yields the desired result.
\end{proof}

\begin{remark}
    Proposition \ref{proposition_commutator_cohomology_E} implies that the cubic relations analogous to the Drinfeld--Serre relations in finite types (see Remark \ref{remark_Drinfeld-Serre}) hold in the geometric setting. More explicitly, one has
    \[
    \sum_{\sigma \in S_3} \bigl[ E_{d_{\sigma(1)}}, \bigl[ E_{d_{\sigma(2)}},E_{d_{\sigma(3)}+1} \bigr] \bigr] = 0
    \]
    for all $d_1,d_2,d_3 \geq 0$.
\end{remark}

\section{The commutator \texorpdfstring{$\bigl[E(z),F(w)\bigr]$}{[E(z),F(w)]}}
\label{section_commutator_E_F}
\subsection{Geometry of the compositions \texorpdfstring{$E_aF_b$}{EaFb} and \texorpdfstring{$F_bE_a$}{FbEa}}
\label{section_commutator_geometry}
As throughout the paper, $S$ denotes a smooth projective surface over $\mathbb{C}$ satisfying Assumptions A and S with respect to a fixed ample divisor $H$ and a pair $(r,c_1) \in \mathbb{N} \times H^2(S,\mathbb{Z})$. We begin with an observation about the nested moduli space that governs the composition $\textcolor{red}{E_a}\textcolor{blue}{F_b}$.

\begin{claim}[\cite{zhao_commutator_moduli_stable}, Section 1.3]
    \label{claim_irreducibility_nested_stable_uparrow}
    The nested moduli space $$\mathfrak{Z}^\uparrow \define \Big\{ \mathcal{F}_0 \subset_{\textcolor{red}{u}} \mathcal{F}_1 \supset_{\textcolor{blue}{v}} \mathcal{F}_2 \Big\}$$ is irreducible.
\end{claim}
\begin{proof}
    Consider the morphism $\mathfrak{Z}^\uparrow \rightarrow \{ \mathcal{F}_1 \supset_{\textcolor{blue}{v}} \mathcal{F}_2\}$. The base space is irreducible, and the morphism is a proper fibration with irreducible fibers of constant dimension. It follows that the total space $\mathfrak{Z}^\uparrow$ is irreducible (see, e.g., \cite{vakil_rising_sea}).
\end{proof}

Combining Claim \ref{claim_irreducibility_nested_stable_uparrow} with the dimension estimates of Section \ref{subsection_dimension_estimates_stable}, a cohomological base change shows that the composition $\textcolor{red}{E_a}\textcolor{blue}{F_b}$ is induced by the following diagram
\begin{equation}
    \label{equation_proof_E_F_3}
    \begin{tikzcd}[row sep=small, column sep=tiny]
	& {\ell_{01}^a\ell_{12}^b} &&&&& \\
	& {\Big\{ \mathcal{F}_0 \subset_{\textcolor{red}{u}} \mathcal{F}_1 \supset_{\textcolor{blue}{v}} \mathcal{F}_2 \Big\}} &&&& {\mathcal{F}_0 \subset_{\textcolor{red}{u}} \mathcal{F}_1 \supset_{\textcolor{blue}{v}} \mathcal{F}_2} \\
	{\mathcal{M} \times \textcolor{red}{S} \times \textcolor{blue}{S}} && {\mathcal{M}} && {(\mathcal{F}_0,\textcolor{red}{u},\textcolor{blue}{v})} && {\mathcal{F}_2.}
	\arrow[dashed, no head, from=1-2, to=2-2]
	\arrow[from=2-2, to=3-1]
	\arrow[from=2-2, to=3-3]
	\arrow[maps to, from=2-6, to=3-5]
	\arrow[maps to, from=2-6, to=3-7]
\end{tikzcd}
\end{equation}

To relate this composition to the reversed product $\textcolor{blue}{F_b}\textcolor{red}{E_a}$, we introduce the quadruple moduli space $\mathfrak{Y}^\mathrm{rot}$. This scheme parametrizes commutative diagrams of Hecke modifications of the form
\begin{equation}
    \label{equation_proof_E_F_2_stable}
    \begin{tikzcd}[sep=small]
	& {\mathcal{F}_1} & \\
	{\mathcal{F}_0} && {\mathcal{F}_2.} \\
	& {\mathcal{F}_1'}
	\arrow["{\textcolor{red}{u}}", hook, from=2-1, to=1-2]
	\arrow["{\textcolor{blue}{v}}"', hook', from=2-3, to=1-2]
	\arrow["{\textcolor{blue}{v}}", hook', from=3-2, to=2-1]
	\arrow["{\textcolor{red}{u}}"', hook, from=3-2, to=2-3]
\end{tikzcd}
\end{equation}
By Propositions \ref{proposition_moduli_spaces_Y_geometric_properties} and \ref{proposition_moduli_spaces_Y_reduced}, the scheme $\mathfrak{Y}^\mathrm{rot}$ is smooth, irreducible, and reduced. Its dimension on the connected component containing the closed point in \eqref{equation_proof_E_F_2_stable} is precisely the expected dimension
\begin{equation}
    \label{equation_proof_E_F_3_stable}
    \operatorname{const} + 2rc_2(\mathcal{F}_0) + 2.
\end{equation}

\begin{claim}[\cite{negut_shuffle_surfaces}, Claim 3.8]
    \label{claim_isomorphism_nested_stable}
    The nested moduli spaces $$\mathfrak{Z}^\uparrow \qquad\text{and}\qquad \mathfrak{Z}^\downarrow \define \Big\{ \mathcal{F}_0 \supset_{\textcolor{blue}{v}} \mathcal{F}_1' \subset_{\textcolor{red}{u}} \mathcal{F}_2 \Big\}$$ are isomorphic over the open locus $\{ \mathcal{F}_0 \neq \mathcal{F}_2 \}$. 
    
    The isomorphism maps the tautological line bundles $\mathcal{L}_{0,1}$ and $\mathcal{L}_{1,2}$ to $\mathcal{L}_{1,2}'$ and $\mathcal{L}_{0,1}'$, respectively.
\end{claim}
\begin{proof}
    The isomorphism is given by the following inverse assignments:
    \[
    \big( \mathcal{F}_0 \subset_{\textcolor{red}{u}} \mathcal{F}_1 \supset_{\textcolor{blue}{v}} \mathcal{F}_2 \big) \mapsto \big( \mathcal{F}_0 \supset_{\textcolor{blue}{v}} \mathcal{F}_0 \cap \mathcal{F}_2 \subset_{\textcolor{red}{u}} \mathcal{F}_2 \big),
    \]
    where $\mathcal{F}_0 \cap \mathcal{F}_2$ lives inside the ideal sheaf $\mathcal{F}_1$, and
    \[
    \big( \mathcal{F}_0 \supset_{\textcolor{blue}{v}} \mathcal{F}_1 \subset_{\textcolor{red}{u}} \mathcal{F}_2 \big) \mapsto \big( \mathcal{F}_0 \subset_{\textcolor{red}{u}} \mathcal{F}_0 \oplus_{\mathcal{F}_1} \mathcal{F}_2 \supset_{\textcolor{blue}{v}} \mathcal{F}_2 \big).
    \]
    These assignments are best understood in terms of closed subschemes rather than ideal sheaves.
\end{proof}

Claim \ref{claim_isomorphism_nested_stable} implies that $\pi^{\uparrow}$ and $\pi^{\downarrow}$ restrict to isomorphisms over $\{ \mathcal{F}_0 \neq \mathcal{F}_2 \}$. We now analyze the geometry of $\mathfrak{Z}^\downarrow$, which governs the composition $\textcolor{blue}{F_b}\textcolor{red}{E_a}$.

\begin{claim}[\cite{zhao_commutator_moduli_stable}, Section 1.3]
    \label{claim_irreducibility_nested_stable}
    The nested moduli space $\mathfrak{Z}^\downarrow$ decomposes into the irreducible components 
    $$\mathfrak{Z}_1 \define \Big\{ \mathcal{F}_0 = \mathcal{F}_2 \Big\} \qquad \text{and} \qquad \mathfrak{Z}_2 \define \overline{\Big\{ \textcolor{red}{u} \neq \textcolor{blue}{v} \Big\}}.$$ 
    The component $\mathfrak{Z}_2$ has the same dimension as the quadruple moduli space $\mathfrak{Y}^\mathrm{rot}$ (see \eqref{equation_proof_E_F_3_stable}), while the diagonal component $\mathfrak{Z}_1$ has dimension 
    $$\operatorname{const} + 2rc_2(\mathcal{F}_0)+r+1.$$ 
\end{claim}
\begin{proof}
    The dimension formulas follow from Proposition \ref{proposition_dimension_nested_stable}.
\end{proof}

When $r=1$, the dimensions of $\mathfrak{Z}_1$ and $\mathfrak{Z}_2$ coincide, and $\mathfrak{Z}^\downarrow$ is equidimensional. For $r>1$, however, the diagonal component $\mathfrak{Z}_1$ has strictly larger dimension than $\mathfrak{Z}_2$. This failure of equidimensionality is the central geometric obstruction that necessitates a refined intersection-theoretic argument.

\subsection{An equidimensionality issue}

To overcome the obstruction identified above, we recall several standard results from intersection theory. Although traditionally formulated in the context of Chow rings, these tools adapt seamlessly to our cohomological framework. The following lemmas and theorem are taken from \cite{fulton_intersection_theory}.

\begin{lemma}[\cite{fulton_intersection_theory}, Example 1.3.1]
\label{lemma_mayer_vietoris}
    Let $Z_1$ and $Z_2$ be closed subschemes of a scheme $X$. Then the sequence
\[\begin{tikzcd}
	{A_k(Z_1 \cap Z_2)} & {A_k(Z_1) \oplus A_k(Z_2)} & {A_k(Z_1 \cup Z_2)} & 0
	\arrow[from=1-1, to=1-2]
	\arrow[from=1-2, to=1-3]
	\arrow[from=1-3, to=1-4]
\end{tikzcd}\]
    is exact for all $k \geq 0$.
\end{lemma}

We also require the excision exact sequence.

\begin{lemma}[\cite{fulton_intersection_theory}, Proposition 1.8]
\label{lemma_excision}
    Let $Z$ be a closed subscheme of a scheme $X$, and set $U \define X \setminus Z$. With the inclusions $i: Z \hookrightarrow X$ and $j: U \hookrightarrow X$, the sequence
\[\begin{tikzcd}
	{A_k(Z)} & {A_k(X)} & {A_k(U)} & 0
	\arrow["{i_*}", from=1-1, to=1-2]
	\arrow["{j^*}", from=1-2, to=1-3]
	\arrow[from=1-3, to=1-4]
\end{tikzcd}\]
    is exact for all $k \geq 0$.
\end{lemma}

Finally, we state the excess intersection formula, which will be the primary tool in the proof of Lemma \ref{lemma_base_change_stable}.

\begin{theorem}[\cite{eisenbud_harris_2016}, Theorem 13.3]
\label{theorem_excess_intersection_formula}
    Let $A$ be a subvariety of a smooth variety $X$ and $B$ a locally complete intersection subvariety of $X$. Then
    $$[A] \cdot [B] = \sum_C i_{C*} \gamma_C,$$
    where the sum is taken over all connected components $C$ of $A \cap B$, $i_C: C \hookrightarrow X$ denotes the inclusion, and 
    $$\gamma_C \define \Big\{ s(C,A) \cdot c(\mathcal{N}_{B/X}|_C) \Big\}_d \in A_d(C),$$ 
    with $d \define \dim A + \dim B - \dim X$ the expected dimension of the intersection.

    If $A$ is also a locally complete intersection, there is a symmetric form 
    $$\gamma_C = \Big\{ s(C,X) \cdot c(\mathcal{N}_{A/X}|_C) \cdot c(\mathcal{N}_{B/X}|_C) \Big\}_d.$$
\end{theorem}

\begin{remark}
    \label{remark_excess_intersection_formula}
    Two technical points regarding Theorem \ref{theorem_excess_intersection_formula} are essential for our computations. First, each connected component $C$ must be endowed with the scheme-theoretic intersection structure inherited from $A \cap B$, as this directly affects the computation of the total Segre classes $s(C,A)$ and $s(C,X)$. Second, whenever $C$ is itself a locally complete intersection in $X$, its total Segre class simplifies to $s(C,X) = s(\mathcal{N}_{C/X}) = c(-\mathcal{N}_{C/X})$.
\end{remark}

For a detailed exposition of Theorem \ref{theorem_excess_intersection_formula}, we refer to Section 13 of \cite{eisenbud_harris_2016}. With these tools in hand, we are ready to prove the following technical lemma, which encapsulates the core geometric complexity of the higher-rank setup.

\begin{lemma}
    \label{lemma_base_change_stable}
    Let $P(z)$ be the polynomial part of the Chern power series $c(\Gamma^{u*}\mathcal{U}_0,z) \in H_{\mathfrak{Z}_1}((z^{-1}))$, defined by 
    $$P(z) \define \bigl[ c(\Gamma^{u*}\mathcal{U}_0,z) \bigr]_{\geq 0} = z^r - c_1(\Gamma^{u*}\mathcal{U}_0)z^{r-1} + \ldots + (-1)^rc_r(\Gamma^{u*}\mathcal{U}_0),$$ where $\mathcal{U}_0$ is the coherent sheaf appearing in the universal flag on $\mathfrak{Z}_1 \times S$ and $\Gamma^u$ is the graph of $p_S: \mathfrak{Z}_1 \rightarrow S$.
    Consider the diagram
\[\begin{tikzcd}[row sep=small]
	& \Gamma_1 &&& \\
	& {\mathfrak{Z}_1} && {\mathfrak{Z}_2} \\
	{\Big\{ \mathcal{F}_0 \supset_{\textcolor{blue}{v}} \mathcal{F}_1' \Big\} \times \textcolor{red}{S}} &&&& {\Big\{ \mathcal{F}_1' \subset_{\textcolor{red}{u}} \mathcal{F}_2 \Big\}} \\
	&& {\mathcal{M}_{(r,c_1,c_2(\mathcal{F}_1'))} \times \textcolor{red}{S}}
	\arrow[dashed, no head, from=1-2, to=2-2]
	\arrow["{g_1}"', from=2-2, to=3-1]
	\arrow["{f_1}"'{pos=0.7}, from=2-2, to=3-5]
	\arrow["{g_2}"{pos=0.7}, from=2-4, to=3-1]
	\arrow["{f_2}", from=2-4, to=3-5]
	\arrow["{{p_+ \times \operatorname{id}_{\textcolor{red}{S}}}}"', from=3-1, to=4-3]
	\arrow["{{p_+ \times p_{\textcolor{red}{S}}}}", from=3-5, to=4-3]
\end{tikzcd}\]
    and define the cohomology class on the diagonal component $\mathfrak{Z}_1$ by $$\Gamma_1 \define (-1)^r \frac{P(\ell)-P(\ell+t)}{t},$$ where $\ell$ is the first Chern class of the tautological line bundle on $\mathfrak{Z}_1$ and $t$ is the first Chern class of the canonical bundle of $S$ (and its various pull-backs). Then the refined base change formula
    \begin{equation}
        \label{equation_base_change_stable}
        (p_+ \times \operatorname{id}_{\textcolor{red}{S}})^*(p_+ \times p_{\textcolor{red}{S}})_* = g_{1*}(\Gamma \cdot f_1^*) + g_{2*}f_2^*
    \end{equation}
    holds.
\end{lemma}
\begin{proof}
    We define the ambient product space
$$X \define \mathcal{M}_{(r,c_1,c_2(\mathcal{F}_0))} \times \mathcal{M}_{(r,c_1,c_2(\mathcal{F}'_1))} \times \mathcal{M}_{(r,c_1,c_2(\mathcal{F}_2))} \times \textcolor{red}{S} \times \textcolor{blue}{S},$$ 
and its closed subschemes 
$$A \define \Big\{ \mathcal{F}_0 \supset_{\textcolor{blue}{v}} \mathcal{F}'_1 \Big\} \times \mathcal{M}_{(r,c_1,c_2(\mathcal{F}_2))} \times \textcolor{red}{S} \qquad \text{and} \qquad B \define \mathcal{M}_{(r,c_1,c_2(\mathcal{F}_0))} \times \Big\{ \mathcal{F}'_1 \subset_{\textcolor{red}{u}} \mathcal{F}_2 \Big\} \times \textcolor{blue}{S}.$$ 
The spaces $X$, $A$, and $B$ are smooth, and we have $\mathfrak{Z}^\downarrow = A \cap B$. To establish \eqref{equation_base_change_stable}, we compute the excess intersection class
$$\Gamma \define i_{A \cap B*} \gamma_{A \cap B}$$ 
using Theorem \ref{theorem_excess_intersection_formula}. The expected dimension of $A \cap B$ is precisely the dimension of $\mathfrak{Z}_2$.

Observe that 
$\mathfrak{Z}_1 \cap \mathfrak{Z}_2 = \operatorname{im} (\pi^\downarrow|^{\mathfrak{Z}_1})$, 
where $\pi^\downarrow: \mathfrak{Y}^{\mathrm{rot}} \rightarrow \mathfrak{Z}^\downarrow$ forgets $\mathcal{F}_1$. The image of $\pi^\downarrow|^{\mathfrak{Z}_1}$ coincides with the image of 
$$\pi_+: \Big\{ \mathcal{F}_{-1} \supset_u \mathcal{F}_0 \supset_u \mathcal{F}_1' \Big\} \rightarrow \Big\{ \mathcal{F}_0 \supset_u \mathcal{F}_1' \Big\}.$$ 
By Proposition \ref{proposition_dimension_nested_stable}, $\dim (\mathfrak{Z}_1 \cap \mathfrak{Z}_2) = \dim \mathfrak{Z}_2 - 1$. Lemma \ref{lemma_mayer_vietoris} then yields the isomorphism
$$A_{\dim \mathfrak{Z}_2}(\mathfrak{Z}_1) \oplus A_{\dim \mathfrak{Z}_2}(\mathfrak{Z}_2) \cong A_{\dim \mathfrak{Z}_2}(\mathfrak{Z}^\downarrow).$$ 
Thus $\Gamma$ is determined by its restrictions $\Gamma_1$ and $\Gamma_2$ to the two components.

By Section 1.7 and Theorem 6.2 of \cite{fulton_intersection_theory}, the excess intersection formula is compatible with restriction to open subschemes of $X$. Define
$$U \define \Big\{ \textcolor{red}{u} \neq \textcolor{blue}{v} \Big\} \qquad \text{and} \qquad V \define X \setminus \mathfrak{Z}_2.$$
Since $\mathfrak{Z}_2 \cap U^\text{c}$ and $\mathfrak{Z}_1 \cap V^\text{c}$ are contained in $\mathfrak{Z}_1 \cap \mathfrak{Z}_2$, Lemma \ref{lemma_excision} implies
$$A_{\dim \mathfrak{Z}_2}(\mathfrak{Z}_2) \cong A_{\dim \mathfrak{Z}_2}(\mathfrak{Z}_2 \cap U) \qquad \text{and} \qquad A_{\dim \mathfrak{Z}_2}(\mathfrak{Z}_1) \cong A_{\dim \mathfrak{Z}_2}(\mathfrak{Z}_1 \cap V).$$

We first determine $\Gamma_2$ by replacing $X$ with $U$. We have
$$\gamma_{A \cap B}|_U = \Big\{ s(A \cap B \cap U, U) \cdot c(\mathcal{N}_{A/X}|_{A \cap B \cap U}) \cdot c(\mathcal{N}_{B/X}|_{A \cap B \cap U}) \Big\}_{\dim \mathfrak{Z}_2}.$$ 
Since $A \cap B$ is transversal along $U$, the normal bundle decomposes as 
$$\mathcal{N}_{A \cap B \cap U/U} \cong \mathcal{N}_{A/X}|_{A \cap B \cap U} \oplus \mathcal{N}_{B/X}|_{A \cap B \cap U}.$$ 
Moreover, by Remark \ref{remark_excess_intersection_formula}, one has $s(A \cap B \cap U, U) = s(\mathcal{N}_{(A \cap B \cap U)/U})$. Whitney's formula then yields 
$\gamma_{A \cap B}|_U = \bigl[ A \cap B \cap U \bigr]$, 
so $\Gamma_2$ is the fundamental class of $\mathfrak{Z}_2$.

We now determine $\Gamma_1$ by replacing $X$ with $V$. On $A \cap B \cap V$, we have $\mathcal{F}_0 = \mathcal{F}_2$ and $\textcolor{red}{u} = \textcolor{blue}{v}$. The localized excess intersection class is
$$\gamma_{A \cap B}|_V = \Big\{ s(A \cap B \cap V, A \cap V) \cdot c(\mathcal{N}_{B/X}|_{A \cap B \cap V}) \Big\}_{\dim \mathfrak{Z}_2}.$$ 
The embedding $A \cap B \cap V \hookrightarrow A \cap V$ is diagonal, its normal bundle is 
$$\mathcal{N}_{A \cap B \cap V / A \cap V} \cong \mathcal{T}_{\mathcal{M}_{c_2(\mathcal{F}_0)}} \oplus \mathcal{T}_S,$$ 
where we omit pull-backs. Remark \ref{remark_excess_intersection_formula} gives
$s(A \cap B \cap V, A \cap V) = s(\mathcal{T}_{\mathcal{M}_{c_2(\mathcal{F}_0)}} \oplus \mathcal{T}_S)$.

By Theorem II.8.17 of \cite{hartshorne_algebraic_geometry}, we have the $K$-theoretic identity
$$[\mathcal{N}_{B/X}|_{A \cap B \cap V}] = [\mathcal{T}_X|_{A \cap B \cap V}] - [\mathcal{T}_B|_{A \cap B \cap V}].$$ 
Set
\begin{equation}
    \label{equation_definition_excess_bundle}
    [\mathcal{E}] \define [\mathcal{N}_{B/X}|_{A \cap B \cap V}] - [\mathcal{N}_{A \cap B \cap V / A \cap V}],
\end{equation} 
then
$$\Gamma_1|_V = \Big\{ c(\mathcal{E}) \Big\}_{\dim \mathfrak{Z}_2}.$$

Let $p \define p_{\mathfrak{Z}_1}: \mathfrak{Z}_1 \times S \rightarrow \mathfrak{Z}_1$ be the projection. Define 
$$
\mathcal{E}xt_{p}(-,-) \define \mathbf{R}(p_* \mathcal{H}om(-,-)) \cong \mathbf{R}p_* \mathbf{R}\mathcal{H}om(-,-),
$$ 
following Proposition 2.58 of \cite{huybrechts_fourier_mukai}, and set 
$$
\mathcal{E}xt_{p}^i(-,-) \define R^i(p_* \mathcal{H}om(-,-)), \qquad \mathcal{H}om_p(-,-) \define R^0(p_* \mathcal{H}om(-,-)) \cong p_* \mathcal{H}om(-,-).
$$

By Theorem II.10.2.1 of \cite{huybrechts_lehn_moduli}, the Kodaira--Spencer map is an isomorphism, yielding
\begin{equation}
    \label{equation_tangent_space_moduli_space_stable}
    \mathcal{T}_{\mathcal{M}_{c_2(\mathcal{F}_1')}} \cong \mathcal{E}xt_p^1(\mathcal{U}_1',\mathcal{U}_1').
\end{equation}
Claim 2.11 of \cite{negut_shuffle_surfaces} gives a short exact sequence
$$
0 \rightarrow \operatorname{Ext}^1(\mathcal{F}_2,\mathcal{F}'_1) \rightarrow \operatorname{Tan}_{(\mathcal{F}'_1 \subset_u \mathcal{F}_2)} \Big\{ \mathcal{F}'_1 \subset_u \mathcal{F}_2 \Big\} \rightarrow \operatorname{Tan}_u S,
$$ 
whose relative version expresses the $K$-theory class of the tangent bundle of $B$:
\begin{equation}
    \label{equation_tangent_space_B}
    [\mathcal{T}_B] = [\mathcal{E}xt_p^1(\mathcal{U}_0,\mathcal{U}_0)] + [\mathcal{E}xt_p^1(\mathcal{U}_2,\mathcal{U}_1')] + [\mathcal{T}_{\textcolor{red}{S}}] + [\mathcal{T}_{\textcolor{blue}{S}}].
\end{equation}

Combining \eqref{equation_definition_excess_bundle}, \eqref{equation_tangent_space_moduli_space_stable}, and \eqref{equation_tangent_space_B}, we obtain
$$
[\mathcal{E}] = [\mathcal{E}xt_p^1(\mathcal{U}_1',\mathcal{U}_1')] - [\mathcal{E}xt_p^1(\mathcal{U}_0,\mathcal{U}_1')] - [\mathcal{T}_S].
$$ 
To simplify, apply $\mathcal{E}xt_p(-,\mathcal{U}_1')$ to the short exact sequence
\begin{equation}
    \label{equation_short_exact_sequence_Hecke}
    0 \rightarrow \mathcal{U}_1' \rightarrow \mathcal{U}_0 \rightarrow \mathcal{L} \otimes \mathcal{O}_{\Gamma^u} \rightarrow 0
\end{equation}
on $\mathfrak{Z}_1 \times S$. The additivity of the Grothendieck group gives
\begin{equation}
    \label{equation_RHom}
    [\mathcal{E}xt_p(\mathcal{U}_0,\mathcal{U}_1')] = [\mathcal{E}xt_p(\mathcal{U}_1',\mathcal{U}_1')] + [\mathcal{E}xt_p(\mathcal{L} \otimes \mathcal{O}_{\Gamma^u},\mathcal{U}_1')].
\end{equation}
The associated long exact sequence on $\mathfrak{Z}_1$ is
\[\begin{tikzcd}[row sep=small]
	0 & {\mathcal{H}om_p(\mathcal{L} \otimes \mathcal{O}_{\Gamma^u},\mathcal{U}_1')} & {\mathcal{H}om_p(\mathcal{U}_0,\mathcal{U}_1')} & {\mathcal{H}om_p(\mathcal{U}_1',\mathcal{U}_1')} & \\
	& {\mathcal{E}xt_p^1(\mathcal{L} \otimes \mathcal{O}_{\Gamma^u},\mathcal{U}_1')} & {\mathcal{E}xt_p^1(\mathcal{U}_0,\mathcal{U}_1')} & {\mathcal{E}xt_p^1(\mathcal{U}_1',\mathcal{U}_1')} \\
	& {\mathcal{E}xt_p^2(\mathcal{L} \otimes \mathcal{O}_{\Gamma^u},\mathcal{U}_1')} & {\mathcal{E}xt_p^2(\mathcal{U}_0,\mathcal{U}_1')} & {\mathcal{E}xt_p^2(\mathcal{U}_1',\mathcal{U}_1')} & 0.
	\arrow[from=1-1, to=1-2]
	\arrow[from=1-2, to=1-3]
	\arrow[from=1-3, to=1-4]
	\arrow[from=1-4, to=2-2]
	\arrow[from=2-2, to=2-3]
	\arrow[from=2-3, to=2-4]
	\arrow[from=2-4, to=3-2]
	\arrow[from=3-2, to=3-3]
	\arrow[from=3-3, to=3-4]
	\arrow[from=3-4, to=3-5]
\end{tikzcd}\]
The sequence terminates after three rows because $p$ has relative dimension $2$. By Theorem II.10.2.1 of \cite{huybrechts_lehn_moduli}, one obtains $\mathcal{H}om_p(\mathcal{U}_1',\mathcal{U}_1') \cong \mathcal{O}_{\mathfrak{Z}_1}$. We also have $\mathcal{H}om_p(\mathcal{U}_0,\mathcal{U}_1') = 0$.

Indeed, for an open $W \subseteq \mathfrak{Z}_1$, we get $
\mathcal{H}om_p(\mathcal{U}_0,\mathcal{U}_1')(W) = \operatorname{Hom}(\mathcal{U}_0|_{W \times S},\mathcal{U}_1'|_{W \times S})$. A non-zero map $\phi: \mathcal{U}_0|_{W \times S} \rightarrow \mathcal{U}_1'|_{W \times S}$ would, over some point, induce a non-zero map $\mathcal{F}_0 \rightarrow \mathcal{F}_1'$. Composing with $\mathcal{F}_1' \hookrightarrow \mathcal{F}_0$ gives a non-zero endomorphism of $\mathcal{F}_0$, which by simplicity (see Proposition \ref{proposition_stable_simple}) must be a scalar multiple of the identity, but this is impossible if it factors through a proper subsheaf. Since $\mathfrak{Z}_1$ is reduced and $\mathcal{U}_1'$ is flat, $\phi = 0$.

Relative Serre duality (see Theorem 3.34 of \cite{huybrechts_fourier_mukai}) yields
$$
\mathcal{E}xt_p^2(\mathcal{U}_0,\mathcal{U}_1') \cong \mathcal{E}xt_p^0(\mathcal{U}_1',\mathcal{U}_0 \otimes \omega_S)^\vee, \qquad \mathcal{E}xt_p^2(\mathcal{U}_1',\mathcal{U}_1') \cong \mathcal{E}xt_p^0(\mathcal{U}_1',\mathcal{U}_1' \otimes \omega_S)^\vee,
$$
and by the same argument as above, we have 
$
[\mathcal{E}xt_p^2(\mathcal{U}_0,\mathcal{U}_1')] = [\mathcal{E}xt_p^2(\mathcal{U}_1',\mathcal{U}_1')]
$.

Substituting these simplifications into \eqref{equation_RHom} yields
\begin{equation}
    \label{equation_computation_E}
    [\mathcal{E}] = [\mathcal{O}_{\mathfrak{Z}_1}] + [\mathcal{E}xt_p(\mathcal{L} \otimes \mathcal{O}_{\Gamma^u},\mathcal{U}_1')] - [\mathcal{T}_S].
\end{equation} 
Applying $\mathcal{E}xt_p(\mathcal{L} \otimes \mathcal{O}_{\Gamma^u},-)$ to \eqref{equation_short_exact_sequence_Hecke} gives
$$
[\mathcal{E}xt_p(\mathcal{L} \otimes \mathcal{O}_{\Gamma^u},\mathcal{U}_1')] = [\mathcal{E}xt_p(\mathcal{L} \otimes \mathcal{O}_{\Gamma^u},\mathcal{U}_0)] - [\mathcal{E}xt_p(\mathcal{O}_{\Gamma^u},\mathcal{O}_{\Gamma^u})].
$$ 
The second term is
\begin{align}
    \label{equation_Ext_Gamma}
    [\mathcal{E}xt_p(\mathcal{O}_{\Gamma^u},\mathcal{O}_{\Gamma^u})] &= [\mathbf{R}p_* \mathbf{R}\mathcal{H}om(\Gamma^u_* \mathcal{O}_{\mathfrak{Z}_1},\Gamma^u_* \mathcal{O}_{\mathfrak{Z}_1})] \nonumber \\ 
    &= [\mathbf{R}p_* \Gamma^u_* \mathbf{R}\mathcal{H}om(\mathbf{L}\Gamma^{u*}\Gamma^u_* \mathcal{O}_{\mathfrak{Z}_1},\mathcal{O}_{\mathfrak{Z}_1})] \nonumber \\ 
    &= [\mathbf{R}\mathcal{H}om(\mathbf{L}\Gamma^{u*}\Gamma^u_* \mathcal{O}_{\mathfrak{Z}_1},\mathcal{O}_{\mathfrak{Z}_1})] \nonumber \\ 
    &= [\wedge^\bullet \mathcal{N}_{\mathfrak{Z}_1 / \mathfrak{Z}_1 \times S}] \nonumber \\ 
    &= [\mathcal{O}_{\mathfrak{Z}_1}] - [\mathcal{T}_S] + [\omega_S^{-1}],
\end{align}
where we have used functorial compatibilities (see Chapter 3 of \cite{huybrechts_fourier_mukai}) and $p \circ \Gamma^u = \operatorname{id}_{\mathfrak{Z}_1}$. The fourth equality follows from the Koszul complex of the graph $\Gamma^u$. Since $\Gamma^u$ is a closed immersion, deriving $\Gamma^u_*$ is unnecessary.

The first term is
\begin{align}
    \label{equation_Ext_U_0}
    [\mathcal{E}xt_p(\mathcal{L} \otimes \mathcal{O}_{\Gamma^u},\mathcal{U}_0)] &= [\mathbf{R}p_* \mathbf{R}\mathcal{H}om(\Gamma^u_*\mathcal{L},\mathcal{U}_0)] \nonumber \\ 
    &= [\mathbf{R}p_* \Gamma^u_* \mathbf{R}\mathcal{H}om(\mathcal{L},\Gamma^{u*} \mathcal{U}_0 \otimes \omega_S^{-1}[-2])] \nonumber \\ 
    &= [\mathbf{R}\mathcal{H}om(\mathcal{L},\Gamma^{u*} \mathcal{U}_0 \otimes \omega_S^{-1}[-2])] \nonumber \\ 
    &= [\Gamma^{u*} \mathcal{U}_0 \otimes \mathcal{L}^{-1} \otimes \omega_S^{-1}],
\end{align}
where the second equality uses Grothendieck--Verdier duality (see Theorem 3.34 of \cite{huybrechts_fourier_mukai}).

Combining \eqref{equation_computation_E}, \eqref{equation_Ext_Gamma}, and \eqref{equation_Ext_U_0}, we obtain $[\mathcal{E}] = \bigl( [\Gamma^{u*} \mathcal{U}_0 \otimes \mathcal{L}^{-1}] - 1 \bigr) [\omega_S^{-1}]$. Define
\begin{align*}
    Q(z) &\define z^r + c_1\left(\Gamma^{u*} \mathcal{U}_0 \otimes \mathcal{L}^{-1} \otimes \omega_S^{-1}\right) z^{r-1} + \ldots + c_r\left(\Gamma^{u*} \mathcal{U}_0 \otimes \mathcal{L}^{-1} \otimes \omega_S^{-1} \right) \\
    &= \prod_{i=1}^r(z+u_i-\ell-t).
\end{align*}
The top Chern class of the excess bundle is then
\begin{align*}
    \Big\{ c(\mathcal{E}) \Big\}_{\dim \mathfrak{Z}_2} &= c_{r-1}(\mathcal{E}) \\
    &= t^{r-1} + c_1\left(\Gamma^{u*} \mathcal{U}_0 \otimes \mathcal{L}^{-1} \otimes \omega_S^{-1}\right) t^{r-2} + \ldots + c_{r-1}\left(\Gamma^{u*} \mathcal{U}_0 \otimes \mathcal{L}^{-1} \otimes \omega_S^{-1} \right) \\
    &= \frac{Q(t) - c_r\left(\Gamma^{u*} \mathcal{U}_0 \otimes \mathcal{L}^{-1} \otimes \omega_S^{-1} \right)}{t} \\
    &= \frac{\prod_{i=1}^r(u_i-\ell) - \prod_{i=1}^r(u_i-\ell-t)}{t} \\
    &= (-1)^r \frac{P(\ell)-P(\ell+t)}{t}.
\end{align*}
This completes the proof of the lemma.
\end{proof}

Having established Lemma \ref{lemma_base_change_stable}, we may now decompose the commutator $\bigl[ \textcolor{red}{E_a},\textcolor{blue}{F_b} \bigr]$. Invoking the birationality of $\pi^\uparrow$ and $\pi^\downarrow$ (corestricted as in Section \ref{section_commutator_geometry}), we obtain a principal part
\begin{equation}
    \label{equation_proof_E_F_4_stable}
    \begin{tikzcd}[row sep=small]
	& {\ell_{01}^a \ell_{12}^b-\ell_{01}'^b \ell_{12}'^a} & \\
	& {\mathfrak{Y}^\mathrm{rot}} \\
	{\mathcal{M} \times \textcolor{red}{S} \times \textcolor{blue}{S}} && {\mathcal{M}}
	\arrow[dashed, no head, from=1-2, to=2-2]
	\arrow[from=2-2, to=3-1]
	\arrow[from=2-2, to=3-3]
\end{tikzcd}
\end{equation}
together with a diagonal contribution
\begin{equation}
    \label{equation_proof_E_F_5_stable}
    \begin{tikzcd}[row sep=small]
	&& {-\ell^{a+b} \cdot \Gamma_1} & \\
	&& {\Big\{ \mathcal{F}_1' \subset_u \mathcal{F}_0 \Big\}} \\
	{\mathcal{M} \times \textcolor{red}{S} \times \textcolor{blue}{S}} & {\mathcal{M} \times S} && {\mathcal{M}.}
	\arrow[dashed, no head, from=1-3, to=2-3]
	\arrow["{p_- \times p_S}"', from=2-3, to=3-2]
	\arrow["{p_-}", from=2-3, to=3-4]
	\arrow["{\operatorname{id}_{\mathcal{M}} \times \Delta}"', from=3-2, to=3-1]
\end{tikzcd}
\end{equation}
The left-hand morphism in \eqref{equation_proof_E_F_4_stable} sends a quadruple diagram of the form \eqref{equation_proof_E_F_2_stable} to $(\mathcal{F}_0,\textcolor{red}{u},\textcolor{blue}{v})$, while the right-hand map projects onto $\mathcal{F}_2$. The diagonal contribution \eqref{equation_proof_E_F_5_stable} arises by identifying the irreducible component $\mathfrak{Z}_1$ with the nested moduli space $\{ \mathcal{F}_1' \subset_u \mathcal{F}_0 \}$; on this component the tautological bundles $\mathcal{L}_{0,1}'$ and $\mathcal{L}_{1,2}'$ coincide, and we denote their shared first Chern class by $\ell$.

Proceeding as in the proof of Proposition \ref{proposition_commutator_cohomology_E}, we work on $\mathfrak{Y}^\mathrm{rot}$ and compare the classes $\ell_{01}^a \ell_{12}^b$ and $\ell_{01}'^b \ell_{12}'^a$. Let $j$ be the inclusion of the vanishing locus
$$
Z \left( \mathcal{L}_{0,1}' \rightarrow \mathcal{L}_{1,2} \right) = Z \left( \mathcal{L}_{1,2}' \rightarrow \mathcal{L}_{0,1} \right) = \Big\{ \mathcal{J}_0 = \mathcal{J}_2, \textcolor{red}{u} = \textcolor{blue}{v} \Big\},
$$
as in Proposition \ref{proposition_vanishing_locus_line_bundle_map}. The restrictions of $\mathcal{L}_{0,1}$ and $\mathcal{L}_{0,1}'$ to the vanishing locus coincide with those of $\mathcal{L}_{1,2}$ and $\mathcal{L}_{1,2}'$, respectively.

The following telescoping computation systematically replaces $\ell_{01}$ by $\ell_{12}'$ and $\ell_{12}$ by $\ell_{01}'$:
\begin{align*}
    \ell_{01}^a \ell_{12}^b-\ell_{01}'^b \ell_{12}'^a = \; &\ell_{01}^a\ell_{12}^b - \ell_{01}^{a-1}\ell_{12}^b\ell_{12}' + \ldots + \ell_{01}\ell_{12}^b\ell_{12}'^{a-1} - \ell_{12}^b\ell_{12}'^a \\
    = \; &j_* \left( \ell_{01}^{a+b-1} + \ldots + \ell_{01}^b\ell_{01}'^{a-1} + \ell_{01}^{b-1}\ell_{01}'^a + \ldots + \ell_{01}'^{a+b-1} \right).
\end{align*}
Consequently, the principal contribution is realized by
\[\begin{tikzcd}[row sep=small]
	&& {\sum_{k=0}^{a+b-1} \ell_{01}^k\ell_{01}'^{a+b-1-k}} & \\
	&& {\mathfrak{Z}_2^\bullet} \\
	{\mathcal{M} \times \textcolor{red}{S} \times \textcolor{blue}{S}} & {\mathcal{M} \times S} && {\mathcal{M}}
	\arrow[dashed, no head, from=1-3, to=2-3]
	\arrow["{(p_+ \times p_S) \circ \pi_-}"', from=2-3, to=3-2]
	\arrow["{p_+ \circ \pi_-}", from=2-3, to=3-4]
	\arrow["{\operatorname{id}_{\mathcal{M}} \times \Delta}"', from=3-2, to=3-1]
\end{tikzcd}\]
which gives the correspondence
\[\begin{tikzcd}[column sep=tiny,row sep=small]
	& {\sum_{k=0}^{a+b-1} (\Delta_{\mathcal{M}} \times \Delta_{S})_* (p_+ \times p_S)_* \pi_{-*}\left( \ell_{01}^k\ell_{01}'^{a+b-1-k} \right)} & \\
	& {\mathcal{M} \times \textcolor{red}{S} \times \textcolor{blue}{S} \times \mathcal{M}} \\
	{\mathcal{M} \times \textcolor{red}{S} \times \textcolor{blue}{S}} && {\mathcal{M}.}
	\arrow[dashed, no head, from=1-2, to=2-2]
	\arrow["{p_{\mathcal{M}} \times p_{\textcolor{red}{S}} \times p_{\textcolor{blue}{S}}}"', from=2-2, to=3-1]
	\arrow["{p_{\mathcal{M}}}", from=2-2, to=3-3]
\end{tikzcd}\]
The diagonal contribution \eqref{equation_proof_E_F_5_stable} is realized by
\[\begin{tikzcd}[row sep=small]
	& {-(\Delta_{\mathcal{M}} \times \Delta_{S})_* (p_- \times p_S)_* \big(\ell^{a+b} \cdot \Gamma_1 \big)} & \\
	& {\mathcal{M} \times \textcolor{red}{S} \times \textcolor{blue}{S} \times \mathcal{M}} \\
	{\mathcal{M} \times \textcolor{red}{S} \times \textcolor{blue}{S}} && {\mathcal{M}.}
	\arrow[dashed, no head, from=1-2, to=2-2]
	\arrow["{p_{\mathcal{M}} \times p_{\textcolor{red}{S}} \times p_{\textcolor{blue}{S}}}"', from=2-2, to=3-1]
	\arrow["{p_{\mathcal{M}}}", from=2-2, to=3-3]
\end{tikzcd}\]

\subsection{Simplification via projective bundle techniques}

To obtain the final expression for the commutator $\bigl[ \textcolor{red}{E_a},\textcolor{blue}{F_b} \bigr]$, we now simplify the two contributions obtained in the previous section, namely
\begin{equation}
    \label{equation_proof_E_F_6_stable}
    \sum_{k=0}^{a+b-1} (\Delta_{\mathcal{M}} \times \Delta_{S})_* (p_+ \times p_S)_* \pi_{-*}\left( \ell_{01}^k\ell_{01}'^{a+b-1-k} \right)
\end{equation}
and
\begin{equation}
    \label{equation_proof_E_F_7_stable}
    -(\Delta_{\mathcal{M}} \times \Delta_{S})_* (p_- \times p_S)_* \big(\ell^{a+b} \cdot \Gamma_1 \big).
\end{equation}
For this purpose, we require two standard facts from intersection theory.

\begin{lemma}[\cite{eisenbud_harris_2016}, Proposition 5.17]
    \label{lemma_Chern_classes_vector_bundle_line_bundle}
    Let $\mathcal{E}$ be a locally free sheaf of rank $e$ and $\mathcal{L}$ a line bundle. Then,
    \begin{align*}
        c_k(\mathcal{E} \otimes \mathcal{L}) &= \sum_{l=0}^k \binom{e-l}{k-l} c_1(\mathcal{L})^{k-l} \cdot c_l(\mathcal{E}) \\
        &= \sum_{i=0}^k \binom{e-k+i}{i} c_1(\mathcal{L})^i \cdot c_{k-i}(\mathcal{E}).
    \end{align*}
\end{lemma}

\begin{lemma}[\cite{fulton_intersection_theory}]
    \label{lemma_fundamental_class_zero_locus}
    Let $X$ be a Cohen--Macaulay scheme and $\mathcal{E}$ a locally free sheaf of rank $e$ on $X$. Suppose that the cosection $\sigma^\vee: \mathcal{E} \rightarrow \mathcal{O}_X$ is regular in the sense that $\operatorname{codim} Z(\sigma) = e$, where $$Z(\sigma) \xhookrightarrow{i} X$$ denotes the zero locus of $\sigma$. Then, $$i_* \bigl[Z(\sigma)\bigr] = c_{\text{top}} (\mathcal{E}^\vee).$$
\end{lemma}

Recall from Section \ref{subsection_nested_stable} that the universal sheaf $\mathcal{U}$ on $\mathcal{M} \times S$ has homological dimension $1$. More precisely, there exists a short exact sequence
\begin{equation}
    \label{equation_proof_E_F_8_stable}
    0 \rightarrow \mathcal{W} \rightarrow \mathcal{V} \rightarrow \mathcal{U} \rightarrow 0,
\end{equation} 
where $\mathcal{W}$ and $\mathcal{V}$ are locally free on $\mathcal{M} \times S$ of ranks $w$ and $v$, respectively. Since the rank of $\mathcal{U}$ equals the rank $r$ of the stable sheaves parametrized by $\mathcal{M}$, we have $r = v - w$.

The universal sheaf lifts to nested moduli spaces such as $\mathfrak{Z}_2^\bullet \times S$, inducing inclusions
\begin{equation}
    \label{equation_proof_E_F_9_stable}
    \mathcal{U}_1' \subset \mathcal{U}_0 \subset \mathcal{U}_1.
\end{equation} 
Over a closed point of the nested moduli space, the fibers of these sheaves recover exactly the corresponding chain of stable sheaves. Each universal sheaf in \eqref{equation_proof_E_F_9_stable} admits a resolution analogous to \eqref{equation_proof_E_F_8_stable}.

We first simplify expression \eqref{equation_proof_E_F_6_stable}. Proposition \ref{proposition_projective_bundle_description} gives the factorization of the morphism 
$\pi_-$ 
as
\[\begin{tikzcd}[row sep=small]
	{\Big\{\mathcal{F}_1' \subset_u \mathcal{F}_0 \subset_u \mathcal{F}_1 \Big\}} & {\mathbb{P}_{\{\mathcal{F}_0 \subset_u \mathcal{F}_1 \}}(\Gamma^{u*}\mathcal{V}_0)} \\
	& {\Big\{\mathcal{F}_0 \subset_u \mathcal{F}_1 \Big\},}
	\arrow["i", hook, from=1-1, to=1-2]
	\arrow["{\pi_-}"', from=1-1, to=2-2]
	\arrow["\rho", from=1-2, to=2-2]
\end{tikzcd}\]
where the closed embedding $i$ is cut out by the cosection 
$\sigma^\vee: \rho^*\Gamma^{u*} \mathcal{W}_0 \otimes \mathcal{O}(-1) \rightarrow \mathcal{O}_{\mathbb{P}(\Gamma^{u*}\mathcal{V}_0)}.$
This cosection is not regular because the codimension of $\mathfrak{Z}^\bullet_2$ in $\mathbb{P}(\Gamma^{u*}\mathcal{V}_0)$ is $w_0-1$, whereas the rank of $\rho^*\Gamma^{u*} \mathcal{W}_0 \otimes \mathcal{O}(-1)$ is $w_0$. We therefore consider the modified cosection
$$\widetilde{\sigma}^\vee: \frac{\rho^*\Gamma^{u*}\mathcal{W}_0}{\rho^*\mathcal{L}_{0,1} \otimes \rho^*\omega_S} \otimes \mathcal{O}(-1) \rightarrow \mathcal{O}_{\mathbb{P}(\Gamma^{u*}\mathcal{V}_0)},$$
where $\omega_S$ denotes the canonical bundle of $S$ and its various pull-backs. By Proposition 2.27 of \cite{negut_hecke_correspondences}, the sheaf
$$\frac{\rho^*\Gamma^{u*}\mathcal{W}_0}{\rho^*\mathcal{L}_{0,1} \otimes \rho^*\omega_S}$$
is locally free of rank $w_0-1$. Applying Lemma \ref{lemma_Chern_classes_vector_bundle_line_bundle} and Lemma \ref{lemma_fundamental_class_zero_locus}, we obtain
$$i_*1 = \sum_{i=0}^{w_0-1} \sum_{j=0}^i c_1\left( \mathcal{O}(1) \right)^{w_0-1-i} \cdot c_1 \left( \rho^*\mathcal{L}_{0,1} \otimes \rho^*\omega_S \right)^j \cdot (-1)^{i-j} c_{i-j}\left( \rho^*\Gamma^{u*}\mathcal{W}_0 \right).$$

Consequently,
\begin{align*}
    &\pi_{-*}\left( \ell_{01}'^{\,a+b-1-k} \right) \\
    =\; &\sum_{i=0}^{w_0} \sum_{j=0}^i (-1)^{a+b-k-i-r-1}c_{a+b-k-i-r-1}\left( -\Gamma^{u*}\mathcal{V}_0 \right) \cdot c_1 \left( \mathcal{L}_{0,1} \otimes \omega_S \right)^j \cdot (-1)^{i-j} c_{i-j}\left( \Gamma^{u*}\mathcal{W}_0 \right),
\end{align*}
where the upper limit has been extended from $w_0-1$ to $w_0$ because the $w_0$-th Chern class of a rank $w_0-1$ vector bundle vanishes.

Using this identity, the principal contribution simplifies to
\begin{align*}
    &\sum_{k=0}^{a+b-1} (\Delta_{\mathcal{M}} \times \Delta_{S})_* (p_+ \times p_S)_* \pi_{-*}\left( \ell_{01}^k\ell_{01}'^{a+b-1-k} \right) \\
    =\; &\sum_{k=0}^{a+b-1} \sum_{i=0}^{w} \sum_{j=0}^i \sum_{h=0}^j \binom{j}{h} (\Delta_{\mathcal{M}} \times \Delta_{S})_* \Big( (p_+ \times p_S)_* \left( \ell^{k+h} \right) \\
    &\hspace{17em}\cdot (-1)^{a+b-k-i-r-1}c_{a+b-k-i-r-1}\left( -\mathcal{V} \right) \\
    &\hspace{17em} \cdot c_1 \left( \omega_S \right)^{j-h} \cdot (-1)^{i-j} c_{i-j}\left( \mathcal{W} \right) \Big).
\end{align*}

It remains to compute $(p_+ \times p_S)_* \left( \ell^{k+h} \right)$. Proposition \ref{proposition_projective_bundle_description} gives the factorization of $p_+$ as
\[\begin{tikzcd}[row sep=small]
	{\Big\{\mathcal{F}_0 \subset_u \mathcal{F}_1 \Big\}} & {\mathbb{P}_{\mathcal{M} \times S}(\mathcal{W}^\vee \otimes \omega_S)} \\
	& {\mathcal{M} \times S,}
	\arrow["{i'}", hook, from=1-1, to=1-2]
	\arrow["{p_+ \times p_S}"', from=1-1, to=2-2]
	\arrow["{\rho'}", from=1-2, to=2-2]
\end{tikzcd}\]
where $i'$ is cut out by
$\sigma'^\vee: \rho'^*\mathcal{V}^\vee \otimes \omega_S \otimes \mathcal{O}(-1) \rightarrow \mathcal{O}_{\mathbb{P}( \mathcal{W}^\vee \otimes \omega_S )}.$
This cosection is regular: the codimension of $\{ \mathcal{F}_0 \subset_u \mathcal{F}_1 \}$ inside $\mathbb{P}( \mathcal{W}^\vee \otimes \omega_S )$ is exactly $w+r$, which matches the rank of $\rho'^*\mathcal{V}^\vee \otimes \omega_S \otimes \mathcal{O}(-1)$. Lemma \ref{lemma_Chern_classes_vector_bundle_line_bundle} and Lemma \ref{lemma_fundamental_class_zero_locus} then yield
$$i'_*1 = \sum_{l=0}^{v} c_1 \left( \mathcal{O}(1) \right)^{v-l} \cdot c_l \left( \rho'^*\mathcal{V} \otimes \omega_S^{-1} \right).$$

A direct computation gives
\begin{align*}
    &(p_+ \times p_S)_* \left( \ell^{k+h} \right) \\
    =\; &(-1)^{r-1} \sum_{l=0}^{v} (-1)^{k+h+v-w+1-l} c_{k+h+v-w+1-l} \left( -\mathcal{W} \otimes \omega_S^{-1} \right) \cdot (-1)^l c_l \left( \mathcal{V} \otimes \omega_S^{-1} \right).
\end{align*}

Substituting this into the previous expression, we find that \eqref{equation_proof_E_F_6_stable} is equivalent to
\begin{align}
\begin{aligned}
    \label{equation_proof_E_F_10_stable}
    &\sum_{k=0}^{a+b-1} \sum_{i=0}^{w} \sum_{j=0}^i \sum_{h=0}^j \sum_{l=0}^{v} \binom{j}{h} (\Delta_{\mathcal{M}} \times \Delta_{S})_* \Bigg( (-1)^l c_l \bigl( \mathcal{V} \otimes \omega_S^{-1} \bigr) \\
    &\hspace{10em} \cdot (-1)^{k+h+r+1-l} c_{k+h+r+1-l} \bigl( -\mathcal{W} \otimes \omega_S^{-1} \bigr) \cdot c_1(\omega_S)^{j-h} \\
    &\hspace{10em} \cdot (-1)^{a+b-k-i-r-1} c_{a+b-k-i-r-1}\bigl( -\mathcal{V} \bigr) \cdot (-1)^{i-j} c_{i-j}\bigl( \mathcal{W} \bigr) \Bigg).
\end{aligned}
\end{align}

We now turn to the diagonal contribution \eqref{equation_proof_E_F_7_stable}. By Proposition \ref{proposition_projective_bundle_description}, the morphism $p_- \times p_S$ factorizes as
\[\begin{tikzcd}[row sep=small]
	{\Big\{\mathcal{F}_1' \subset_u \mathcal{F}_0 \Big\}} & {\mathbb{P}_{\mathcal{M} \times S}(\mathcal{V})} \\
	& {\mathcal{M} \times S,}
	\arrow["{i''}", hook, from=1-1, to=1-2]
	\arrow["{{p_- \times p_S}}"', from=1-1, to=2-2]
	\arrow["{\rho''}", from=1-2, to=2-2]
\end{tikzcd}\]
with $i''$ cut out by $\sigma''^\vee: \rho''^*\mathcal{W} \otimes \mathcal{O}(-1) \rightarrow \mathcal{O}_{\mathbb{P}(\mathcal{V})}$.
This cosection is regular, since the codimension of $\{ \mathcal{F}_1' \subset_u \mathcal{F}_0 \}$ in $\mathbb{P}(\mathcal{V})$ is $v-r$, which equals the rank of $\rho''^*\mathcal{W} \otimes \mathcal{O}(-1)$. Lemma \ref{lemma_Chern_classes_vector_bundle_line_bundle} and Lemma \ref{lemma_fundamental_class_zero_locus} therefore give
$$i''_*1 = \sum_{i=0}^w c_1 \left( \mathcal{O}(1) \right)^{w-i} \cdot c_i \left( \rho''^*\mathcal{W}^\vee \right).$$

The cohomology class $\Gamma_1$ is a polynomial of degree $r-1$ in $\ell$ whose coefficients are Chern classes of $\Gamma^{u*}\mathcal{U}_0$. Hence we may write
$$\frac{P(\ell)-P(\ell+t)}{t} = \sum_{j=0}^{r-1} (p_- \times p_S)^*\gamma_j \cdot \ell^j,$$ 
where $\gamma_j \in H_{\mathcal{M} \times S}$. We then compute
\begin{align}
    \label{equation_proof_E_F_20_stable}
    &(p_- \times p_S)_* \left( -\ell^{a+b} \cdot \Gamma_1 \right) \\ \nonumber
    =\; &(-1)^{r-1} (p_- \times p_S)_* \left( \ell^{a+b} \cdot \frac{P(\ell)-P(\ell+t)}{t} \right) \\ \nonumber
    =\; &(-1)^{r-1} \sum_{j=0}^{r-1} \sum_{i=0}^w  \gamma_j \cdot (-1)^{a+b-i}c_{a+b-i} \left( -\mathcal{V} \right) \cdot (-1)^i c_i \left( \mathcal{W} \right).
\end{align}

\subsection{Residue calculus and the final expression}

We now convert the multiple sums obtained above into compact generating functions using formal residue calculus. Following \cite{negut_shuffle_surfaces}, for a formal Laurent series expanded around infinity, say
\[
\sum_{i=n}^\infty a_i z^{-i},
\]
we denote by \(\underset{z=\infty}{\operatorname{Res}}\) the coefficient of \(z^{-1}\):
\[
\underset{z=\infty}{\operatorname{Res}} \Bigg[ \sum_{i=n}^\infty \frac{a_i}{z^i} \Bigg] \define a_1.
\]
This formalism provides a systematic way to extract coefficients from rational series.

We recall the expansions
\begin{equation}
    c(\mathcal{U},w+t) = \Bigg( \sum_{i=0}^\infty (-1)^i w^{v-i} c_i \bigl( \mathcal{V} \otimes \omega_S^{-1} \bigr) \Bigg) \cdot
       \Bigg( \sum_{j=0}^\infty (-1)^j w^{-w-j} c_j \bigl( -\mathcal{W} \otimes \omega_S^{-1} \bigr) \Bigg) \label{equation_proof_E_F_14}
\end{equation}
and
\begin{equation}
    c(-\mathcal{U},z) = \Bigg( \sum_{i=0}^\infty (-1)^i z^{-v-i} c_i \bigl( -\mathcal{V} \bigr) \Bigg) \cdot
       \Bigg( \sum_{j=0}^\infty (-1)^j z^{w-j} c_j \bigl( \mathcal{W} \bigr) \Bigg). \label{equation_proof_E_F_15}
\end{equation}

Comparing \eqref{equation_proof_E_F_14} and \eqref{equation_proof_E_F_15} with \eqref{equation_proof_E_F_10_stable}, we see that the latter is equivalent to
\[
(-1)^{r-1} (\Delta_{\mathcal{M}} \times \Delta_S)_* \left( \underset{w=\infty}{\operatorname{Res}} \; \underset{z=\infty}{\operatorname{Res}} \left[ \frac{c(\mathcal{U},w+t) \cdot c(-\mathcal{U},z) \cdot (z^{a+b}-w^{a+b})}{(z-w)(z-w-t)} \right] \right).
\]
Similarly, \eqref{equation_proof_E_F_20_stable} becomes
\[
(-1)^{r-1} \underset{z=\infty}{\operatorname{Res}} \left[ z^{a+b} \cdot c(-\mathcal{U},z) \cdot \frac{P'(z)-P'(z+t)}{t} \right],
\]
where \(P'(z) \define \bigl[ c(\mathcal{U},z) \bigr]_{\geq 0}\) is the polynomial part of the Chern power series.

Combining the principal and diagonal contributions, the commutator \(\bigl[ \textcolor{red}{E_a},\textcolor{blue}{F_b} \bigr]\) is realized by the correspondence
\[\begin{tikzcd}[row sep=small]
	& {\text{cohomology class}} & \\
	& {\mathcal{M}\times \textcolor{red}{S} \times \textcolor{blue}{S} \times \mathcal{M}} \\
	{\mathcal{M} \times \textcolor{red}{S} \times \textcolor{blue}{S}} && {\mathcal{M},}
	\arrow[dashed, no head, from=1-2, to=2-2]
	\arrow["{{p_{\mathcal{M}} \times p_{\textcolor{red}{S}} \times p_{\textcolor{blue}{S}}}}"', from=2-2, to=3-1]
	\arrow["{{p_{\mathcal{M}}}}", from=2-2, to=3-3]
\end{tikzcd}\]
where the cohomology class on \(\mathcal{M} \times \textcolor{red}{S} \times \textcolor{blue}{S} \times \mathcal{M}\) is
\begin{align}
    \begin{aligned}
    \label{equation_proof_E_F_11_stable}
    (-1)^{r-1} (\Delta_{\mathcal{M}} \times \Delta_S)_* \Bigg( &\underset{w=\infty}{\operatorname{Res}} \; \underset{z=\infty}{\operatorname{Res}} \left[ \frac{c(\mathcal{U},w+t) \cdot c(-\mathcal{U},z) \cdot (z^{a+b}-w^{a+b})}{(z-w)(z-w-t)} \right] \\
    + &\underset{z=\infty}{\operatorname{Res}} \left[ z^{a+b} \cdot c(-\mathcal{U},z) \cdot \frac{P'(z)-P'(z+t)}{t} \right] \Bigg).
    \end{aligned}
\end{align}

To complete the proof of Proposition \ref{proposition_commutator_cohomology_E_F_operators_stable}, we perform a final algebraic simplification.

\begin{claim}
    \label{claim_exchange_order_residues_stable}
    The expression \eqref{equation_proof_E_F_11_stable} simplifies to
    \[
    (-1)^r (\Delta_{\mathcal{M}} \times \Delta_S)_* \left( \underset{z=\infty}{\operatorname{Res}} \left[\frac{1}{t} \cdot \frac{c(\mathcal{U},z+t)}{c(\mathcal{U},z)} \cdot z^{a+b} \right] \right).
    \]
\end{claim}

\begin{proof}
    We proceed in two steps. First, observe that 
    \[
    \underset{w=\infty}{\operatorname{Res}} \; \underset{z=\infty}{\operatorname{Res}} \left[ \frac{c(\mathcal{U},w+t) \cdot c(-\mathcal{U},z) \cdot \bigl( z^{a+b} - w^{a+b} \bigr)}{(z-w)(z-w-t)} \right] = \underset{w=\infty}{\operatorname{Res}} \; \underset{z=\infty}{\operatorname{Res}} \left[ \frac{c(\mathcal{U},w+t) \cdot c(-\mathcal{U},z) \cdot z^{a+b}}{(z-w)(z-w-t)} \right].
    \]
    This equality holds because expanding the denominator $(z-w)(z-w-t)$ reveals that multiplying by negative powers of $z$ cannot yield the strictly positive powers of $z$ required to produce a non-zero residue at infinity.

    Second, exchanging the order in which we compute the residues introduces an error term. Concretely, we have
\begin{align}
    \label{equation_proof_claim_residues_1_stable}
    \begin{aligned}
        &\underset{w=\infty}{\operatorname{Res}} \; \underset{z=\infty}{\operatorname{Res}} \left[ \frac{c(\mathcal{U},w+t) \cdot c(-\mathcal{U},z) \cdot z^{a+b}}{(z-w)(z-w-t)} \right] \\
        = \; &\underset{z=\infty}{\operatorname{Res}} \; \underset{w=\infty}{\operatorname{Res}} \left[ \frac{c(\mathcal{U},w+t) \cdot c(-\mathcal{U},z) \cdot z^{a+b}}{(z-w)(z-w-t)} \right] + \text{error term}.
    \end{aligned}
\end{align}

    Now, consider the iterated residue
\begin{align}
    \label{equation_proof_claim_residues_2_stable}
    \begin{aligned}
        &\underset{z=\infty}{\operatorname{Res}} \; \underset{w=\infty}{\operatorname{Res}} \left[ \frac{c(\mathcal{U},w+t) \cdot c(-\mathcal{U},z) \cdot z^{a+b}}{(z-w)(z-w-t)} \right] \\
        = \; &\underset{z=\infty}{\operatorname{Res}} \left[ c(-\mathcal{U},z) \cdot z^{a+b} \cdot \underset{w=\infty}{\operatorname{Res}} \left[ \frac{c(\mathcal{U},w+t)}{(z-w)(z-w-t)} \right] \right].
    \end{aligned}
\end{align}
    We evaluate the inner residue by noting that the power series $c(\mathcal{U},w+t)$ can be replaced by its polynomial part $P'(w+t)$. To compute the residue of the rational function
    $$\frac{P'(w+t)}{(z-w)(z-w-t)},$$
    we invoke the global residue theorem. The function has poles at $w=z$, $w=z-t$ and $w=\infty$. Consequently, one obtains
    \begin{align*}
        &\underset{w=\infty}{\operatorname{Res}} \left[ \frac{P'(w+t)}{(z-w)(z-w-t)} \right] \\
        =\; &\underset{w=z}{\operatorname{Res}} \left[ \frac{P'(w+t)}{(w-z)(w-(z-t))} \right] + \underset{w=z-t}{\operatorname{Res}} \left[ \frac{P'(w+t)}{(w-z)(w-(z-t))} \right] \\
        =\; &\frac{P'(z+t)}{t} - \frac{P'(z)}{t}.
    \end{align*}
    Substituting this back into \eqref{equation_proof_claim_residues_2_stable} simplifies the iterated residue to 
    $$\underset{z=\infty}{\operatorname{Res}} \left[ c(-\mathcal{U},z) \cdot z^{a+b} \cdot \frac{P'(z+t)-P'(z)}{t} \right],$$ 
    which exactly cancels the contribution in \eqref{equation_proof_E_F_11_stable} arising from the diagonal component $\mathfrak{Z}_1$. 

    Finally, by standard results in formal residue calculus, the error term in \eqref{equation_proof_claim_residues_1_stable} is given by the sum of the residues at the finite poles $w = z$ and $w = z-t$:
\begin{align*}
&\underset{z=\infty}{\operatorname{Res}} \; \underset{w=z}{\operatorname{Res}} \left[ \frac{c(\mathcal{U},w+t) \cdot c(-\mathcal{U},z) \cdot z^{a+b}}{(z-w)(z-w-t)} \right] + \underset{z=\infty}{\operatorname{Res}} \; \underset{w=z-t}{\operatorname{Res}} \left[ \frac{c(\mathcal{U},w+t) \cdot c(-\mathcal{U},z) \cdot z^{a+b}}{(z-w)(z-w-t)} \right] \\
=\, &\underset{z=\infty}{\operatorname{Res}} \left[ \frac{c(\mathcal{U},z+t) \cdot c(-\mathcal{U},z) \cdot z^{a+b}}{-t} \right] + \underset{z=\infty}{\operatorname{Res}} \left[ \frac{c(\mathcal{U},z) \cdot c(-\mathcal{U},z) \cdot z^{a+b}}{t} \right] \\
=\, &\underset{z=\infty}{\operatorname{Res}} \left[ \frac{c(\mathcal{U},z+t) \cdot c(-\mathcal{U},z) \cdot z^{a+b}}{-t} \right],
\end{align*}
where we used the identity $c(\mathcal{U},z) \cdot c(-\mathcal{U},z) = 1$, which causes the second term to vanish.
    Multiplying by $(-1)^{r-1}$ from \eqref{equation_proof_E_F_11_stable} absorbs the minus sign, yielding the final factor of $(-1)^r$.    
\end{proof}

The algebraic simplification established in Claim \ref{claim_exchange_order_residues_stable} concludes the proof of Proposition \ref{proposition_commutator_cohomology_E_F_operators_stable}. This completes the intersection-theoretic construction of the Yangian action on the cohomology of the moduli space of stable sheaves.

\begin{remark}
    \label{remark_excision_approach}
    There is an alternative approach to proving Proposition \ref{proposition_commutator_cohomology_E_F_operators_stable} that relies on excision in cohomology. As a consequence of the isomorphism established in Claim \ref{claim_isomorphism_nested_stable}, the restriction of the commutator \(\bigl[ \textcolor{red}{E_a},\textcolor{blue}{F_b} \bigr]\) to the open locus \(\{\textcolor{red}{u} \neq \textcolor{blue}{v}\}\) vanishes. By excision, this implies that the commutator is supported on the diagonal. Hence there exists a cohomology class \(\Xi \in H_{\mathcal{M} \times S}\) such that \(\bigl[ \textcolor{red}{E_a},\textcolor{blue}{F_b} \bigr]\) is given by the correspondence $(\Delta_{\mathcal{M}} \times \Delta_S)_* \Xi$ on $\mathcal{M} \times \textcolor{red}{S} \times \textcolor{blue}{S} \times \mathcal{M}$, where $\Delta_{\mathcal{M}} \times \Delta_S: \mathcal{M} \times S \hookrightarrow \mathcal{M} \times \textcolor{red}{S} \times \textcolor{blue}{S} \times \mathcal{M}$ is the product of the diagonal embeddings. Since the pushforward \((\Delta_{\mathcal{M}} \times \Delta_S)_*\) is injective, \(\Xi\) is uniquely determined, and it suffices to evaluate the commutator on a single test class that detects it. 

    The excision approach is sketched in \cite{negut_rationaltrigonometricellipticalgebras}, but, to the best of our knowledge, it has not been fully carried out to identify the explicit class \(\Xi\).
\end{remark}

\bibliographystyle{alpha}
\bibliography{biblio}

\end{document}